\documentclass[12pt]{amsart}
\usepackage[margin=1in]{geometry}
\usepackage{ulem} 
\usepackage{amsmath}
\usepackage{amssymb}
\usepackage{amsthm}
\usepackage{mathrsfs}
\usepackage{dsfont}
\usepackage{subfig}
\usepackage{lineno}
\usepackage{enumerate}
\usepackage{graphicx}
\usepackage{epstopdf}
\usepackage{multirow}
\usepackage{color}
\usepackage{xspace}
\usepackage{pdfsync}
\usepackage{hyperref} 
\usepackage{mathtools}
\usepackage{algorithmicx}
\usepackage{algpseudocode}
\usepackage{algorithm}
\DeclareGraphicsExtensions{.pdf}
\graphicspath{{./PDF/}}
\DeclareMathOperator*{\argmin}{argmin}
\newcommand{\bq}{\begin{equation}}
\newcommand{\eq}{\end{equation}}
\newcommand{\R}{\mathbb{R}}

\newcommand{\abs}[1]{\left\vert#1\right\vert}

\newcommand{\bO}{\mathcal{O}}

\newcommand{\Dt}{\mathcal{D}}

\newcommand{\QW}{quadratic Wasserstein metric\xspace}
\newcommand{\MA}{Monge-Amp\`ere\xspace}

\newcommand{\diag}{\text{diag}}
\newcommand{\vp}{v^\perp}
\newcommand{\M}{\mathcal{M}}
\newcommand*\Laplace{\mathop{}\!\mathbin\bigtriangleup}

\algnewcommand{\LineComment}[1]{\State \(\triangleright\) #1}

\newtheorem{theorem}{Theorem}
\theoremstyle{lemma}

\newtheorem{definition}{Definition}

\theoremstyle{remark}

\newtheorem{remark}{Remark}

\begin{document}

\title[Optimal Transport For Challenging Seismic Inverse Problems]{Analysis and Application of Optimal Transport For Challenging Seismic Inverse Problems}

\author{Yunan Yang}


\address{Courant Institute of Mathematical Sciences, New York University, 251 Mercer Street, New York, NY 10012 USA}
\email{yunan.yang@courant.nyu.edu}

\keywords{computational seismology, full-waveform inversion, optimal transport, Wasserstein distance}

\begin{abstract}
In seismic exploration, sources and measurements of seismic waves on the surface are used to determine model parameters representing geophysical properties of the earth. Full-waveform inversion (FWI) is a nonlinear seismic inverse technique that inverts the model parameters by minimizing the difference between the synthetic data from the forward wave propagation and the observed true data in PDE-constrained optimization. The traditional least-squares method of measuring this difference suffers from three main drawbacks including local minima trapping, sensitivity to noise, and difficulties in reconstruction below reflecting layers. Unlike the local amplitude comparison in the least-squares method, the quadratic Wasserstein distance from the optimal transport theory considers both the amplitude differences and the phase mismatches when measuring data misfit. We will briefly review our earlier development and analysis of optimal transport-based inversion and include improvements, for example, a stronger convexity proof. The main focus will be on the third ``challenge'' with new results on sub-reflection recovery.
\end{abstract}
\date{\today}
\maketitle

\section{Introduction}
At the heart of seismic exploration is the estimation of essential geophysical properties including wave velocity. The development of man-made seismic sources and advanced recording devices now facilitates measurements of entire wavefields in time and space rather than using merely the travel time to estimate the subsurface properties. The full-wavefield setup is a more controlled setting and provides a large amount of data, which is needed for an accurate inverse process of estimating geophysical properties. 
The computational technique referred to as full-waveform inversion (FWI)~\cite{lailly1983seismic,tarantola1982generalized} utilizes information of the entire wavefield and follows the standard strategy of a partial differential equation (PDE) constrained optimization. Even three-dimensional inversion of subsurface elastic parameters using FWI is now possible and has become increasingly popular in exploration applications~\cite{yang2016review}. Currently, FWI can reconstruct sub-surface parameters with stunning detail and resolution~\cite{Virieux2017}. 
Research on FWI in both academia and industry has been very active over the past decade resulting in many new and innovative algorithms and software implementations.

Phase-based inversion methods such as traveltime tomography~\cite{luo1991wave} estimate the background velocity, while linear inversion techniques~\cite{Dai2013} fix the background velocity model, and update the reflectivity distribution.
Unlike these two classes of methods, FWI aims to recover both the low- and high-wavenumber components of the model by considering the full wavefield information.
In both time~\cite{tarantola1982generalized} and frequency~\cite{pratt1990inverse1} domains, the least-squares norm ($L^2$) has been the most widely used misfit function. It is, however now well known that inversion techniques based on $L^2$ face three critical obstacles.

First, the accuracy of $L^2$-based FWI is severely hampered by the lack of low-frequency data and a poor starting model. These limitations are mainly due to the ill-posedness of the inverse problem. The PDE-constrained optimization in FWI is typically solved by local optimization methods in which the subsurface model is described by using a large number of unknowns, and the number of model parameters is determined a priori~\cite{tarantola2005inverse}. As the name suggests, local methods only use the local gradient of the $L^2$ objective function, which is usually nonconvex with respect to the model parameters. As a result, the inversion process is easily trapped in local minima. Recent developments focus on this multiparameter and multi-mode modeling, but there is a dilemma. The more realistic the model is, the more parameters it has, resulting in even worse ill-posedness and even non-uniqueness.

Second, in addition to the difficulties with local minima, an additional problem of the $L^2$ norm is exacerbated by the fact that observed signals usually suffer from noise in the measurements. 
All seismic data contains either natural or experimental equipment noise. For example, the ocean waves lead to extremely low-frequency noise in the marine acquisition. Wind and cable motions also generate random noise.
As a result of the overfitting issue, high-frequency noise in the reconstruction is boosted during the iterative process. Stronger noise can even lead the inversion to local minima. 
Therefore, in selecting a good objective function, its robustness with respect to noise is essential. 

Third, traditional FWI has difficulty in accurately updating deeper features with reflection-dominated data. Diving waves are wavefronts continuously refracted upwards through the earth due to the presence of a vertical velocity gradient. Due to limitations of the source-receiver distribution, there might be no diving waves traveling through the depth of interests or being recorded by the receivers, and reflections are usually the only available information representing the subsurface models. Conventional FWI using reflection data has been problematic in the absence of a really good initial model. Conventional $L^2$-based FWI only recovers a migration-type structure with severe overshooting. The high-wavenumber features updated by reflections often slow down the recovery of the missing low-wavenumber components. Often, the entire optimization scheme stalls.

The current challenges of $L^2$ norm-based FWI motivate us to replace the traditional $L^2$ norm with a new metric with better convexity and stability for seismic inverse problems.  Engquist and Froese~\cite{EFWass} first proposed to use the Wasserstein distance as an alternative objective function measuring the difference between synthetic data $f$ and observed data $g$. In our previous studies of the quadratic Wasserstein metric ($W_2$)~\cite{yang2017application,yangletter,engquist2016optimal,Survey1, Survey2}, we have addressed the first two challenges, the nonconvexity of the traditional $L^2$ norm, and its sensitivity to noise. We mainly focused on FWI that primarily uses diving waves or data from shallow reflectors to refine the velocity models. 

In this paper, we present results that demonstrate yet another advantage of $W_2$. Precisely we will show that $W_2$ is also able to mitigate the third drawback of the traditional $L^2$ method. The new material is related to the third challenge of the $L^2$ norm-based FWI and is beyond the well-known local minima or, so-called, cycle skipping issues. We will investigate properties of optimal transport for challenging inversion tests with reflection-dominated data and demonstrate that partial inversion for velocity below the deepest reflecting interface is still possible by using the quadratic Wasserstein distance from optimal transport theory. 
 
The paper is arranged as follows. We start by introducing the essential background of full-waveform inversion in Section~\ref{sec:FWI}, including the forward problem and the inverse problem, and its formulation as a PDE-constrained optimization. In Section~\ref{sec:OT}, we will briefly revisit several key concepts of optimal transport and define the quadratic Wasserstein distance ($W_2$), the main tool of this paper. We will address the importance of data normalization as well as numerical methods for optimal transport. Using the fact that optimal transport problem with the quadratic cost can be rigorously related to the \MA equation, in Section~\ref{sec:MA} we present a monotone finite difference \MA solver developed in~\cite{FOFiltered} which is proven to converge to the corresponding viscosity solution~\cite{FroeseTransport}. To illustrate the efficiency of the solver we present results for the well known Marmousi benchmark. The rest of the paper (Section~\ref{sec:challenge1}, \ref{sec:challenge2} and \ref{sec:challenge3}) is dedicated to the new or improved analysis of essential properties of $W_2$ based inversion. Optimal transport compares signals globally and naturally combines misfits in both amplitude differences and phase mismatches. The improved results of $W_2$-based inversion demonstrate the capacity of this new objective function in providing better convexity in inversion for many different types of seismic data, including transmission, refraction, and reflection. The $W_2$ distance captures the essential low-frequency components of the data residual, which is directly linked to the low-wavenumber structures of the velocity model. Both mathematical analysis and numerical examples demonstrate that $W_2$ is an advantageous choice for the objective function in data-driven inversion. 

\section{Full-Waveform Inversion} \label{sec:FWI}
Seismic data contains interpretable information about subsurface properties. Imaging predicts the spatial distribution of earthquakes and specifies the values of geological properties that are useful in exploration seismology. The state-of-the-art inverse method in exploration geophysics is full-waveform inversion, which can be seen as a PDE-constrained optimization. The forward operators are various types of wave equations. The goal of FWI is to find both the small-scale and large-scale components that describe geophysical properties using the entire content of seismic time history at a receiver, so-called traces. 
\subsection{The Forward Problem}
The forward problem of seismic inversion amounts to modeling and simulating the propagation of waves. The earth is complex, with various heterogeneities and multiple scales. Therefore, fast and accurate forward modeling is a significant step in seismic imaging. The current research in FWI covers multiple-parameter inversion using seismic waveforms, where the forward modeling includes anisotropic parameters, density, and attenuation factors~\cite{yang2016review} as well as viscoelastic modeling. It should be noted that the more parameters the model has, the less well-posed the inverse problem is. The real physics is far more complicated than the simple acoustic setting of this paper, but the industry standard is the acoustic model in time or frequency domains.

In this paper, we consider the inverse problem of finding the wave velocity of an acoustic wave equation in the interior of a domain from knowing the Cauchy boundary data together with natural boundary conditions~\cite{engquist1977absorbing}. The modeled synthetic data is $f(\mathbf{x_r},t;m) = u(\mathbf{x_r},z=0,t;m)$, where $\mathbf{x_r}$ is the receiver location, $u$ is the wavefield, obtained by solving, for example, the following 2D acoustic wave equation:
\begin{equation}\label{eq:FWD}
     \left\{
     \begin{array}{rl}
     & m(\mathbf{x})\frac{\partial^2 u(\mathbf{x},t)}{\partial t^2}- \Laplace u(\mathbf{x},t) = s(\mathbf{x},t),\\
    & u(\mathbf{x}, 0 ) = 0,                \\
    & \frac{\partial u}{\partial t}(\mathbf{x}, 0 ) = 0,    \\
     \end{array} \right.
\end{equation}
where the model parameter is the squared slowness $m(\mathbf{x}) = \frac{1}{c(\mathbf{x})^2}$ where $c(\mathbf{x})$ is the wave velocity, $u(\mathbf{x},t)$ is the wavefield and $s(\mathbf{x},t)$ is the source. The above wave equation can be seen as the forward operator $F$ such that  $f = F(m)$. It is a linear PDE but a nonlinear map from the model $m(\mathbf{x})$ to the data $u(\mathbf{x},t)$, where $\mathbf{x} \in \mathbb{R}^2$. 
We note that there are numerous techniques for approximating~\eqref{eq:FWD}, for example, discontinuous and continuous finite elements, spectral elements and finite difference methods. As our focus is on the inverse problem rather than on the solution of the forward problem, we will restrict our discretization to standard finite difference methods~\cite{moczo2007finite}. 


\subsection{FWI as PDE-Constrained Optimization}
The goal of FWI is to recover the subsurface model parameters $m^*$ from the observed true data $g$ such that 
\bq
g(\mathbf{x_r},t) = F(m^*) = u(\mathbf{x_r},z=0,t;m^*).
\eq 
The forward operator $F$ is nonlinear, and the inverse problem most often does not fulfill Hadamard's postulates of well-posedness. 

An alternative way of formulating the inverse problem is to estimate the true model parameter $m^*$ through the solution of an optimization problem
\bq \label{eq:simple_misfit}
m^* = \argmin \limits_m J(F(m), g),
\eq
where $J$ is a suitable choice of objective/loss/misfit function characterizing the difference between the synthetic data $F(m)$  generated by the current (and inaccurate) model parameter $m$ and the observable true data $g$. Full-waveform inversion (FWI) is a data-fitting approach, similar to many other inverse problems that are formulated as PDE-constrained optimization. \newline

\subsection{The Choice of Objective Function}~\label{sec:fwi_obj}
In seismic inversion, the misfit function, sometimes also referred to as the objective or loss function, is a functional on the data domain that measures the mismatch between the synthestic data and the observed data. The conventional objective function is the least-squares norm ($L^2$), which suffers from local minimum trapping and sensitivity to noise~\cite{Virieux2017}:
\begin{equation}\label{eq:L2_misfit}
J_0(m)=\frac{1}{2}\sum\limits_{r=1}^R \int\limits_0^{T_0} \abs{f(\mathbf{x_r},t;m)-g(\mathbf{x_r},t)}^2dt,
\end{equation}
where $g$ is the observed data, $f = F(m)$ is the simulated data, $\mathbf{x_r}$ are receiver locations, $T_0$ is the total recording time and $m$ is the model parameter. This formulation can also be extended to the case of multiple sources. In both time and frequency domains~\cite{tarantola1982generalized,pratt1990inverse1}, the $L^2$ norm has been the most widely used misfit function in FWI. 
The oscillatory and periodic nature of waveforms leads to a primary challenge in inversion called ``cycle skipping''. 
If the true data and the initial synthetic data are more than half a wavelength ($>\frac{\lambda}{2}$) away from each other, the first gradient of the $L^2$ norm can point in the wrong direction regarding the phase mismatch, but can nonetheless reduce the data misfit in the fastest manner~\cite{Beydoun1988}. Mathematically, this is due to the highly nonlinear nature of the inverse problem, the nonconvexity of the objective function, and finally can result in finding local minima. Figure~\ref{fig:2_ricker_signal} illustrates one typical example where the $L^2$ norm suffers from the ``cycle skipping'' issue, and consequently results in many local minima in the optimization landscape, as shown in Figure~\ref{fig:2_ricker_L2}.

\begin{figure}
\centering
    \subfloat[Two signals]{\includegraphics[height=0.1\textheight]{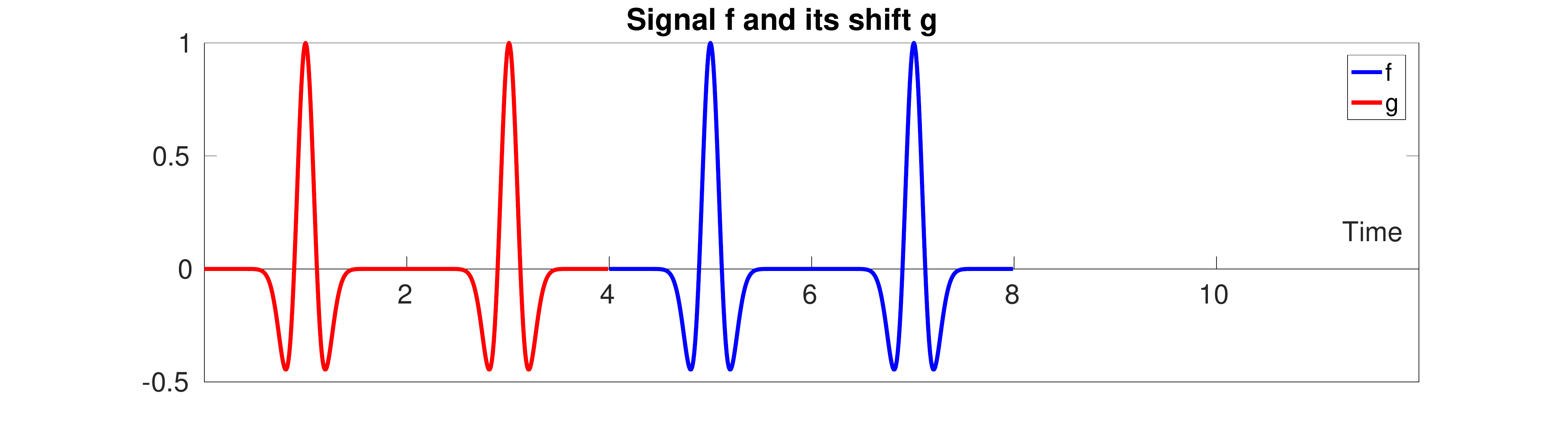}\label{fig:2_ricker_signal}}
      \subfloat[$L^2$ misfit of $s$]
{\includegraphics[height=0.1\textheight]{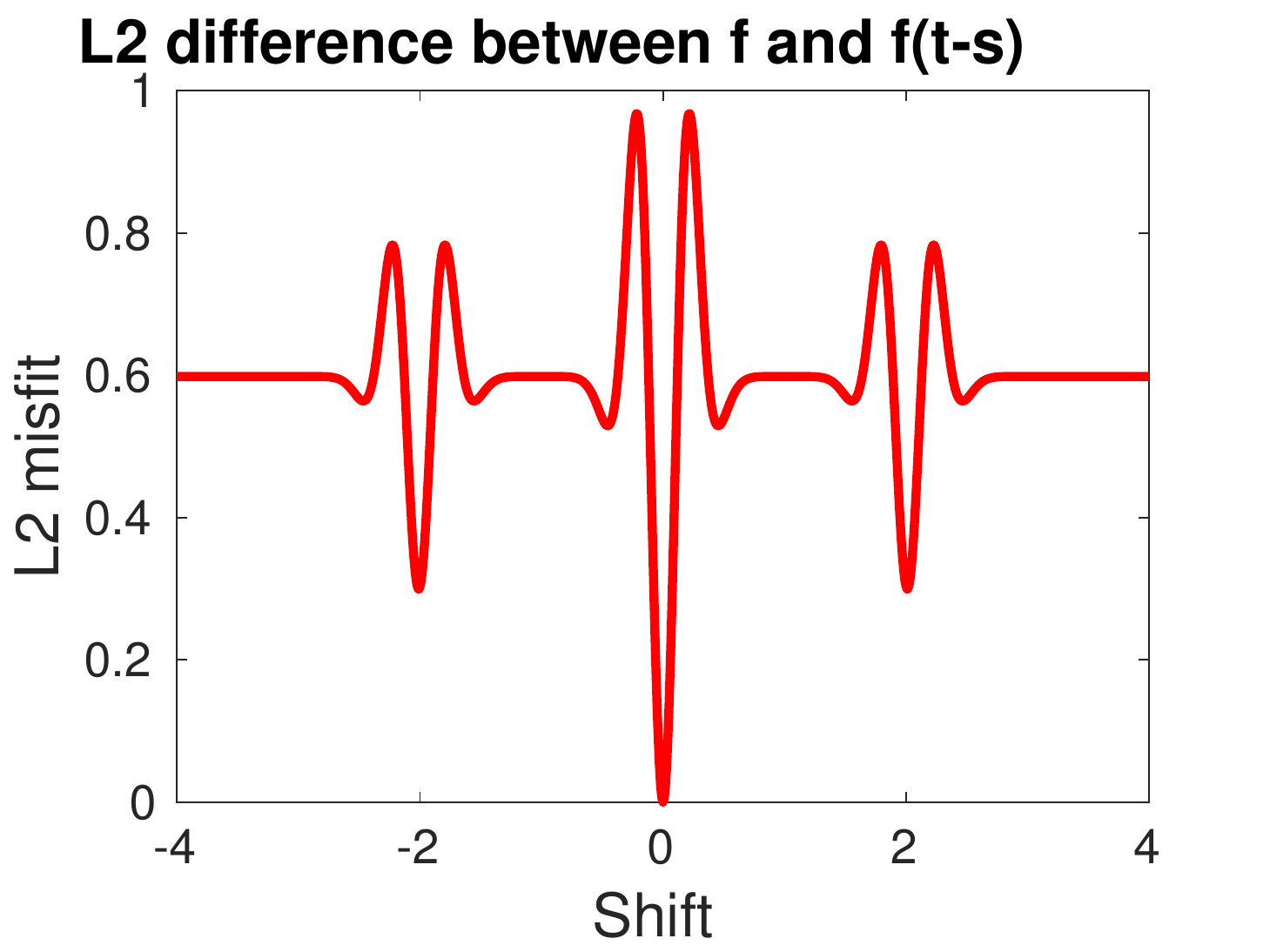}\label{fig:2_ricker_L2}}
      \subfloat[$W_2$ metric of $s$]
{\includegraphics[height=0.1\textheight]{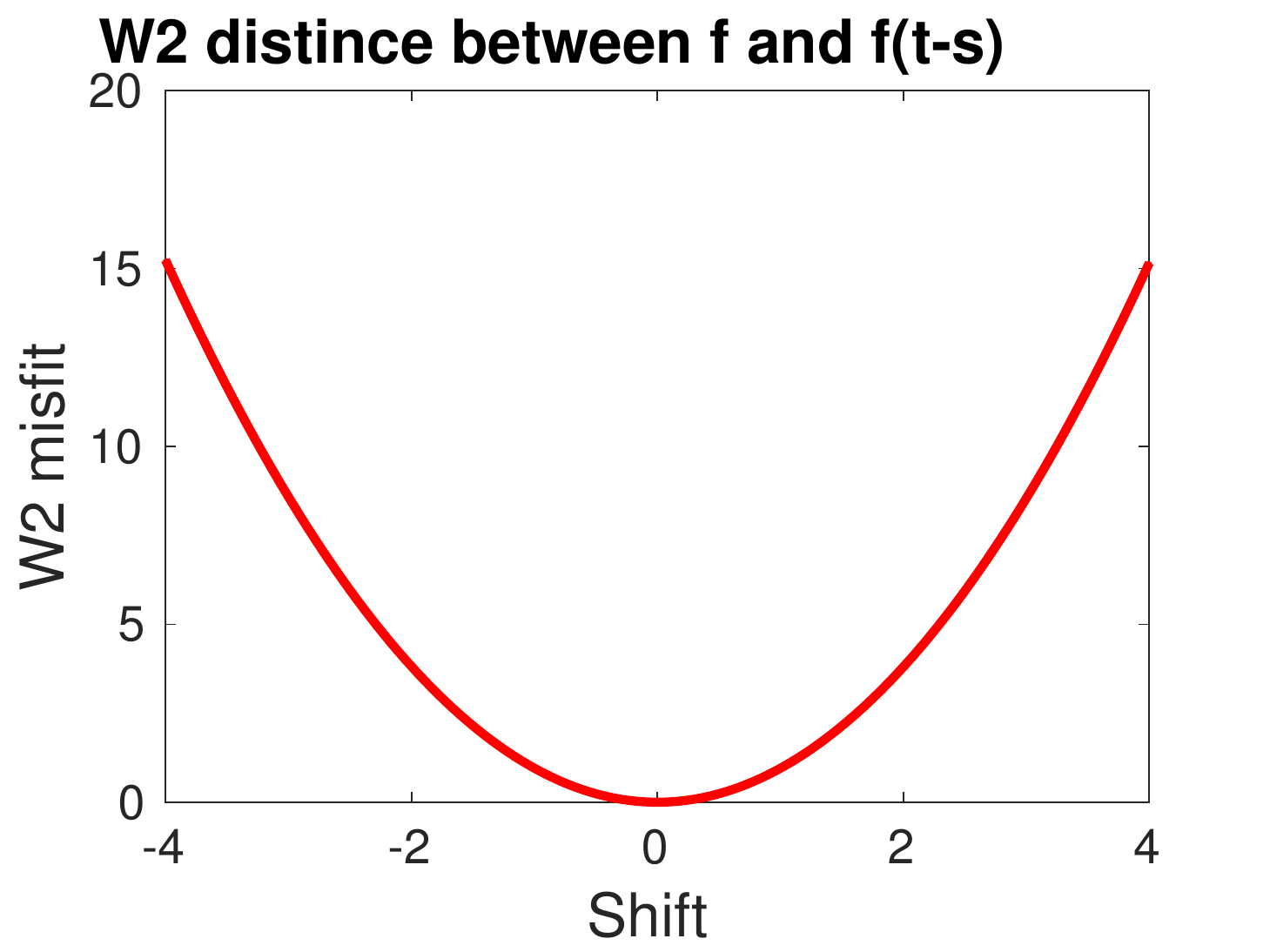}\label{fig:2_ricker_W2}}
\caption{(a)~A signal consisting of two Ricker wavelets (blue) and its shift (red);~(b)~$L^2$ norm between $f(t)$ and $f(t-s)$ in terms of shift $s$;~(c)~$W_2$ distance between $f(t)$ and $f(t-s)$ in terms of shift $s$.}~\label{fig:2_ricker}
\end{figure}

The current challenges of FWI motivate us to replace the traditional $L^2$ norm with a new metric of better convexity and stability for seismic inverse problems.  
Soon after~\cite{EFWass}, there have been fruitful activities in the past four years in developing the idea of using optimal transport based objective functions for FWI from both academia~\cite{chen2017quadratic,Puthawala2018,engquist2016optimal,W1_2D,W1_3D,yang2017application,
Ballesio2018,Motamed2018} and industry~\cite{qiu2017full,Ramos2018, poncet2018fwi} with several field data applications. 

By using the $W_2$ distance to measure the two signals in Figure~\ref{fig:2_ricker_signal}, we obtain a globally convex optimization landscape as illustrated in Figure~\ref{fig:2_ricker_W2}. It has quite different features compared with Figure~\ref{fig:2_ricker_L2}, the result of the $L^2$ norm. The $W_2$ distance naturally combines misfits in both amplitude and phase,  and consequently achieves global convexity in data shifts. The quadratic Wasserstein distance naturally avoids many local minima of  $L^2$ norm that are generated by phase mismatches.

\section{Optimal Transport and the Wasserstein Distance}~\label{sec:OT}
The optimal transport problem seeks the most efficient way of transforming one distribution of mass to the other, relative to a given cost function. It was first introduced by Monge in 1781~\cite{Monge} and later expanded on by Kantorovich~\cite{kantorovich1960mathematical}. 
Optimal transport related techniques are nonlinear as they explore the model through both signal intensities and their locations. 
Since the 1990s, the mathematical analysis of optimal transport problems~\cite{brenier1991polar,Evans,Gangbo1996,rachev1998mass, Santambrogio,Villani,villani2008optimal}, together with recent advancements in numerical methods~\cite{ryu2018unbalanced,benamou2014numerical,Cuturi2013,Julien2011,Li2016,Oberman2015}, have driven the ongoing development of numerous applications based on optimal transport~\cite{kolouri2016transport}. Next, we will briefly review basic concepts in optimal transport and introduce the Wasserstein distance.

\subsection{The Wasserstein Metric}
Let X and Y be two metric spaces with nonnegative Borel measures $\mu$ and $\nu$ respectively. Assume X and Y have equal total measure:
\bq
\int_X d\mu = \int_Y d\nu.
\eq
Without loss of generality, we will hereafter assume the total measure to be one, i.e., $\mu$ and $\nu$ are probability measures. Given two nonnegative densities $f = d\mu$ and $g=d\nu$, we are interested in the mass-preserving map $T$ such that $f = g \circ T$. 
\begin{definition}[Mass-preserving map] \label{def:mass_preserve}
A transport map $T: X \rightarrow Y$ is mass-preserving if for any measurable
set $B \in Y$,
\bq ~\label{eq:mass_preserve1}
\mu (T^{-1}(B)) = \nu(B).
\eq
Then $\nu$ is said to be the push-forward of $\mu$ by $T$, and we write $\nu = T_\# \mu $.
\end{definition}

The transport cost function $c(x,y)$ maps pairs $(x,y) \in X\times Y$ to $\mathbb{R}\cup \{+\infty\}$, which denotes the cost of transporting one unit mass from location $x$ to $y$. The most common choices of $c(x,y)$ include $|x-y|$ and $|x-y|^2$, which correspond to the Manhattan distance and the Euclidean distance between vectors $x$ and $y$ respectively. Once we find a mass-preserving map $T$, the cost corresponding to $T$ is 
\[
I(T,f,g,c) = \int\limits_Xc(x,T(x))f(x)\,dx. 
\]
We are interested in finding the optimal map that minimizes the total cost
\[
I(f,g,c) = \inf\limits_{T\in\M}\int\limits_Xc(x,T(x))f(x)\,dx,
\]
where $\M$ is the set of all maps that rearrange $f$ into $g$.

Thus, we have informally defined the optimal transport problem, the optimal map, as well as the optimal cost, which is also called the Wasserstein distance for the class of cost function $c(x,y) = |x-y|^p$:
\begin{definition}[The Wasserstein distance]~\label{def:OT}
  We denote by $\mathscr{P}_p(X)$ the set of probability measures with finite moments of order $p$. For all $p \in [1, \infty)$,   
\bq~\label{eq:static}
W_p(\mu,\nu)=\left( \inf _{T_{\mu,\nu}\in \mathcal{M}}\int_{\mathbb{R}^d}\left|x-T_{\mu,\nu}(x)\right|^p d\mu(x)\right) ^{\frac{1}{p}},\quad \mu, \nu \in \mathscr{P}_p(X).
\eq
$\mathcal{M}$ is the set of all maps that rearrange the distribution $\mu$ into $\nu$.
\end{definition}
In this paper, we will focus on the quadratic cost ($p=2$). The corresponding misfit is the quadratic Wasserstein distance ($W_2$). We will also assume that $\mu$ does not give mass to small sets~\cite{brenier1991polar}, which is a necessary condition to guarantee the existence and uniqueness of the optimal map under the quadratic cost. This requirement is also natural for seismic signals.

For FWI, one can use the quadratic Wasserstein metric ($W_2$) to measure the misfit in the time domain, and the $L^2$ norm for the spatial domain. Computing the misfit function  becomes a 1D optimal transport problem, which can be solved explicitly. The overall misfit is then
\bq \label{eq:W21D}
J_1(m) = \sum\limits_{r=1}^R W_2^2(f(x_r,t;m),g(x_r,t)),
\eq
where $f(x_r,t)$ and $g(x_r,t)$ are the normalized synthetic data and the normalized observed data at each fixed spatial location $x_r$. We will discuss the topic of data normalization in Section~\ref{sec:OT_data_normalization}.


One can also compare the entire datasets without fixing the spatial location by solving a higher-dimensional optimal transport problem:
\bq \label{eq:W22D}
J_2(m) = W_2^2(f(\mathbf{x},t;m),g(\mathbf{x},t)).
\eq
We will introduce computation methods for these two optimal transport based objective functions $J_1$ and $J_2$ in Section~\ref{sec:OT_Numerical} and Section~\ref{sec:MA}, respectively.

\subsection{Data Normalization for Seismic Waveforms} \label{sec:OT_data_normalization}
There are two main requirements for signals $f$ and $g$ in optimal transport theory:
\bq
f(t)\geq 0,\ g(t)\geq 0,\ <f> = \int f(t) dt = \int g(t) dt = <g>.
\eq
Since these constraints are not expected for seismic signals, some data pre-processing is needed before we can implement the Wasserstein-based FWI. Previously in~\cite{yangletter,yang2017application}, we normalized seismic signals by adding a constant as follows,
\bq\label{eq:linear}
\tilde{f}(t) = \frac{f(t) + c}{<f+c>},\ \tilde{g}(t) = \frac{g(t) + c}{<g+c>},\ c = \min_t(f(t),g(t)).
\eq
This works remarkably well in realistic, large-scale examples~\cite{yangletter,yang2017application} together with the adjoint-state method~\cite{vogel2002computational,Plessix} for optimization using either the 1D or \MA based techniques. However, this linear normalization does not produce a convex misfit functional with respect to simple shifts.

Based on the analysis of different normalization methods in~\cite{Survey2}, a bijection between the original data and the normalized data is essential so as not to deteriorate the ill-posedness of the inverse problem.
An exponential-based normalization was proposed in~\cite{qiu2017full} to transform seismic signals into probability distributions. We proposed a new normalization in~\cite{Survey2} that satisfies most of the essential properties:
\begin{equation}\label{eq:mixed}
     \tilde{f}(t) = \left\{
     \begin{array}{rl}
     &  (f(t) + \frac{1}{c}) / b,\ f(t) \geq 0,\ c>0   \\
     & \frac{1}{c} \exp(cf(t)) / b,\ f(t)<0  \end{array} \right.
\end{equation}
where $b = <(f + \frac{1}{c})\mathds{1}_{f\geq 0} + \frac{1}{c} \exp(cf) \mathds{1}_{f< 0}  >$.

For the limit of small $c >0$, this is a linear scaling, which directly follows from Taylor expansion of the exponential part. In the limit of large $c$ values, $f$ obviously converges to $f^+ = \max \{f,0\}$. This is a $C^1$ function, compatible with the adjoint-state method outlined in Appendix~\ref{sec:adjoint_state_method}. The coefficient $c$ is selected based on the amplitude of the raw signal. This sign-sensitive normalization keeps the convexity of the quadratic Wasserstein distance with respect to signal translation as shown in Figure~\ref{fig:cvx_mixed}.

\begin{figure}
\includegraphics[width=0.9\textwidth]{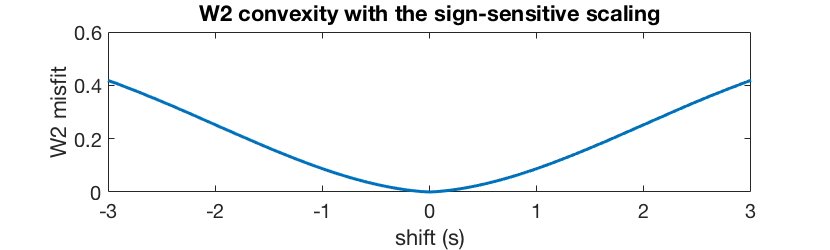}
\caption{The $W_2$ misfit regarding signal shift $s$ using sign-sensitive scaling, i.e., $W_2^2(f(t-s),f(t))$.}\label{fig:cvx_mixed}
\end{figure}

Optimal transport using exponential scaling or sign-sensitive scaling mainly focuses on the transport of $f^+$. 
If we denote the normalization function in ~\eqref{eq:mixed} as an operator $P$ such that $\tilde{f} = Pf$, we can use both parts of the signal by considering the following objective function: 
\bq\label{eq:mixmix}
J_{3}(m) = W_2^2 (P(f), P(g)) + W_2^2(P(-f), P(-g)).
\eq 
\ 
\subsection{Numerical Computation of Optimal Transport}~\label{sec:OT_Numerical}
In higher dimensions, it is not trivial to solve the optimal transport problem between the synthetic data and the observed data from the inverse problem. Here we will briefly introduce two main ideas for calculating the Wasserstein distance in seismic inversion: the 1D technique, and computational optimal transport in higher dimensions.
\subsubsection{The 1D Technique} \label{sec:OT_Numerical_1D}
In~\cite{yang2017application}, we propose a 1D technique for approximating the $W_2$ distance in full-waveform inversion (FWI). A trace is the time history measured at a receiver. For example, the 1D seismic signal $f(\mathbf{x_r},t;m) = u(\mathbf{x_r},z=0,t;m)$ is a trace of the whole observable waveform $f(\mathbf{x},t;m)$  at receiver $\mathbf{x_r}$. The advantage of this trace-by-trace technique is to compute the Wasserstein distance based on the 1D explicit formula; see Theorem~\ref{thm:OT1D} below. The objective function $J_1$ in~\eqref{eq:W21D} is mathematically equivalent to a measure that is $W_2$ metric in the time domain and $L^2$ norm in the spatial domain.

For $f$ and $g$ in one dimension, it is possible to exactly solve the optimal transport problem~\cite{Villani} in terms of the cumulative distribution functions
\bq \label{eq:F&G}
F(x) = \int_{-\infty}^x f(t)\,dt, \quad G(y) = \int_{-\infty}^y g(t)\,dt.
\eq

In fact, the optimal map is just the unique monotone rearrangement of the density $f$ into $g$. To compute the Wasserstein distance ($W_p$), we need the cumulative distribution functions $F$ and $G$ and their inverses $F^{-1}$ and $G^{-1}$ as the following theorem states:

\begin{theorem}[Optimal transportation on $\R$~\cite{villani2008optimal}]\label{thm:OT1D}
Let $0 < f, g < \infty$ be two probability density functions, each supported on a connected subset of $\R$.  Then the optimal map from $f$ to $g$ is $T = G^{-1}\circ F$.
\end{theorem}

From the theorem above, we derive another formulation for the 1D quadratic Wasserstein metric:
\bq\label{eq:myOT1D}
W_2^2(f,g)  = \int_0^1 |F^{-1} - G^{-1}|^2  dy = \int_X |x-G^{-1}(F(x))|^2 f(x)dx.
\eq
The corresponding Fr\'{e}chet derivative is
\bq~\label{eq:1D_ADS_C}
\begin{split}
\frac{\partial W_2^2(f,g)}{\partial f} = & \left( \int_t^{T_0}-2(s-G^{-1}(F(s))\frac{dG^{-1}(y)}{dy}\biggr\rvert_{y=F(s)}  f(s) ds \right) dt \\ & +
 |t-G^{-1}(F(t))|^2 dt.
\end{split}
\eq
The Fr\'{e}chet derivative of the objective function is often referred to as the adjoint source term in the adjoint wave equation~\eqref{eq:FWI_adj}. It is a key element of computing the model gradient $\frac{\partial J}{\partial m}$ based on the adjoint-state method~\cite{Plessix}.

\subsubsection{Computational Optimal Transport in Higher Dimensions}
Optimal transport is a well-studied subject in mathematical analysis while the computation techniques are comparatively underdeveloped. 
In general, the computation for optimal transport becomes very challenging in higher dimensions. There are combinatorial techniques that typically are computationally costly in higher dimensions, for example, the Hungarian algorithm~\cite{kuhn1955hungarian}. However, there are recent results on speeding up the numerical methods, especially for large-scale problems. 

Many numerical methods were proposed based on the Benamou-Brenier fluid formulation~\cite{BenBre,Li2016,ryu2018unbalanced}, the optimal transport problem with regularizers~\cite{benamou2015iterative,essid2018quadratically}, the dual formulation based on linear programming~\cite{Oberman2015,schmitzer2016sparse,liu2018multilevel}, as well as the solution to a \MA partial differential equation~\cite{benamou2014numerical,benamou2017minimal}. In particular, we discuss seismic inversion based on the \MA equation in Section~\ref{sec:MA}. 

These fast computational algorithms and new thrusts of research efforts are helping to translate attractive theoretical properties of optimal transport into elegant and scalable tools for a wide variety of applications involving science and engineering. On the other hand, it is important and of great research interest to quantify the trade-off in performance between the acceleration gained from the optimal transport based methods, and the slowdown due to the higher computational cost.


\section{Large-Scale Seismic Inversion Based on the \MA Equation}\label{sec:MA}
Optimal transport with the quadratic cost function $c(x,y) = |x-y|^2$ can be rigorously related to the \MA equation~\cite{brenier1991polar,KnottSmith}, which enables the construction of more efficient methods for computing the quadratic Wasserstein metric. We present a monotone finite difference \MA solver, proven to converge to the viscosity solution~\cite{barles1991convergence,FroeseTransport}. 
Based on this solver, we can compute the $W_2$ distance and the Fr\'{e}chet derivative required by the adjoint-state method (see Appendix~\ref{sec:adjoint_state_method}) to compute the gradient of the model parameters. The results of inversions based on the \MA solver are compared with ones that are produced by using the $L^2$ norm to measure the misfit between the synthetic data and the observed true data.
\subsection{Optimal Transport and the \MA Equation}
If the optimal transport map $T(x)$  is sufficiently smooth and $\det(\nabla T(x) ) \neq 0$, Definition~\ref{def:mass_preserve} naturally leads to the requirement
\bq\label{eq:mass_preserve2}
f(x) = g(T(x))\det(\nabla T(x)).
\eq 
As discussed in Section~\ref{sec:convexity}, the optimal transport map is cyclically monotone and a cyclically monotone mapping is equivalent to the gradient of a convex function~\cite{brenier1991polar,KnottSmith}.  Making the substitution $T(x) = \nabla u(x)$ into the constraint~\eqref{eq:mass_preserve2} leads to the \MA equation
\bq \label{eq:MAA}
\det (D^2 u(x)) = \frac{f(x)}{g(\nabla u(x))}, \quad u \text{ is convex}.
\eq
In order to compute the misfit between distributions $f$ and $g$, we first compute the optimal map $T(x) = \nabla u(x)$ via the solution of this \MA equation coupled to the non-homogeneous Neumann boundary condition 
\bq\label{eq:BC}
\nabla u(x) \cdot \nu = x\cdot \nu, \,\, x \in \partial X.
\eq
The squared $W_2$ distance is then given by
\bq\label{eq:WassMA}
W_2^2(f,g) = \int_X f(x)\abs{x-\nabla u(x)}^2\,dx.
\eq
%
\subsection{The Finite Difference \MA solver}
As we have seen for the quadratic Wasserstein distance, the optimal map can be computed via the solution of a \MA partial differential equation~\cite{benamou2014numerical}. This approach has the advantage of drawing on the well-developed field of numerical analysis of PDE. We solve the \MA equation numerically for the viscosity solution using an almost-monotone finite difference method, relying on the following reformulation of the \MA operator, which automatically enforces the convexity constraint~\cite{FroeseTransport}.


The numerical scheme of~\cite{benamou2014numerical} uses the theory of~\cite{barles1991convergence} to construct a convergent discretization of the \MA equation~\eqref{eq:MAA} as stated in Theorem~\ref{thm:MA_convergence}. A variational characterization of the determinant on the left hand side, which also involves the negative part of the eigenvalues, was proposed as the following equation:
\begin{multline}\label{eq:MA_convex}
{\det}(D^2u) = \\ \min\limits_{\{v_1,v_2\}\in V}\left\{\max\{u_{v_1,v_1},0\} \max\{u_{v_2,v_2},0\}+\min\{u_{v_1,v_1},0\} + \min\{u_{v_2,v_2},0\}\right\},
\end{multline}
where $V$ is the set of all orthonormal bases for $\R^2$.  

Equation~\eqref{eq:MA_convex} can be discretized by computing the minimum over finite number of directions $\{\nu_1,\nu_2\}$, which may require the use of a wide stencil.  
In the low-order version of the scheme, the minimum in~\eqref{eq:MA_convex} is approximated using only two possible values.  The first uses directions aligning with the grid axes.
\begin{multline}\label{MA1}
MA_1[u] = \max\left\{\Dt_{x_1x_1}u,\delta\right\}\max\left\{\Dt_{x_2x_2}u,\delta\right\} \\+ \min\left\{\Dt_{x_1x_1}u,\delta\right\} + \min\left\{\Dt_{x_2x_2}u,\delta\right\} - f / g\left(\Dt_{x_1}u, \Dt_{x_2}u\right) - u_0,
\end{multline}
where $\Dt_{x_1}$, $\Dt_{x_2}$, $\Dt_{x_1x_1}$ and $\Dt_{x_2x_2}$ are the first order and second order finite difference operators. If we assume $dx$ is the resolution of the grid, $\delta>K\Delta x/2$ is a small parameter that bounds second derivatives away from zero. $u_0$ is the solution value at a fixed point in the domain. $K$ is the Lipschitz constant in the $y$-variable of $f(x)/g(y)$.

For the second value, we rotate the axes to align with the corner points in the stencil, which leads to
\begin{multline}\label{MA2}
MA_2[u] = \max\left\{\Dt_{vv}u,\delta\right\}\max\left\{\Dt_{\vp\vp}u,\delta\right\} + \min\left\{\Dt_{vv}u,\delta\right\} + \min\left\{\Dt_{\vp\vp}u,\delta\right\}\\ - f / g\left(\frac{1}{\sqrt{2}}(\Dt_{v}u+\Dt_{\vp}u), \frac{1}{\sqrt{2}}(\Dt_{v}u-\Dt_{\vp}u)\right) - u_0.
\end{multline}
Then the monotone approximation of the \MA equation is
\bq\label{eq:MA_compact} M_M[u] \equiv -\min\{MA_1[u],MA_2[u]\} = 0. \eq
We also define a second-order approximation, obtained from a standard centred difference discretization,
\bq\label{eq:MA_nonmon} M_N[u] \equiv -\left((\Dt_{x_1x_1}u)(\Dt_{x_2x_2}u)-(\Dt_{x_1x_2}u^2)\right) + f/g\left(\Dt_{x_1}u,\Dt_{x_2}u\right) + u_0 = 0.\eq
These are combined into an almost-monotone approximation of the form
\bq\label{eq:MA_filtered} M_F[u] \equiv M_M[u] + \epsilon S\left(\frac{M_N[u]-M_M[u]}{\epsilon}\right) \eq
where $\epsilon$ is a small parameter and the filter $S$ is designed to stick to the monotone scheme if $M_M[u]$ and $M_N[u]$ have very different values. Neumann boundary conditions are implemented using standard one-sided differences.  As described in~\cite{engquist2016optimal,FroeseTransport}, the formal Jacobian $\nabla M_F[u]$ of the scheme can be obtained exactly.


\begin{theorem}[Convergence to Viscosity Solution~{\cite[Theorem 4.4]{FroeseTransport}}]\label{thm:MA_convergence}
Let the \MA equation \eqref{eq:MAA} have a unique viscosity solution and let $g>0$ be Lipschitz continuous on $\R^d$. Then the solutions of the scheme \eqref{eq:MA_filtered} converge to the viscosity solution of \eqref{eq:MAA} with a formal discretization error of $\bO(Lh^2)$ where $L$ is the Lipschitz constant of the right hand side and $h$ is the resolution of the grid.
\end{theorem} 


Once the discrete solution $u$ is computed, the squared Wasserstein metric is approximated via
\bq\label{eq:WassDiscrete}  
W_2^2(f,g) \approx \sum\limits_{j=1}^n (x_j-D_{x_j}u)^T\diag(f)(x_j-D_{x_j}u) dt,
\eq
where $n$ is the dimension of the data $f$ and $g$.
Then the gradient of the discrete squared Wasserstein metric can be expressed as
\bq 
\frac{\partial W_2^2(f,g) }{\partial f} = \sum\limits_{j=1}^n  \left[-2\nabla M_F^{-1}[u]^TD_{x_j}^T\diag(f) \right](x_j - D_{x_j}u)dt  + \sum\limits_{j=1}^n  |x_j-D_{x_j}u|^2dt.
\eq
Similar to Equation~\eqref{eq:1D_ADS_C}, this term is a discretized version of the Fr\'{e}chet derivative of the misfit function~\eqref{eq:W22D} with respect to the synthetic data $f$, i.e., the source term $\frac{\partial J}{\partial f}$ in the adjoint wave equation~\eqref{eq:FWI_adj}.


\begin{figure}
\centering
  \subfloat[]{\includegraphics[width=0.25\textwidth]{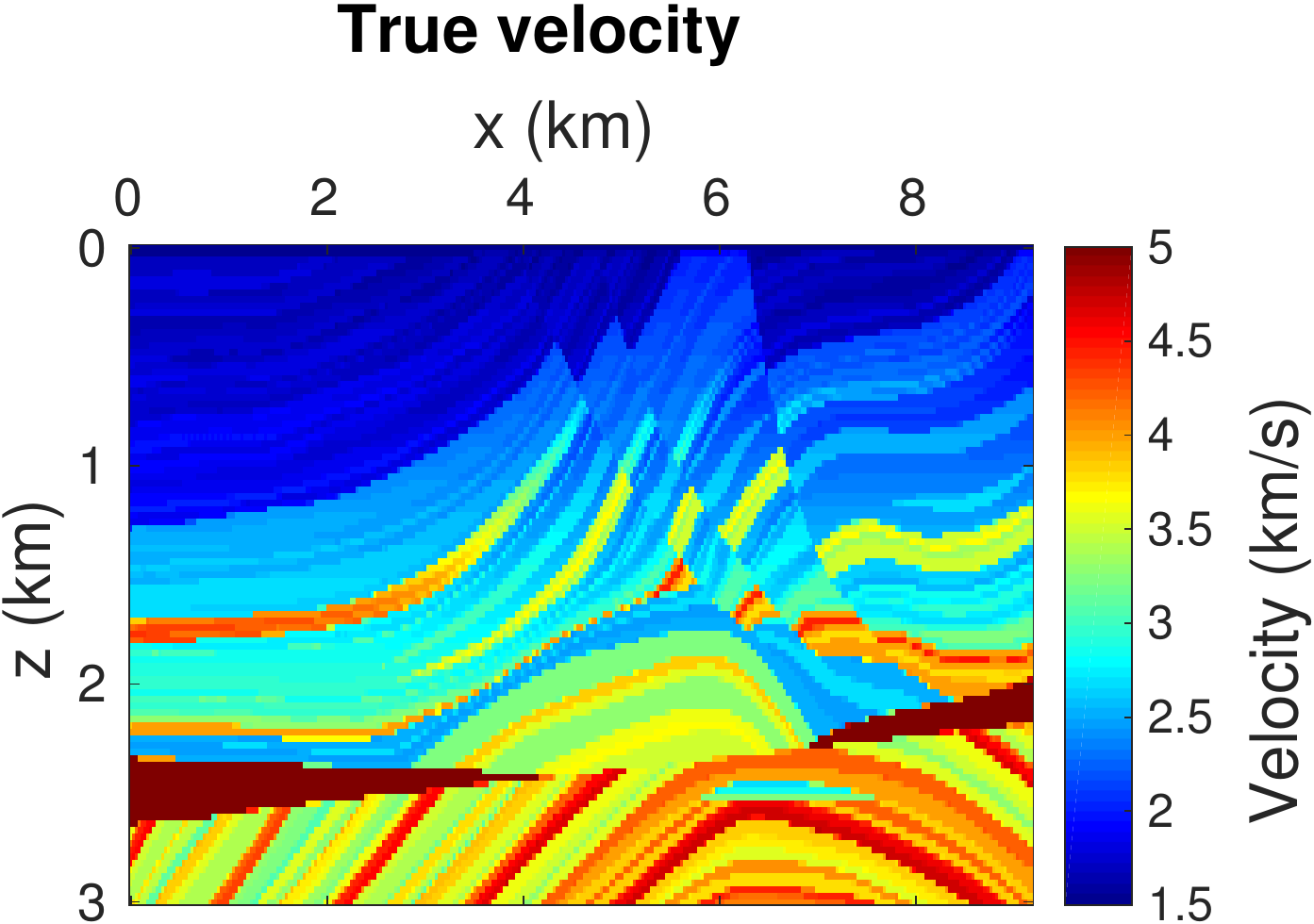}\label{fig:marm2_true}}
  \subfloat[]{\includegraphics[width=0.25\textwidth]{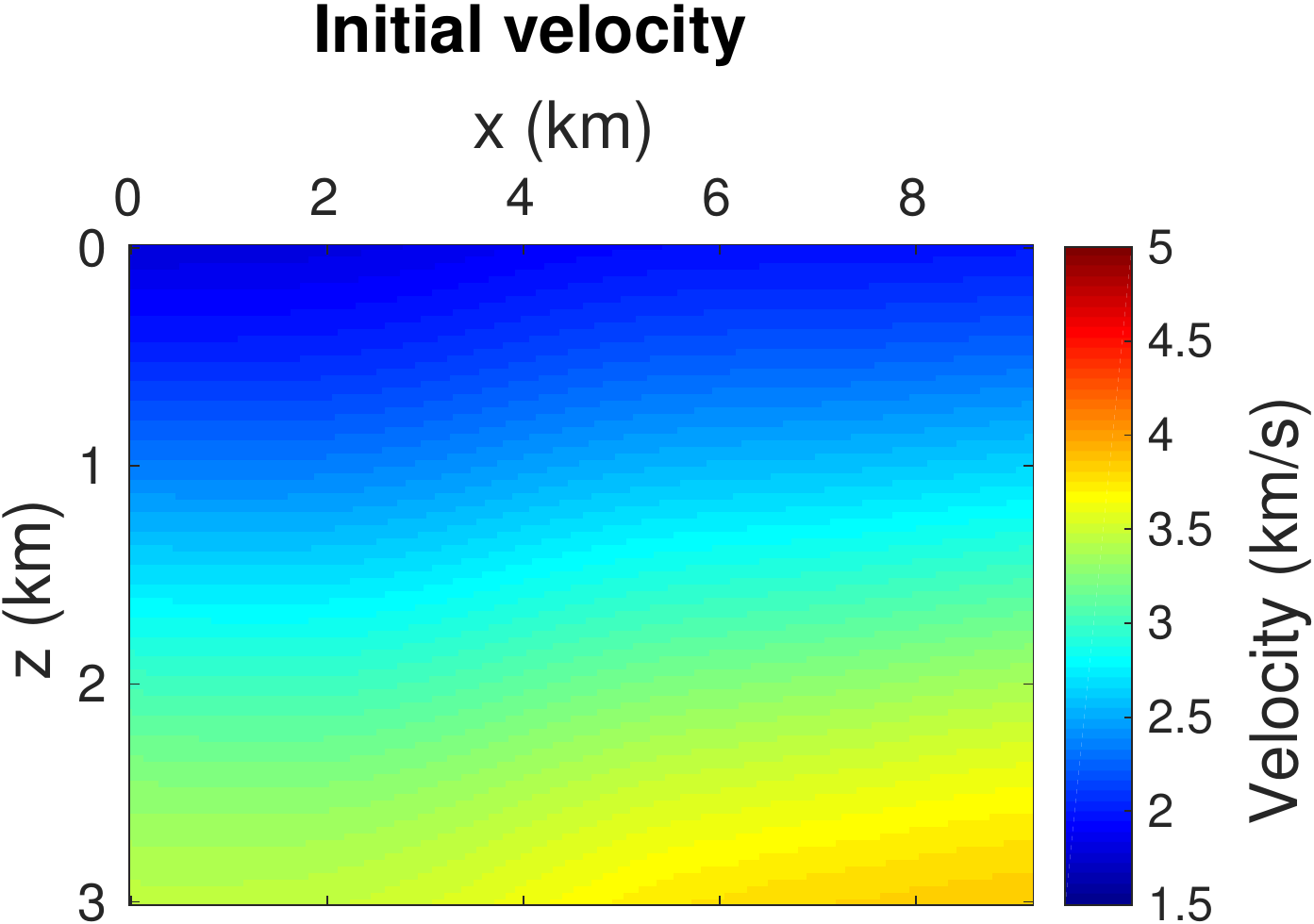}\label{fig:marm2_v0}}
    \subfloat[]{\includegraphics[width=0.25\textwidth]{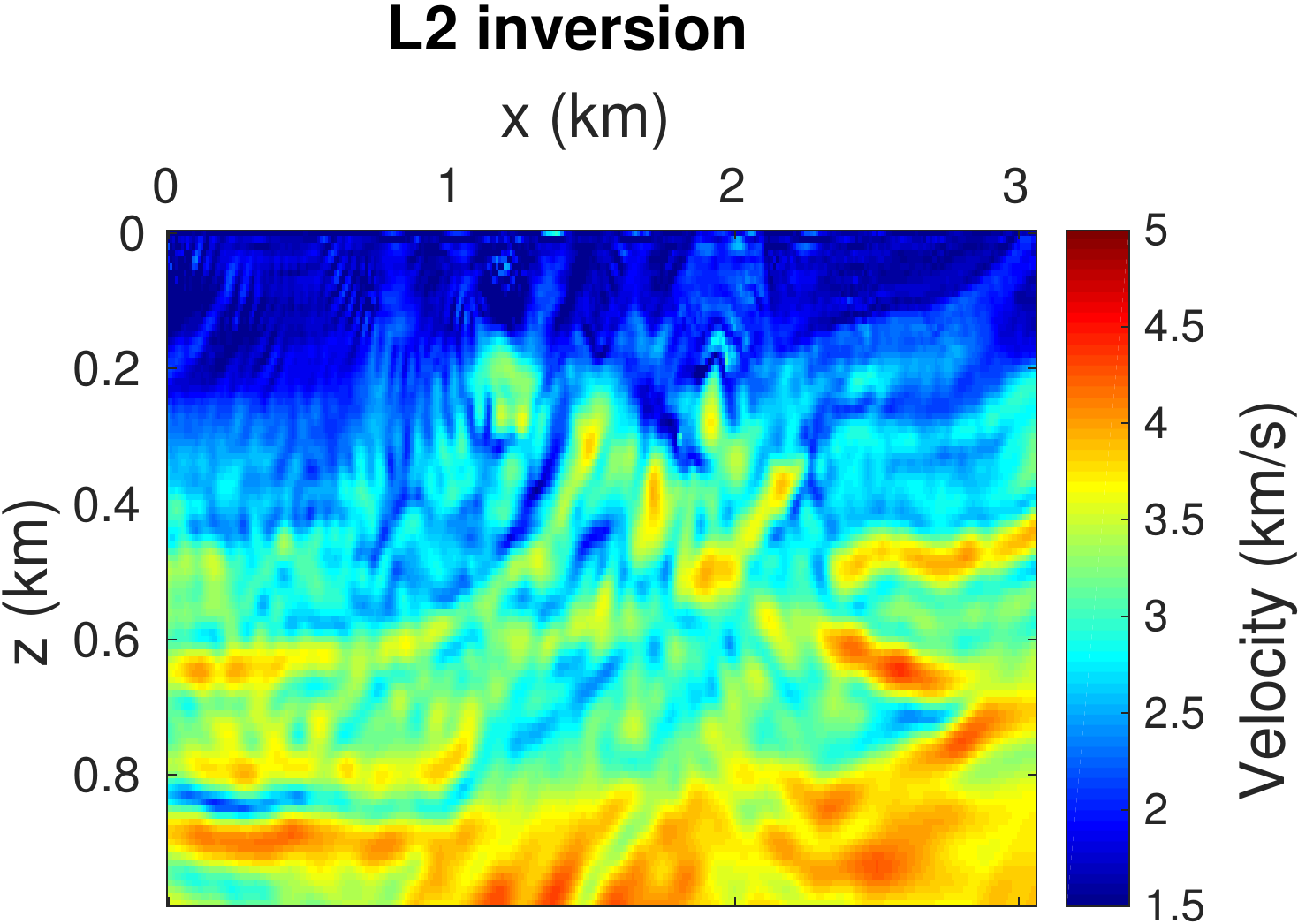}\label{fig:marm_L2}}
  \subfloat[]{\includegraphics[width=0.25\textwidth]{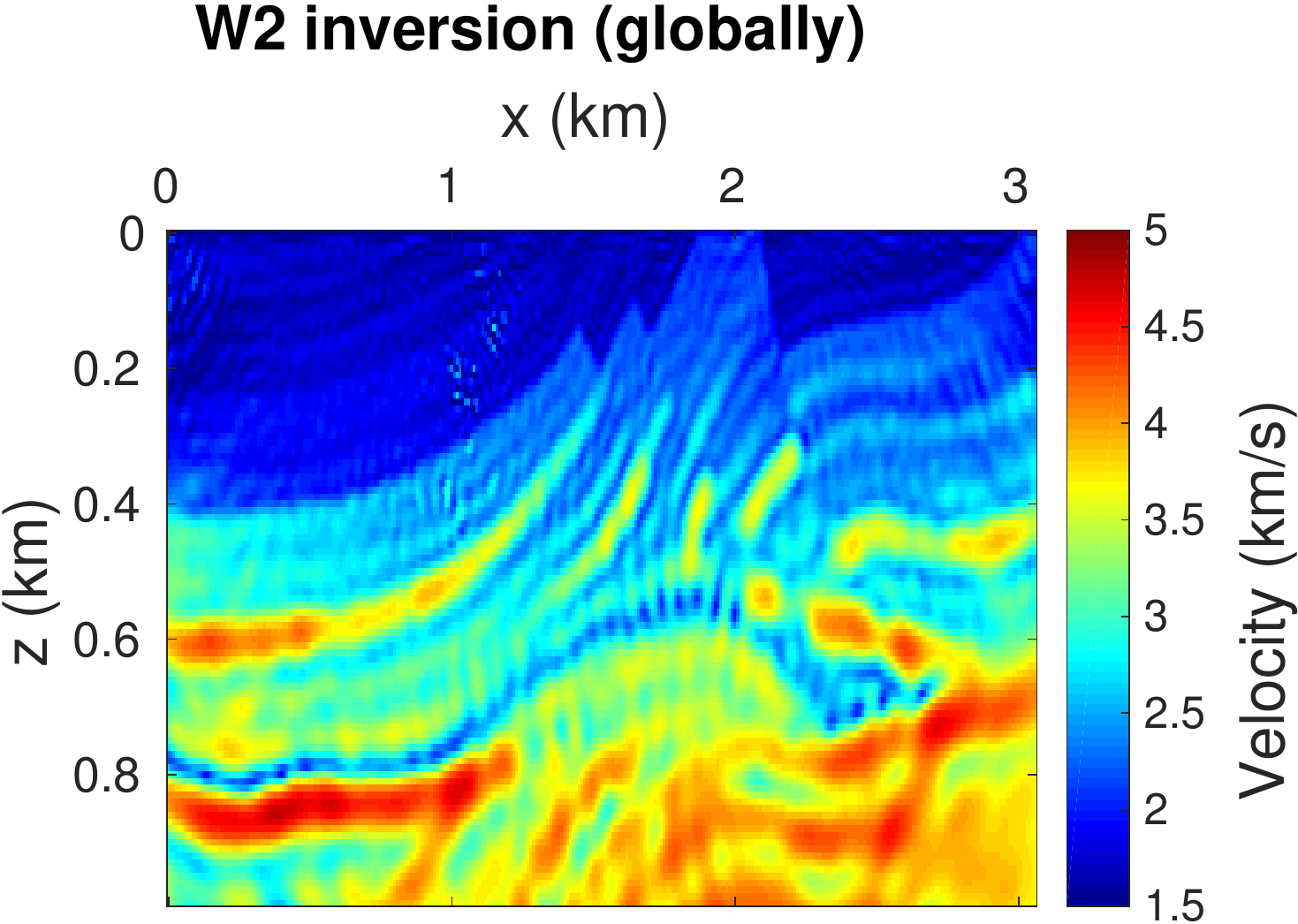}\label{fig:marm_w2_2D}}
  \caption{Scaled Marmousi model: (a)~true velocity; (b)~inital velocity; (c)~$L^2$ inversion result and (d)~$W_2$ inversion result based on the \MA solver.}
  \label{fig:marm2_true,marm2_v0}
\end{figure}

\subsection{Numerical Results}~ \label{sec:MA_Marm}
In this section, we invert a scaled Marmousi model using~\eqref{eq:W22D} as the objective function. The Marmousi model is based on a profile through the North Quenguela Trough in the Cuanza Basin in Angola~\cite{versteeg1994marmousi}. It has been a standard benchmark model for exploration geophysics since the 1980s. The $W_2$ distance in this test is computed based on the solution to the \MA equation~\eqref{eq:MAA}. Figure~\ref{fig:marm2_true} is the P-wave velocity of the true Marmousi model, but in this experiment, a scaled model is used (1 km in depth and 3 km in width). The inversion starts from an initial model (a scaled version of Figure~\ref{fig:marm2_v0}) that is the true velocity smoothed by a Gaussian filter. It no longer has all the fine structures of the true model. We place 11 evenly spaced sources on top at 50 m depth, and 307 receivers on top at the same depth with a 10 m fixed acquisition. The discretization of the forward wave equation is 10 m in the $x$ and $z$ directions and 1 ms in time. The source is a Ricker wavelet, which is the second derivative of the Gaussian function with a peak frequency of 15 Hz, and a bandpass filter applied to remove the frequency components from 0 to 2 Hz. 

We use L-BFGS, a quasi-Newton method as the optimization algorithm~\cite{liu1989limited}. We keep ten vectors in memory for the optimization scheme. Inversions are terminated after 200 iterations. Figure~\ref{fig:marm_L2} shows the final inversion result using the traditional least-squares method whose result demonstrates spurious high-frequency artifacts due to a point-by-point comparison of amplitude. The final inversion result using the $W_2$ distance is shown in Figure~\ref{fig:marm_w2_2D}. We solve the \MA equation~\eqref{eq:MAA} numerically in each iteration to compute the $W_2$ distance. Figure~\ref{fig:marm_w2_2D} shows that $W_2$-based inversion mitigates the local minima issue suffered by the traditional $L^2$ norm. It is a result of the ideal convexity of the $W_2$ distance which we will further discuss in Section~\ref{sec:challenge1}.

In numerical experiments carried out in~\cite{FroeseTransport}, the total computational complexity required to solve the Monge-Amp\`ere equation varied from $\bO(N)$ to $\bO(N^{1.3})$ where $N$ was the total number of grid points. We observed the same complexity when solving~\eqref{eq:MAA}, where $f$ and $g$ are normalized wavefields. Figure~\ref{fig:MA-1iter-time} shows the CPU time of solving one single \MA equation to obtain the optimal transport map between wavefields of different grid size in time ($nt$) and space ($nx$). The slope of the log-log plot shows that the complexity is about $\bO(N)$. Figure~\ref{fig:MA-20iter-time} presents the CPU time of running 20 iterations of inversion in both $L^2$ norm and the $W_2$ distance. Although the \MA-based $W_2$ inversion takes 3 to 4 times longer than the $L^2$-based inversion, the overall complexity is $\bO(N)$ for both objective functions and the $W_2$-based FWI helps mitigate the local minima that the least-squares method suffers from. 

\begin{figure}
\centering
  \subfloat[]{\includegraphics[width=0.45\textwidth]{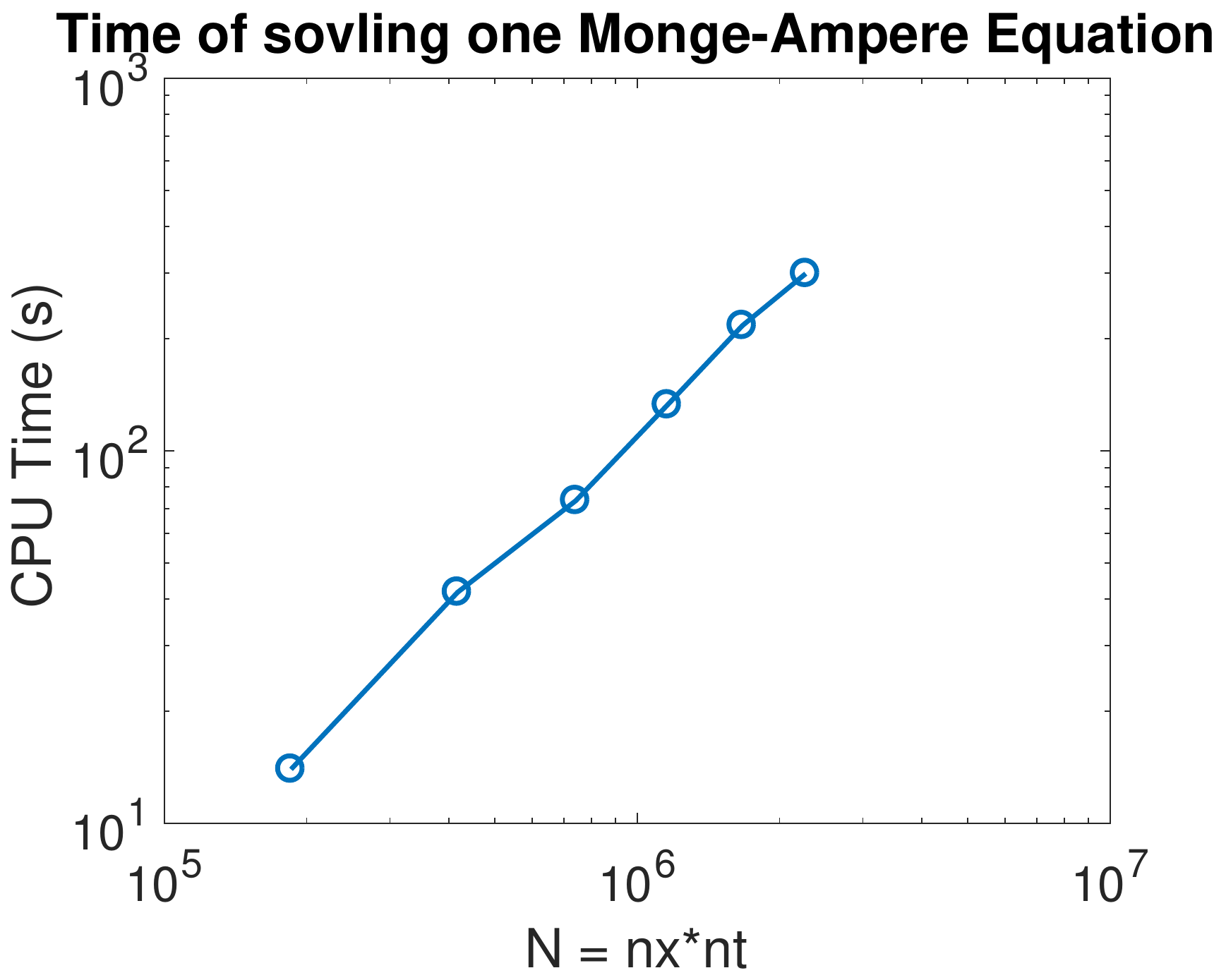}\label{fig:MA-1iter-time}}
  \subfloat[]{\includegraphics[width=0.45\textwidth]{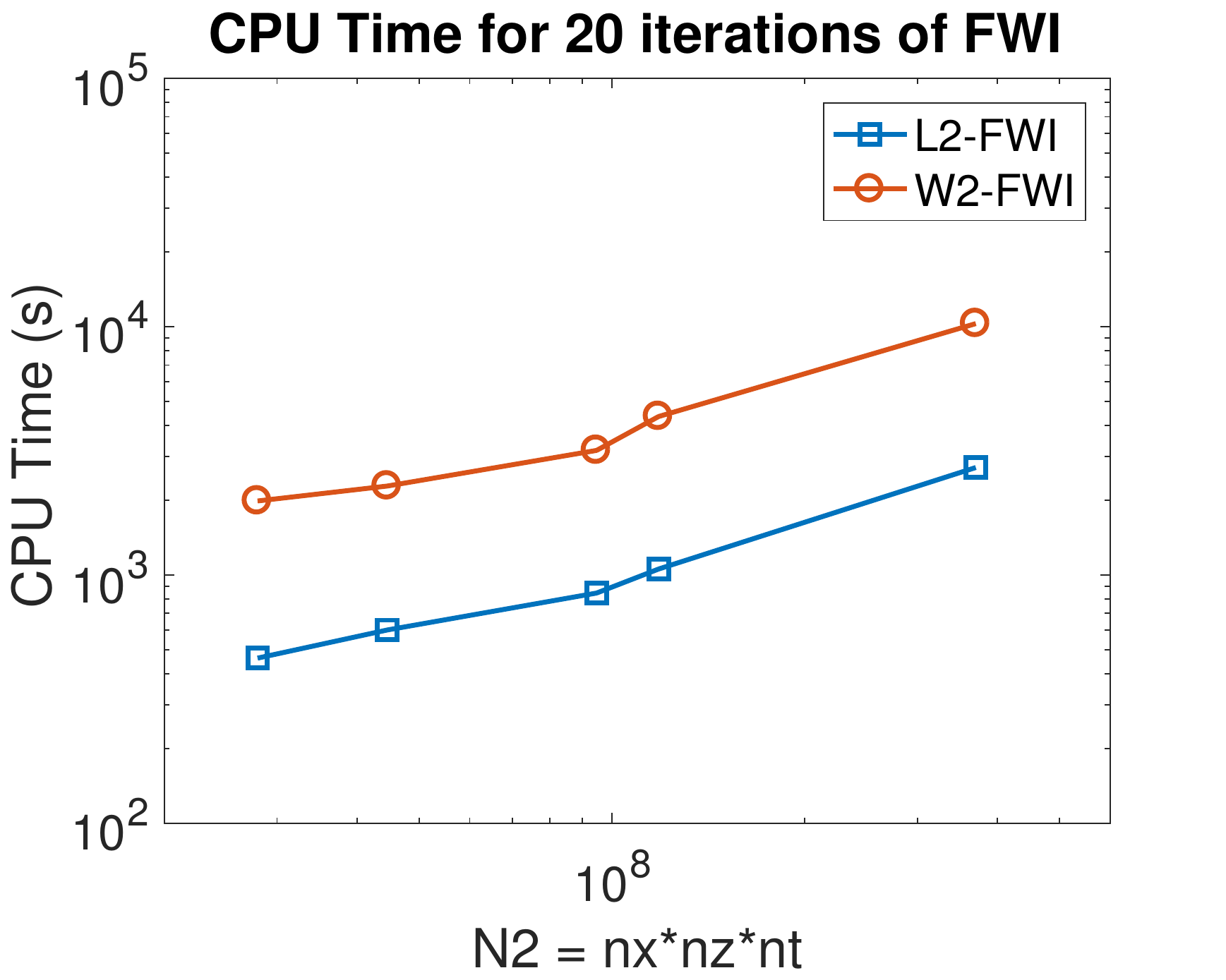}\label{fig:MA-20iter-time}}
  \caption{(a) CPU time of solving one \MA equation; (b) CPU time of running 20 iterations of inversions.}
  \label{fig:MA-time}
\end{figure}

\section{Challenge One: Inversion with Local Minima} \label{sec:challenge1}
Two primary concerns in full-waveform inversion (FWI) are the well-posedness of the underlying model recovery problem and the suitability of the misfit between the observed data $g$ and the synthetic data $f$ for minimization.
Several desirable properties of the quadratic Wasserstein metric ($W_2$) relating to convexity rigorously demonstrate the advantages of this transport-based idea in seismic inverse problems.
\subsection{Convexity of $W_2$} \label{sec:convexity}
As we demonstrated in~\cite{engquist2016optimal}, the squared Wasserstein metric has several properties that make it attractive as a choice for misfit function.  One highly desirable feature is its convexity with respect to several parameterizations that occur naturally in seismic waveform inversion~\cite{yang2017application}. For example, variations in the  wave velocity lead to  simulated $f$ that are derived from shifts,
\bq\label{eq:shift}
f(x;s) = g(x+s\eta), \quad \eta \in \R^n,
\eq
 or dilations,
\bq\label{eq:dilation}
f(x;A) = g(Ax), \quad A^T = A, \, A > 0,
\eq
 applied to the observation $g$.
Variations in the strength of a reflecting surface or the focusing of seismic waves can also lead to local rescalings of the form
\bq\label{eq:rescaleLocal}
f(x;\beta) = \begin{cases}  \beta g(x), & x \in E\\ g(x), & x \in \R^n\backslash E.\end{cases}
\eq

Proving the convexity of $W_2^2$ follows from the interpretation of the misfit as a transportation cost, with the underlying transportation cost exhibiting a great deal of structure.  In particular, the cyclical monotonicity of the transport map $T(x)$ leads readily to estimates of 
\[ W_2^2(f({\lambda m_1 + (1-\lambda)m_2}), g), \quad 0 < \lambda < 1,  \]
which in turn yields the desired convexity results.  This was studied in detail in~\cite{engquist2016optimal}, where the following theorem was proved.

\begin{theorem}[Convexity of squared Wasserstein metric~{\cite{engquist2016optimal}}]\label{thm:convexity}
The squared Wasserstein metric $W_2^2(f(m),g)$ is convex with respect to the model parameters $m$ corresponding to a shift~$s$ in~\eqref{eq:shift}, the eigenvalues of a dilation matrix~$A$ in~\eqref{eq:dilation}, or the local rescaling parameter~$\beta$ in~\eqref{eq:rescaleLocal}.
\end{theorem} 

\subsection{Joint Convexity in Signal Translation and Dilation}
The translation and dilation in the wavefields are direct effects of variations in the velocity $v$, as can be seen from the D'Alembert's formula that solves the 1D wave equation~\cite{engquist2016optimal}. In particular, we will reformulate Theorem~\ref{thm:convexity} as a joint convexity of $W_2$ with respect to both signal translation and dilation and prove it in a more general setting. In practice, the perturbation of model parameters will cause both signal translation and dilation simultaneously, and the convexity with respect to both changes is an ideal property for gradient-based optimization.

We regard the normalized synthetic data and observed data as two probability densities $f = d\mu$ and $g =d\nu$ on $\mathbb{R}^d$ when using optimal transport for seismic inversion. Normalization is the first and the most important step in seismic inversions that turns oscillatory seismic signals into probability measures. It is one prerequisite of optimal transport~\cite{qiu2017full,metivier2018graph,Survey2, yang2017application}.
Since seismic signals are partial measurements of the boundary value in~\eqref{eq:FWD}, they are compactly supported in $\mathbb{R}^d$ and bounded from above. Hence, it is natural to assume $\int_{\mathbb{R}^d}\int_{\mathbb{R}^d} |x-y|^2 f(x)g(y)dxdy < +\infty$.  
The $W_2$ distance then has convexity in translation and dilation as the following Theorem~\ref{thm: biconvex} states. 

The proof of Theorem~\ref{thm: biconvex} relies on the concept of cyclical monotonicity, a powerful tool characterizing the general geometrical structure of the optimal tranposrt plan in the Kantorovich problem:
\bq~\label{eq:static2}
\inf_{\pi} I[\pi] = \bigg\{ \int_{X \times Y} c(x,y) d\pi\ |\ \pi \geq 0\ \text{and}\ \pi \in \Pi(\mu, \nu) \bigg\} ,
\eq
where $\Pi (\mu, \nu) =\{ \pi \in \mathcal{P}(X\times Y)\ |\ (P_X)_\# \pi = \mu, (P_Y)_\# \pi = \nu  \}$. Here $(P_X)$ and $(P_Y)$ denote the two projections, and $(P_X)_\# \pi$ and $(P_Y)_\# \pi $ are two measures obtained by pushing forward $\pi$ with these two projections. Different from~\eqref{eq:static}, Kantorovich relaxed the constraints~\cite{kantorovich1960mathematical}. Instead of searching for a map $T$, the transference plan $\pi$ is considered, which is also a measure supported by the product space $X\times Y$. Brenier~\cite{brenier1991polar} has proved that the optimal map in ~\eqref{eq:static} coincide with the optimal plan in~\eqref{eq:static2} in the sense that $\pi =(Id \times T)_\# \mu$ if $\mu$ gives no mass to small sets.

It has been proved that optimal plans for any cost function have cyclically monotone support~\cite{villani2008optimal}. In addition, the concept can be used to formulate a \textit{sufficient} condition of the optimal transport map~\cite{villani2008optimal,brenier1991polar,KnottSmith} under certain mild assumptions\cite{ambrosio2013user}. 
\begin{definition}[Cyclical monotonicity]~\label{def: c-cyc}
We say that $\pi \subset X\times Y$ is cyclically monotone if for any $m\in\mathbb{N}^+$, $(x_i,y_i) \in \pi$, $1\leq i \leq m$, implies
\bq\label{eq:cyclical}
\sum_{i=1}^{m} c(x_i, y_i) \leq  \sum_{i=1}^{m} c(x_i, y_{i-1}),\ (x_0,y_0) \equiv (x_m,y_m).
\eq 
\end{definition}

Next, we will prove a stronger convexity result than Theorem~\ref{thm:convexity}. The following new theorem states a joint convexity in multiple variables with respect to both translation and dilation changes in the data. 
Assume that $s_k\in \mathbb{R},\ k = 1,\dots,d$ is a set of translation parameters and $\{e_k\}_{k=1}^d$ is the standard basis of the Euclidean space $\mathbb{R}^d$.  
$A=\diag (1/\lambda_1,\dots,1/\lambda_d)$ is a dilation matrix where $\lambda_k\in \mathbb{R}^+, k= 1,\dots,d$. 
We define $f_{\Theta}$ as jointly the translation and dilation transformation of function $g$ such that
\bq \label{eq:f_theta}
f_{\Theta}(x)=\det(A)g(A(x-\sum_{k=1}^d s_k e_k)).
\eq
We will prove the convexity in terms of the multivariable $\Theta = \{s_1, \dots, s_d, \lambda_1, \dots, \lambda_d\}$.  

\begin{theorem}[Convexity of $W_2$ in translation and dilation]
\label{thm: biconvex}
Assume that $f_{\Theta}$ is jointly the translation and dilation transformation of $g$ defined as Equation~\eqref{eq:f_theta}. The optimal map between $f_{\Theta}(x)$ and $g(y)$ is $y = T_{\Theta}(x)$ where $\langle T_{\Theta}(x), e_k \rangle = \frac{1}{\lambda_k} (\langle x , e_k \rangle - s_k), k= 1,\dots,d$. Moreover, $I(\Theta) = W_2^2(f_{\Theta}(x),g)$ is a convex function of the multivariable $\Theta$.
\end{theorem}

\begin{proof}[Proof of Theorem~\ref{thm: biconvex}]
First, we will justify that $y =T_{\Theta}(x)$ is a measure-preserving map according to Definition~\ref{def:mass_preserve}. It is sufficient to check that $T_{\Theta}$ satisfies Equation~\ref{eq:mass_preserve2}:
\bq
f_{\Theta}(x) =\det(A)g(A(x-\sum_{k=1}^d s_k e_k)) =\det(A)g(T_{\Theta}(x)) =\det(\nabla T_{\Theta}(x)) g(T_{\Theta}(x)).
\eq

Next, we will show that the new joint measure $\pi_{\Theta}=(Id \times T_{\Theta})\# \mu_{\Theta}$ is cyclically monotone. This is based on two lemmas from \cite[p80]{Villani} and the fundamental theorem of optimal transport in~\cite[p10]{ambrosio2013user} on the equivalence of optimality and cyclical monotonicity under the condition
\bq
\int_{\mathbb{R}^d}\int_{\mathbb{R}^d} |x-y|^2 f(x)g(y)dxdy < +\infty.
\eq
For $c(x,y) = |x-y|^2$, the cyclical monotonicity in Definition~\ref{def: c-cyc} is equivalent to
\bq
\sum_{i=1}^{m}x_i \cdot (T(x_i)- T(x_{i-1})) \geq  0,
\eq 
for any given set of $\{x_i\}_{i=1}^m \subset X$. For $T_{\Theta}(x)$, we have
\begin{align}
   \sum_{i=1}^{m}x_i \cdot (T_{\Theta}(x_i)-T_{\Theta}(x_{i-1})) & =  
   \sum_{i=1}^{m}  \sum_{k=1}^d  \langle x_i,e_k \rangle \cdot  (\langle T_{\Theta}(x_{i}) ,e_k \rangle - \langle T_{\Theta}(x_{i-1}) ,e_k \rangle)   \label{eq:c_new}  \\ 
&=   \sum_{i=1}^{m}  \sum_{k=1}^d   \frac{1}{\lambda_k} \langle x_i,e_k \rangle \cdot  (\langle x_i,e_k \rangle - \langle x_{i-1},e_k \rangle ) \\
&=  \frac{1}{2} \sum_{k=1}^d  \frac{1}{\lambda_k} \sum_{i=1}^{m}   |\langle x_i,e_k \rangle - \langle x_{i-1},e_k \rangle |^2 \geq 0
 \label{eq:c_old}
    \end{align}

The last inequality~\eqref{eq:c_old} states the support of the transport plan $\pi_{\Theta}=(Id \times T_{\Theta})\# \mu_{\Theta}$ is  cyclically monotone. By the uniqueness of monotone measure-preserving optimal maps between two distributions~\cite{mccann1995existence}, we assert that $T_{\Theta}(x)$ is the optimal map between $f_{\Theta}$ and $g$. The squared $W_2$ distance between $f_{\Theta}$ and $g$ is
 \begin{align}
 I(\Theta)   =  W_2^2(f_{\Theta},g) & = \int_{X} |x-T_{\Theta} (x)|^2 f_{\Theta}(x) dx\\
     & = \int_{Y} \sum_{k=1}^d |(\lambda_k-1) \langle y,e_k\rangle + s_k |^2 d\nu \\
 	& =   \sum_{k=1}^d a_k (\lambda_k-1)^2    + 2 \sum_{k=1}^d b_k s_k (\lambda_k-1)  + \sum_{k=1}^d s_k^2,  \label{eqn:convex}
\end{align}
where $a_k =\int_{Y} |\langle y,e_k\rangle|^2 d\nu $ and $b_k =\int_Y \langle y,e_k\rangle d\nu$.

Equation~$I(\Theta)$ is a quadratic function. We can compute its Hessian matrix
\bq
H(\Theta)  =  \begin{pmatrix}
I_{s_1s_1} &\dots & I_{s_1s_d} & I_{s_1\lambda_1}&\dots & I_{s_1\lambda_d}  \\
\vdots   &\ddots & \vdots	   &\vdots					&\ddots  &  \vdots  \\
I_{s_d s_1} &\dots & I_{s_ds_d} & I_{s_d\lambda_1}&\dots & I_{s_d\lambda_d}  \\
I_{\lambda_1 s_1} &\dots & I_{\lambda_1 s_d} & I_{\lambda_1 \lambda_1}&\dots & I_{\lambda_1 \lambda_d}  \\
\vdots   &\ddots & \vdots	   &\vdots					&\ddots  &  \vdots  \\       
I_{\lambda_d s_1} &\dots & I_{\lambda_d s_d} & I_{\lambda_d \lambda_1}&\dots & I_{\lambda_d \lambda_d}  \\
\end{pmatrix}  = 
         \begin{pmatrix}
1            &\dots    & 0              & b_1                       &\dots    & 0   \\
\vdots   &\ddots & \vdots	   &\vdots					&\ddots  &  \vdots  \\
0            &\dots   &  1                &0                          &\dots    & b_d \\
b_1          &\dots  & 0             & a_1                         &\dots    & 0   \\
\vdots   &\ddots & \vdots	   &\vdots					&\ddots  &  \vdots  \\       
0             &\dots  &b_d           & 0                           &\dots & a_d  \\
        \end{pmatrix}. 
\eq 
$H(\Theta)$ is a symmetric matrix with eigenvalues:
\bq
\frac{1}{2} \left( a_k +1\pm \sqrt{a_k^2 - 2a_k + 4b_k^2 + 1} \right),\ k=1,\dots, d.
\eq
Since $a_k =\int_{Y} |\langle y,e_k\rangle|^2 d\nu \geq 0$ by definition, and 
 \begin{align}
&\quad  \frac{1}{4}  \left(a_k +1\right)^2 - \left(\frac{1}{2} \sqrt{a_k^2 - 2a_k + 4b_k^2 + 1} \right)^2 \\ & =  \frac{1}{4} \left((a_k+1)^2 - \left(a_k^2 - 2a_k + 4b_k^2 +1 \right) \right) = a_k - b_k^2  \\
&=  \int_{Y} |\langle y,e_k\rangle|^2 d\nu \int_{Y} 1^2 d\nu - \left(\int_Y \langle y,e_k\rangle d\nu \right)^2 \\
& \geq 0 \quad \text{by Cauchy-Schwarz inequality},
\end{align}
all the eigenvalues of $H(\Theta)$ are nonnegative. Therefore, the Hessian matrix of $I(\Theta)$ is symmetric positive semidefinite, which completes our proof that $W_2^2(f_{\Theta},g)$ is a convex function with respect to $\Theta = \{s_1, \dots, s_d, \lambda_1, \dots, \lambda_d\}$, the combination of translation and dilation variables.
\end{proof}

\begin{figure}
\centering
   \subfloat[$W^2_2(f_{[\mu,\sigma]},g)$]{\includegraphics[height=0.25\textwidth]{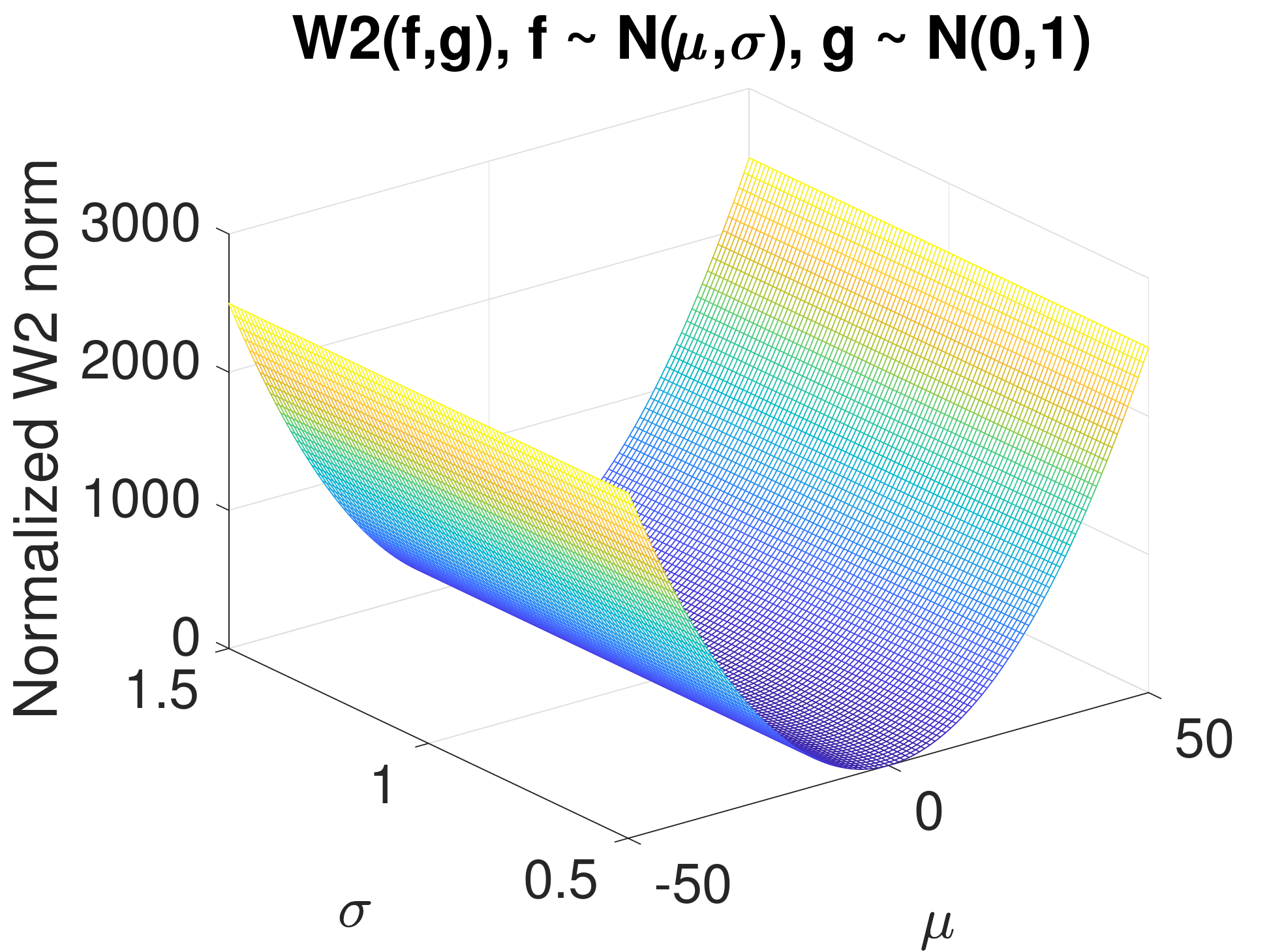}\label{fig:Gaussian_W2_2var}}
      \subfloat[$W^2_2(f_{[\mu,1]},g)$]{\includegraphics[height=0.25\textwidth]{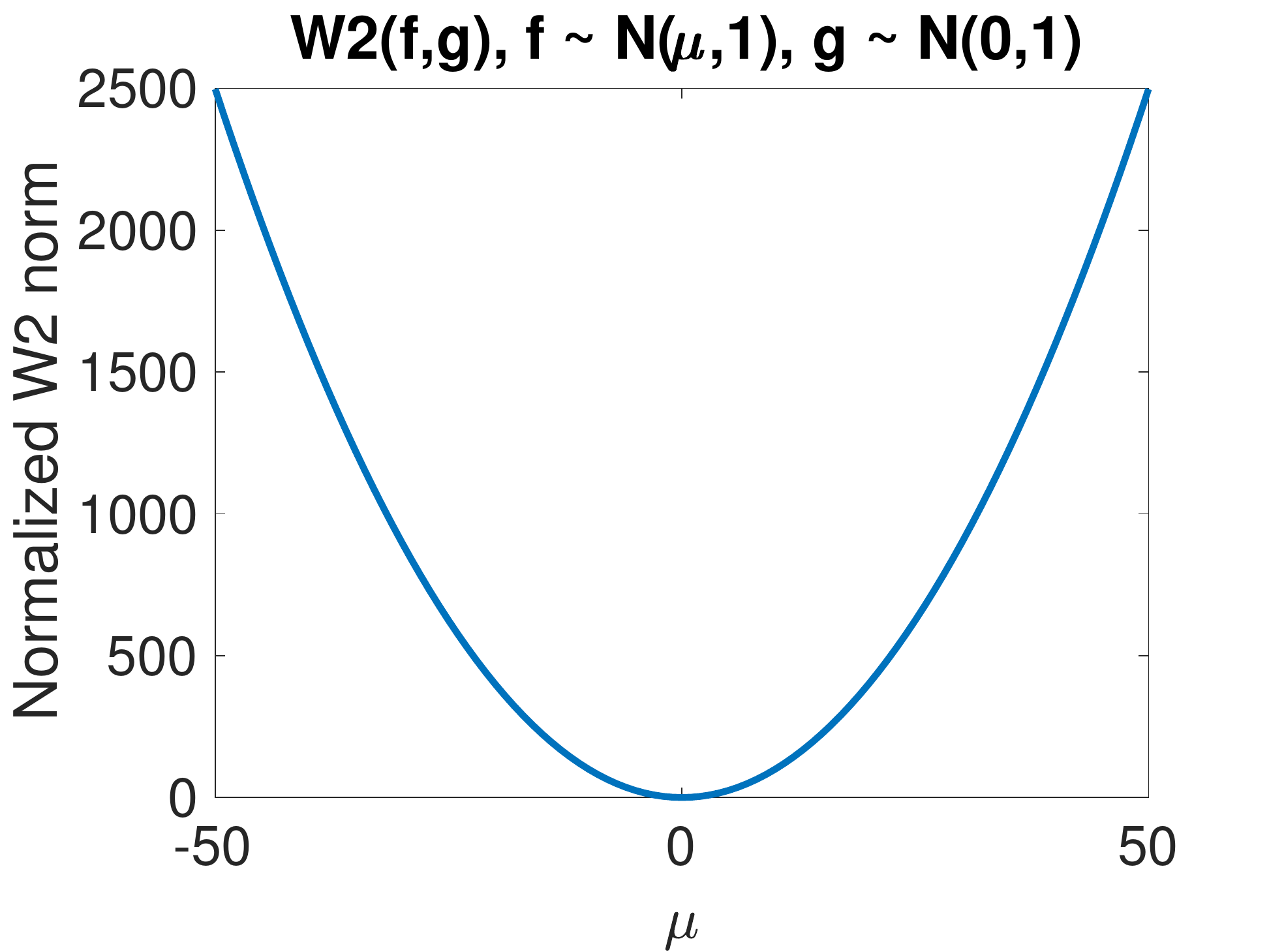}\label{fig:Gaussian_W2_mu}}
   \subfloat[$W^2_2(f_{[0,\sigma]},g)$]{\includegraphics[height=0.25\textwidth]{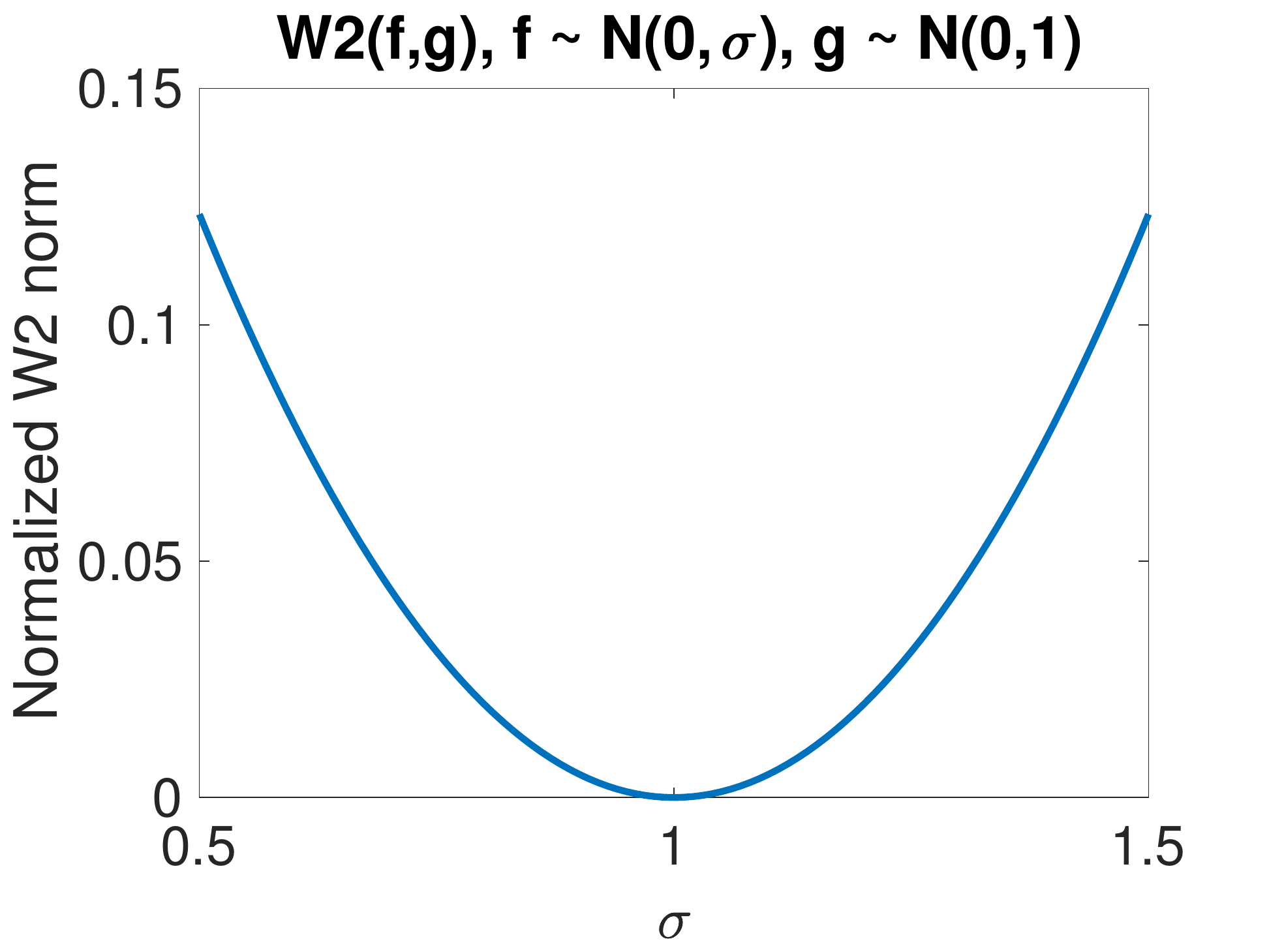}\label{fig:Gaussian_W2_sigma}}\\
   \caption{$W^2_2(f_{[\mu,\sigma]},g)$ where $f_{[\mu,\sigma]}$ and $g$ are density functions of $\mathcal{N}(\mu,\sigma)$ and $\mathcal{N}(0,1)$.}~\label{fig:Gaussian_W2}
\end{figure}

In Figure~\ref{fig:Gaussian_W2}, we illustrate the joint convexity of the $W_2$ distance with respect to both translation and dilation by comparing density functions of normal distributions. We set $f_{[\mu,\sigma]}$ as the density function of the 1D normal distribution $\mathcal{N}(\mu,\sigma)$. We treat $g$ as the reference signal which is the density function of $\mathcal{N}(0,1)$. Figure~\ref{fig:Gaussian_W2} is the optimization landscape of $W^2_2(f_{[\mu,\sigma]},g)$ as a multivariable function. The plots show that $W^2_2(f_{[\mu,\sigma]},g)$ is globally convex with respect to both translation and dilation.

\begin{remark}
Although there seems to be no uniqueness in Figure~\ref{fig:Gaussian_W2_2var}, the cross-sections in Figure~\ref{fig:Gaussian_W2_mu} and Figure~\ref{fig:Gaussian_W2_sigma} illustrate that $W_2$ is still convex around the reference values with a unique minimum. There is, of course, a scale difference in the two variables since $W_2$ is more sensitive to signal translation than dilation as a result of Definition~\ref{def:OT}. However, the amplitude difference in the form of dilation is always considered even with no phase mismatches as seen in Figure~\ref{fig:Gaussian_W2_sigma} when the translation parameter is fixed.
\end{remark}


%

\subsection{Numerical Example}
Among different types of recorded seismic data, diving waves (wavefronts continuous refracted upwards through the earth due to the presence of a vertical velocity gradient) have been the driving force behind the success of full-waveform inversion (FWI) in achieving high-resolution model recovery in shallow-water environments. Inaccurate kinematic features in the initial model often lead to shifts in seismic signals, especially common in diving waves. Figure~\ref{fig:2_ricker} illustrates one representative example. As discussed in Section~\ref{sec:fwi_obj}, the oscillatory and periodic nature of waveforms leads to a primary challenge in inversion known as \textit{cycle-skipping issues} in exploration geophysics. Mathematically, it is a local minima problem in this PDE-constrained optimization.

Next, we redo the \textit{full} Marmousi model based on the 1D technique, computing the Wasserstein distance together with the sign-sensitive normalization in~\eqref{eq:mixmix}. For each receiver, we first normalize the seismic signals and then solve the optimal transport problem in 1D. This is the full Marmousi model while the example presented in Section~\ref{sec:MA_Marm} is a scaled version due to limitations of the \MA solver.  

Figure~\ref{fig:marm2_L2_grad} and Figure~\ref{fig:marm2_W2_grad} are the gradients in the first iteration of the inversion using two objective functions respectively. Starting from a highly smooth initial model shown in Figure~\ref{fig:marm2_v0}, $W_2$ inversion already focuses on the ``peak'' of the Marmousi model in the very first iteration. The darker area in Figure~\ref{fig:marm2_W2_grad} matches many features of the true Marmousi model in~Figure~\ref{fig:marm2_true}. After 300 L-BFGS iterations, final $L^2$ and $W_2$ inversion results are shown in Figure~\ref{fig:marm2_L2} and~\ref{fig:marm2_w2_1D}. The data residuals before and after inversion (Figure~\ref{fig:MARM2_res0} and~\ref{fig:MARM2_W2res}) illustrate that the $W_2$ based method fits the data quite well. The convergence histories in Figure~\ref{fig:marm2_conv} demonstrate that $W_2$ not only rapidly reduces the relative misfit to 10\% in just 20 iterations but also avoids the local minimum that $L^2$ based inversion easily converges to.

Compared with the relatively noisy inversion result based on the \MA equation (Figure~\ref{fig:marm_w2_2D}), the trace-by-trace 1D technique, on the other hand, can make use of exact formulas for 1D optimal transportation, which allows for accurate computations even when the data is highly non-smooth.

\begin{figure}
\centering
  \subfloat[$L^2$ first gradient]{\includegraphics[width=0.23\textwidth]{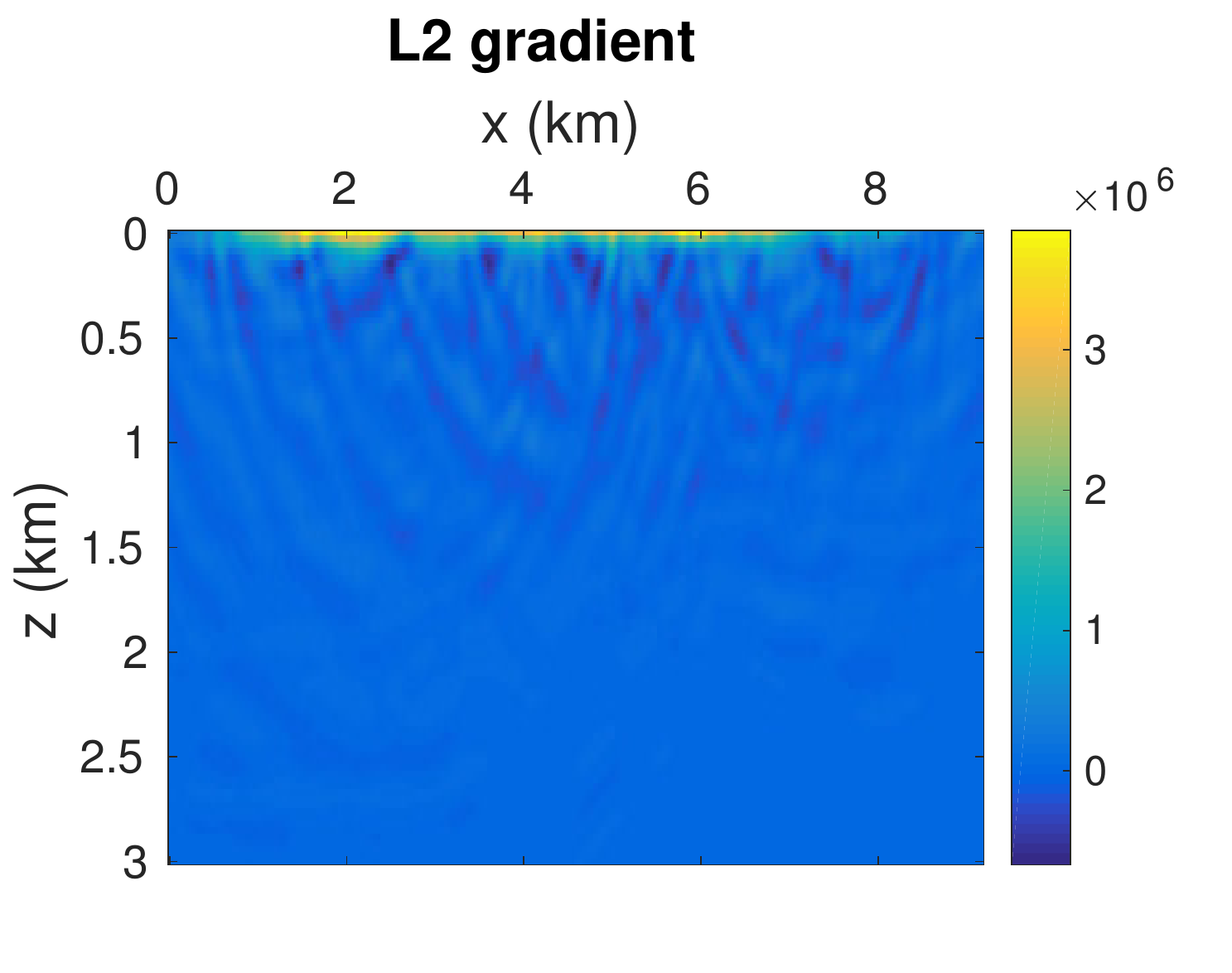}\label{fig:marm2_L2_grad}}
  \subfloat[$W_2$ first gradient]{\includegraphics[width=0.23\textwidth]{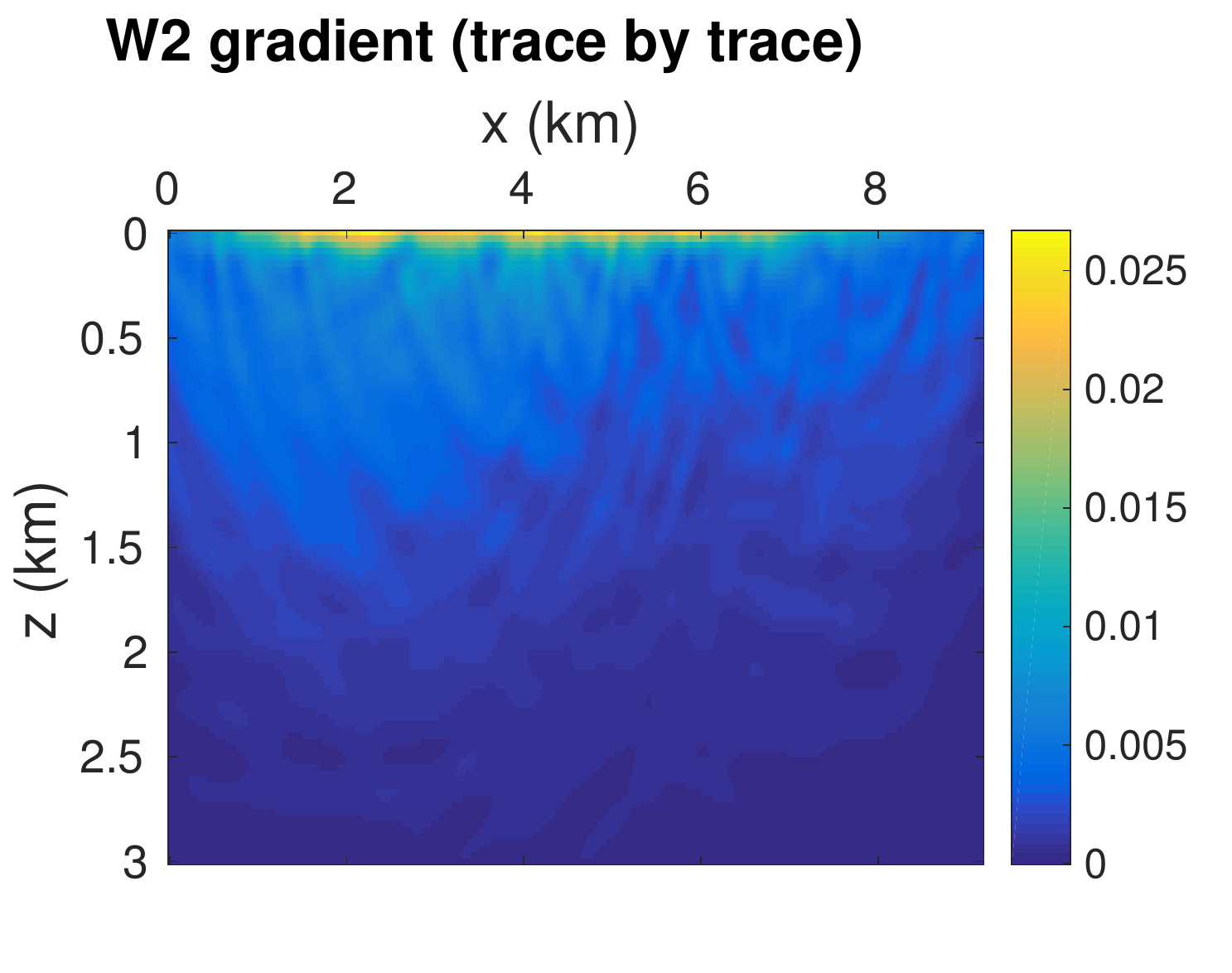}\label{fig:marm2_W2_grad}}
    \subfloat[$W_2$ initial residual]{\includegraphics[width=0.23\textwidth]{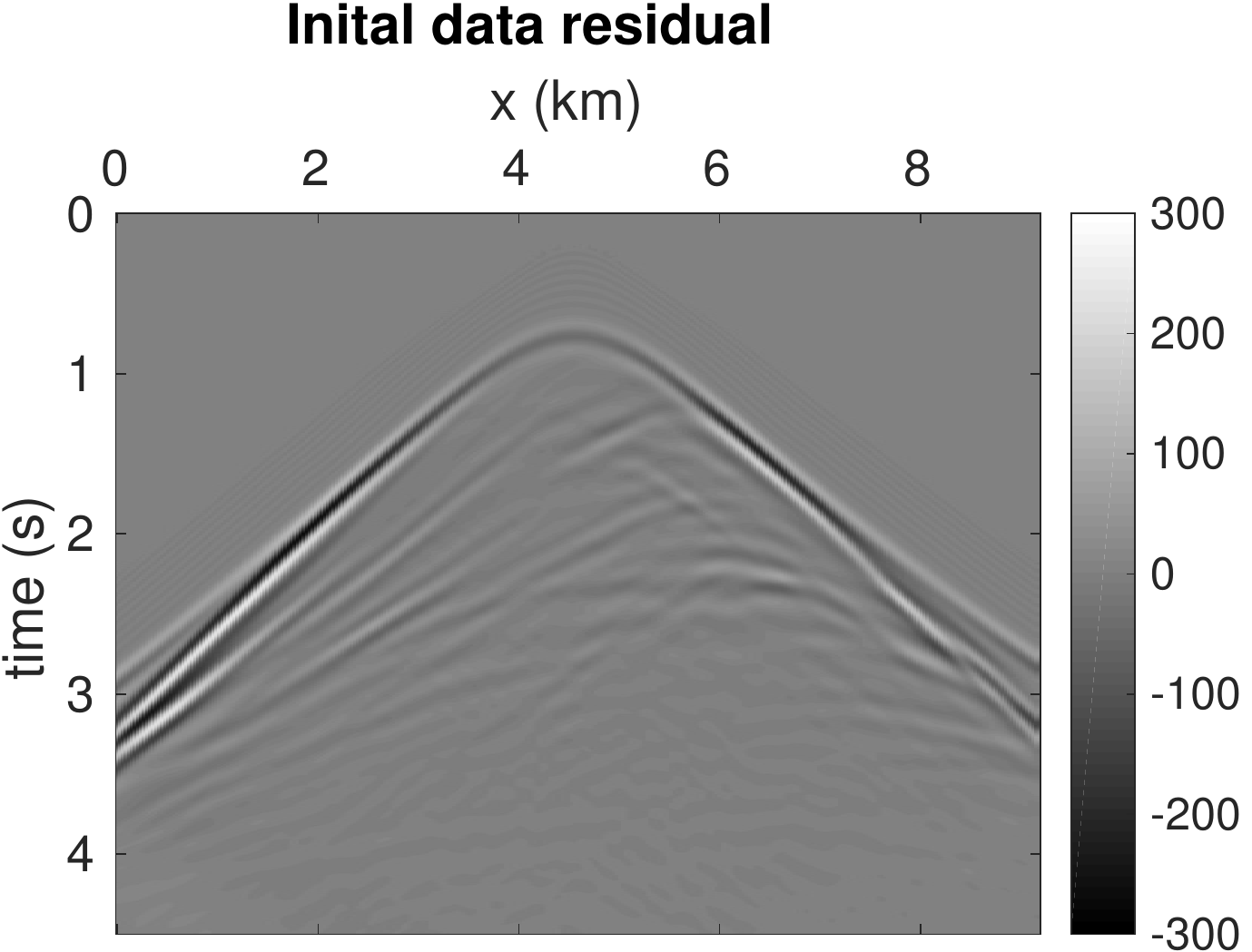}\label{fig:MARM2_res0}}
  \subfloat[$W_2$ final residual]{\includegraphics[width=0.23\textwidth]{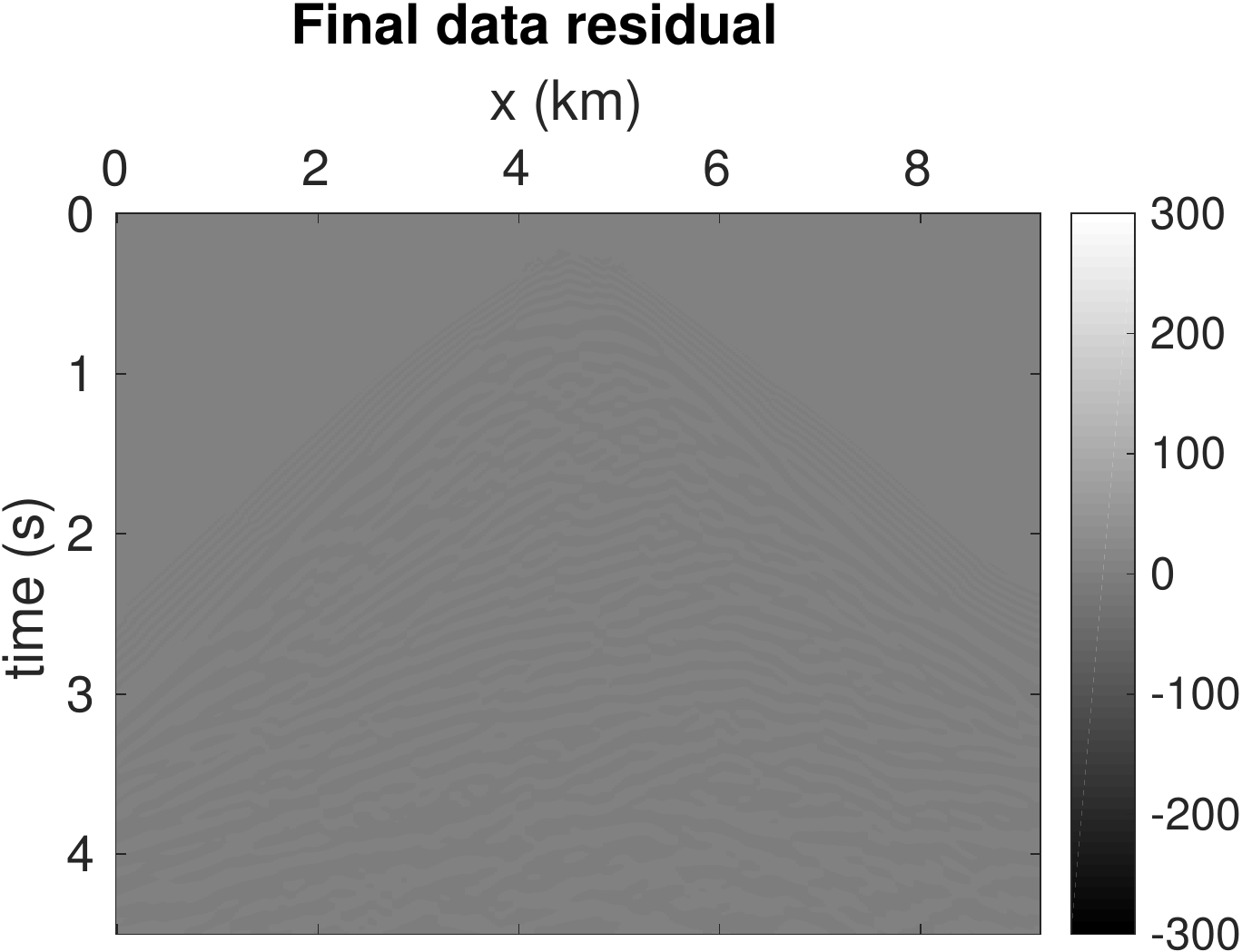}\label{fig:MARM2_W2res}}
  \caption{Full Marmousi Model: the gradient in the first iteration of (a)~$L^2$ and (b)~trace-by-trace $W_2$; The difference between the observable true data  and the synthetic data using (c)~the initial model shown in Figure~\ref{fig:marm2_v0}, and (d)~the final $W_2$ inversion result shown in Figure~\ref{fig:marm2_w2_1D}.}
  \label{fig:marm2_grad}
\end{figure}

\begin{figure}
\centering
  \subfloat[$L^2$ final result]{\includegraphics[width=0.24\textwidth]{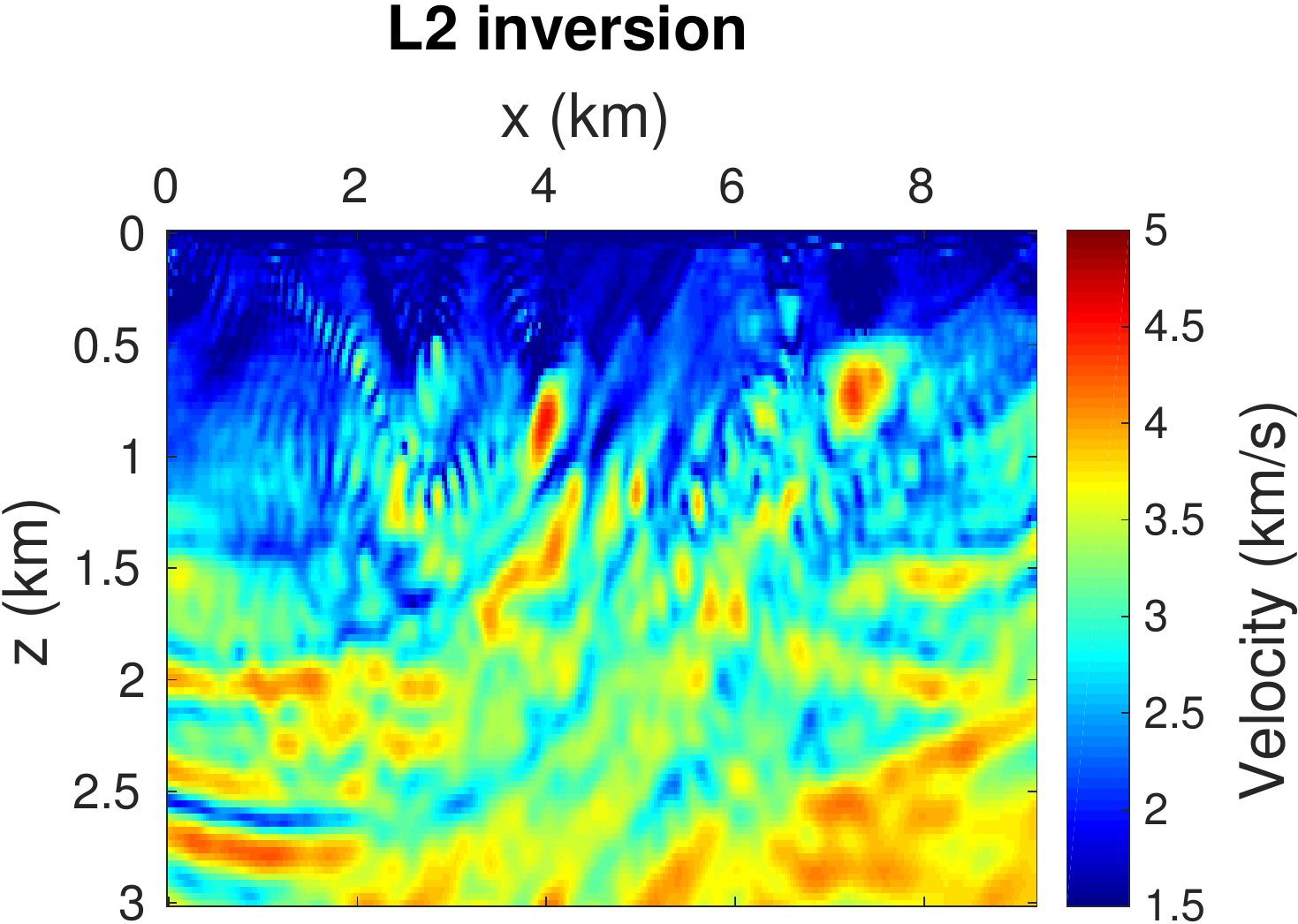}\label{fig:marm2_L2}}
  \subfloat[$W_2$ final result]{\includegraphics[width=0.24\textwidth]{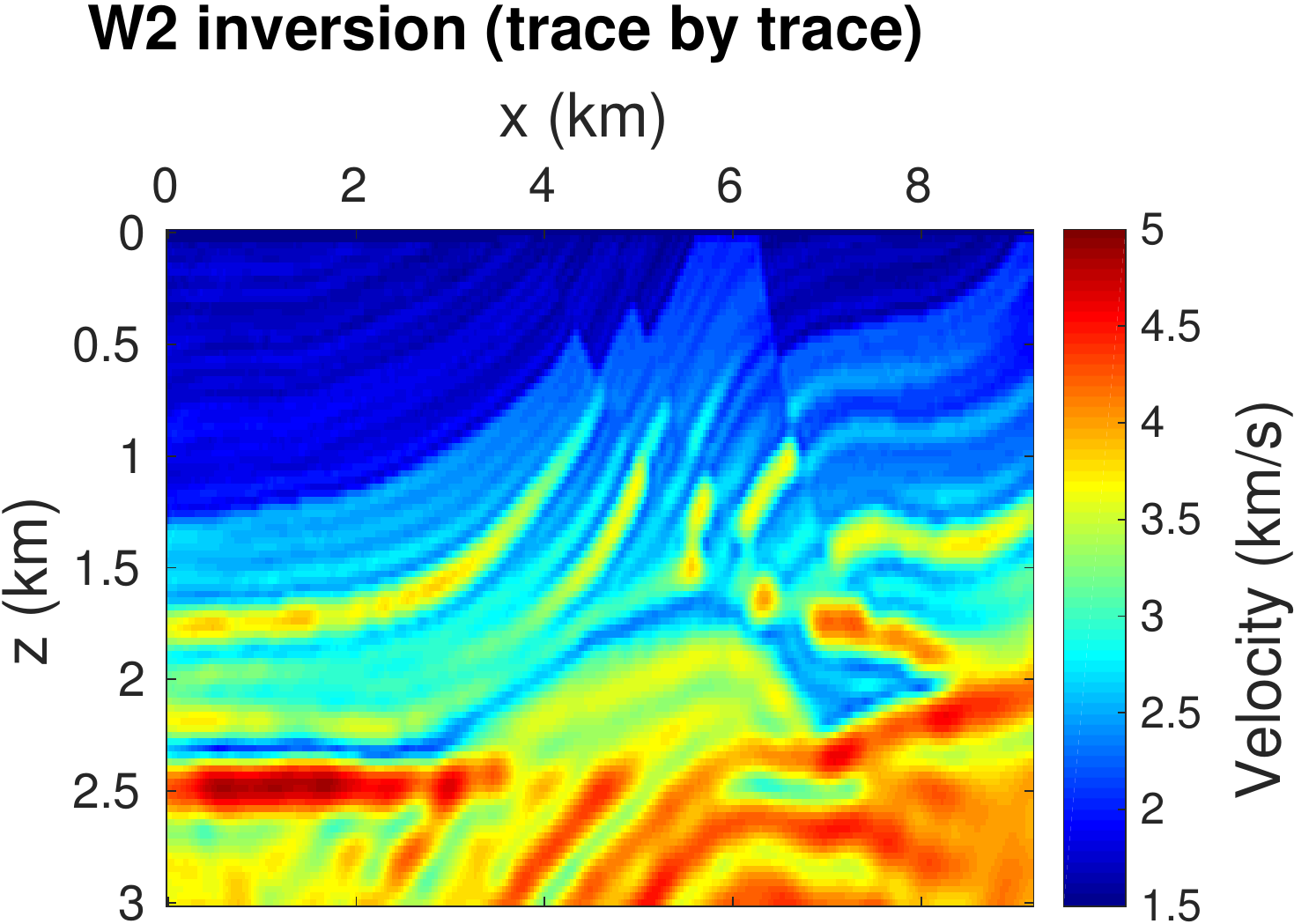}\label{fig:marm2_w2_1D}}
   \subfloat[Convergence histories along iterations]{\includegraphics[width=0.5\textwidth]{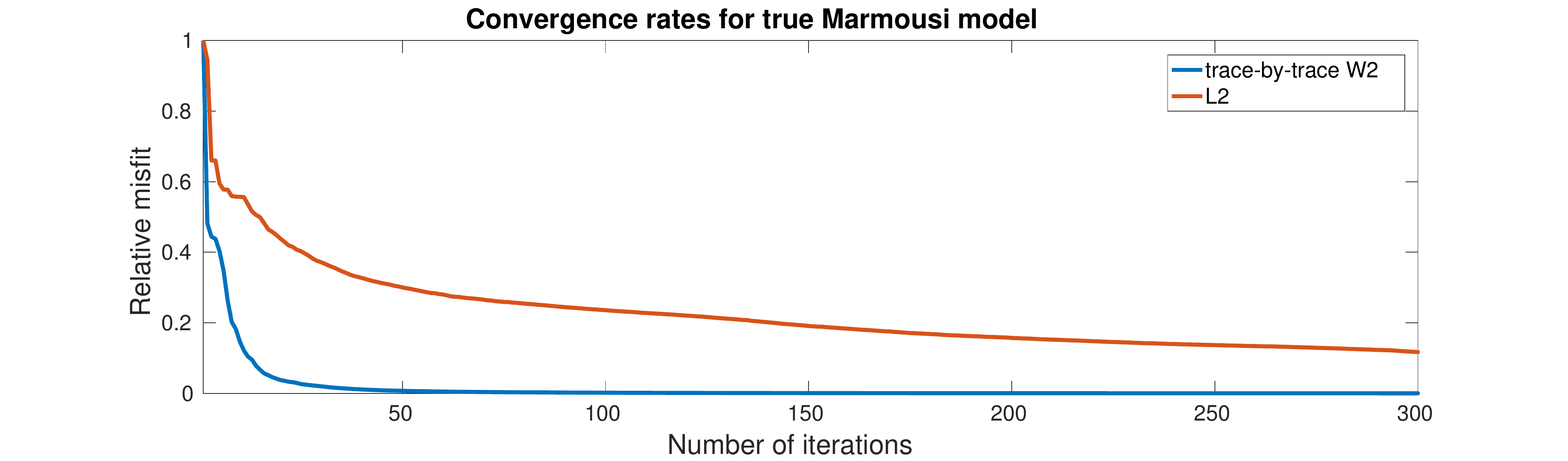}\label{fig:marm2_conv} }
  \caption{Full Marmousi model: inversion results after 300 iterations of (a)~the $L^2$-FWI; (b)~the trace-by-trace $W_2$-FWI and (c)~the convergence histories of $L^2$-FWI (red) and trace-by-trace $W_2$-FWI (blue).}
  \label{fig:marm2_inv}
\end{figure}

\section{Challenge Two: Inversion with Noisy Data}
\label{sec:challenge2}
Besides local minima issues, the problem of the $L^2$ norm is exacerbated by the fact that observed signals usually suffer from noise in the measurements. In the practical application of full-waveform inversion (FWI), it is natural to experience noise in the true observable data, and therefore robustness with respect to noise is a desirable property of the objective function. 
\subsection{Insensitivity to Noise}
Another essential property of optimal transport is the insensitivity to noise. The $W_2$ distance is proved to be insensitive to mean-zero noise, and the property applies to any dimension of the data in~\cite{engquist2016optimal}. This is a natural result of optimal transport theory since the $W_2$ metric defines a global comparison that considers not only the difference in signal intensity but also the phase mismatches.

\begin{theorem}[Insensitivity to noise~{\cite{engquist2016optimal}}]\label{thm:noise}
Let $f_{ns}$ be $f$ with a piecewise constant additive noise of mean zero uniform distribution.
The squared Wasserstein metric $W_2^2(f,f_{ns})$ is of $\bO(\frac{1}{N})$ where $N$ is the number of pieces of the additive noise in $f_{ns}$.
\end{theorem} 

On the other hand, the $L^2$ norm is known to be sensitive to noise since the misfit between clean and noisy data is calculated as the sum of squared noise amplitude at each sampling point. Noise has $\bO(1)$ effect on $L^2$ based objective function, while $\bO(\frac{1}{N})$ effect on the squared $W_2$ distance.

\subsection{Numerical Example}

\begin{figure}
\centering
  \subfloat[True velocity]{\includegraphics[height = 3cm]{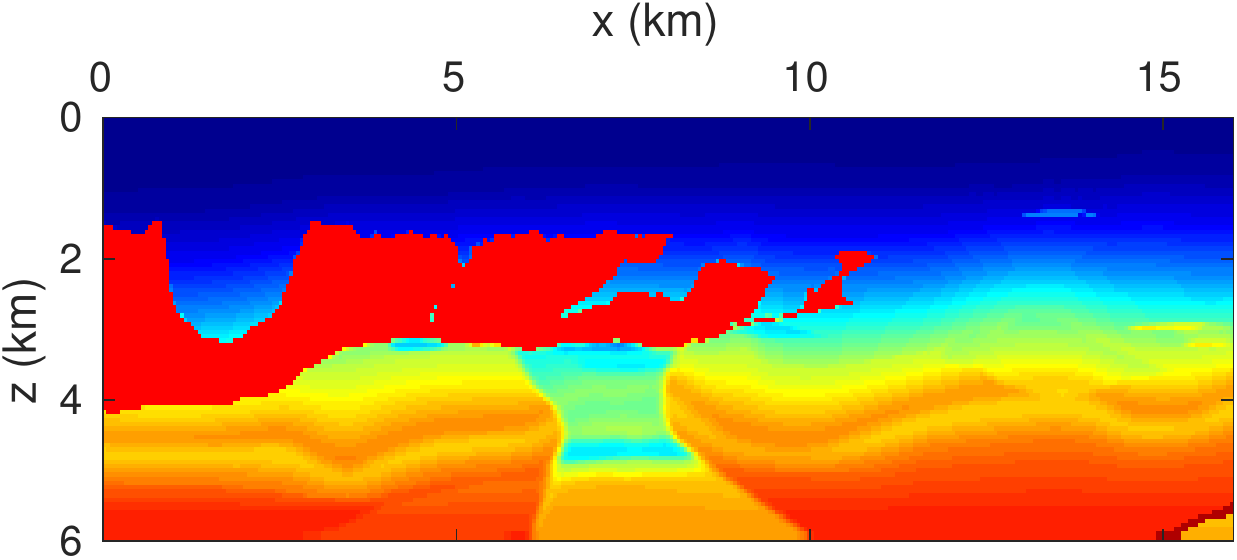}\label{fig:BP2_true}}
  \subfloat[Initial velocity]{\includegraphics[height = 3cm]{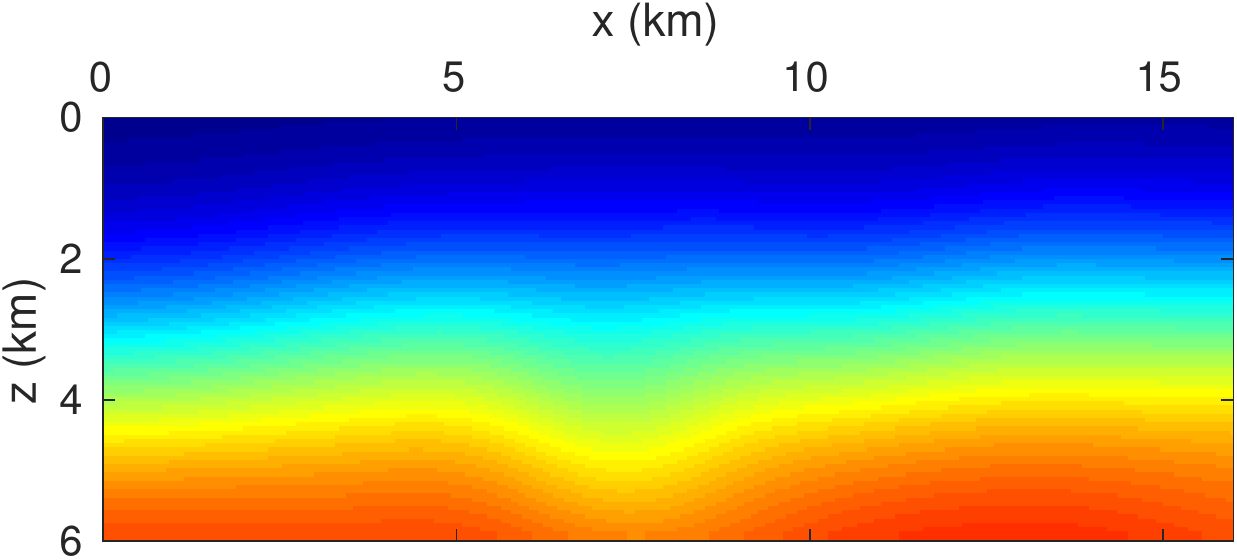}\label{fig:BP2_v0}}\\
 \subfloat[$W_2$ inversion final result]{\includegraphics[height = 3cm]{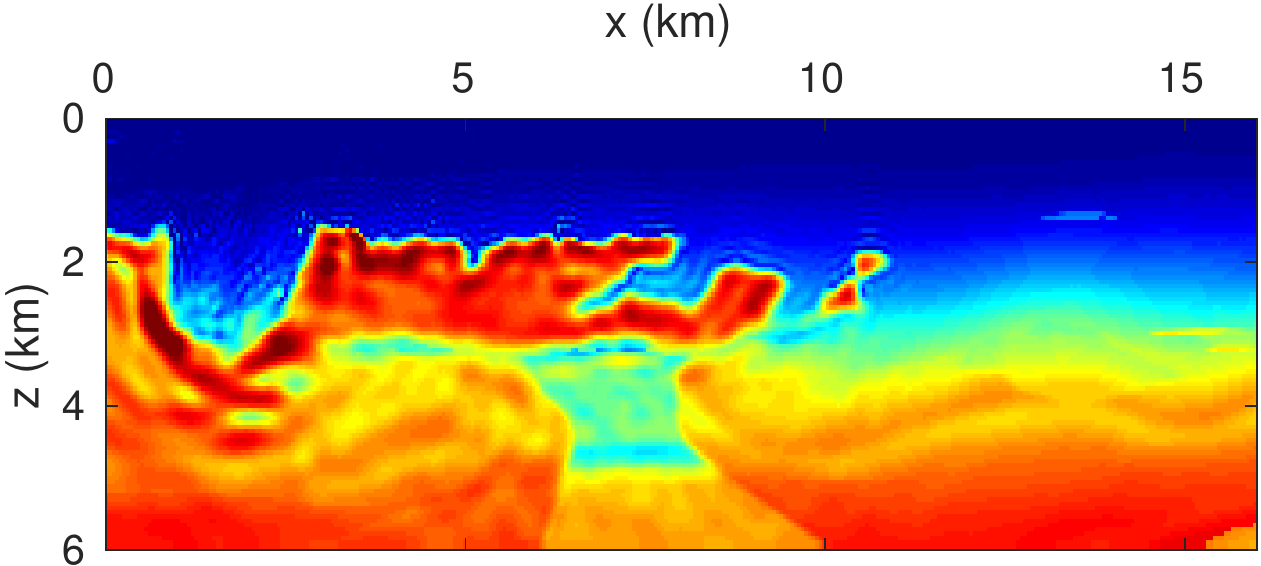}\label{fig:BP_w2_1D}} 
\subfloat[$L^2$ inversion final result]{\includegraphics[height = 3cm]{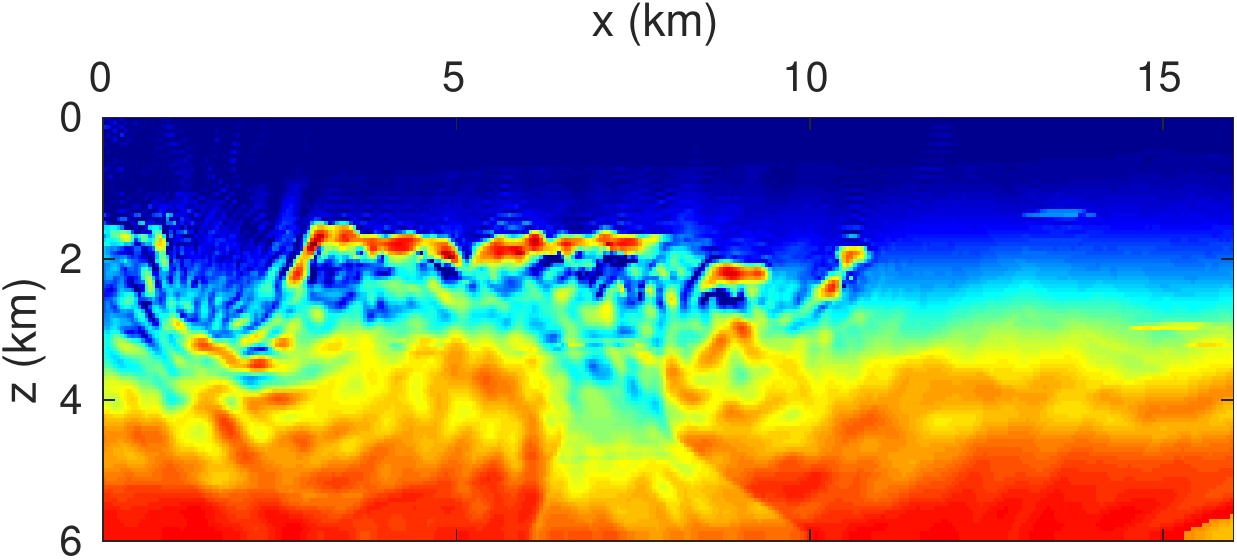}\label{fig:BP_L2}}\\
      \subfloat[Noisy and clean data]{\includegraphics[height = 3cm]{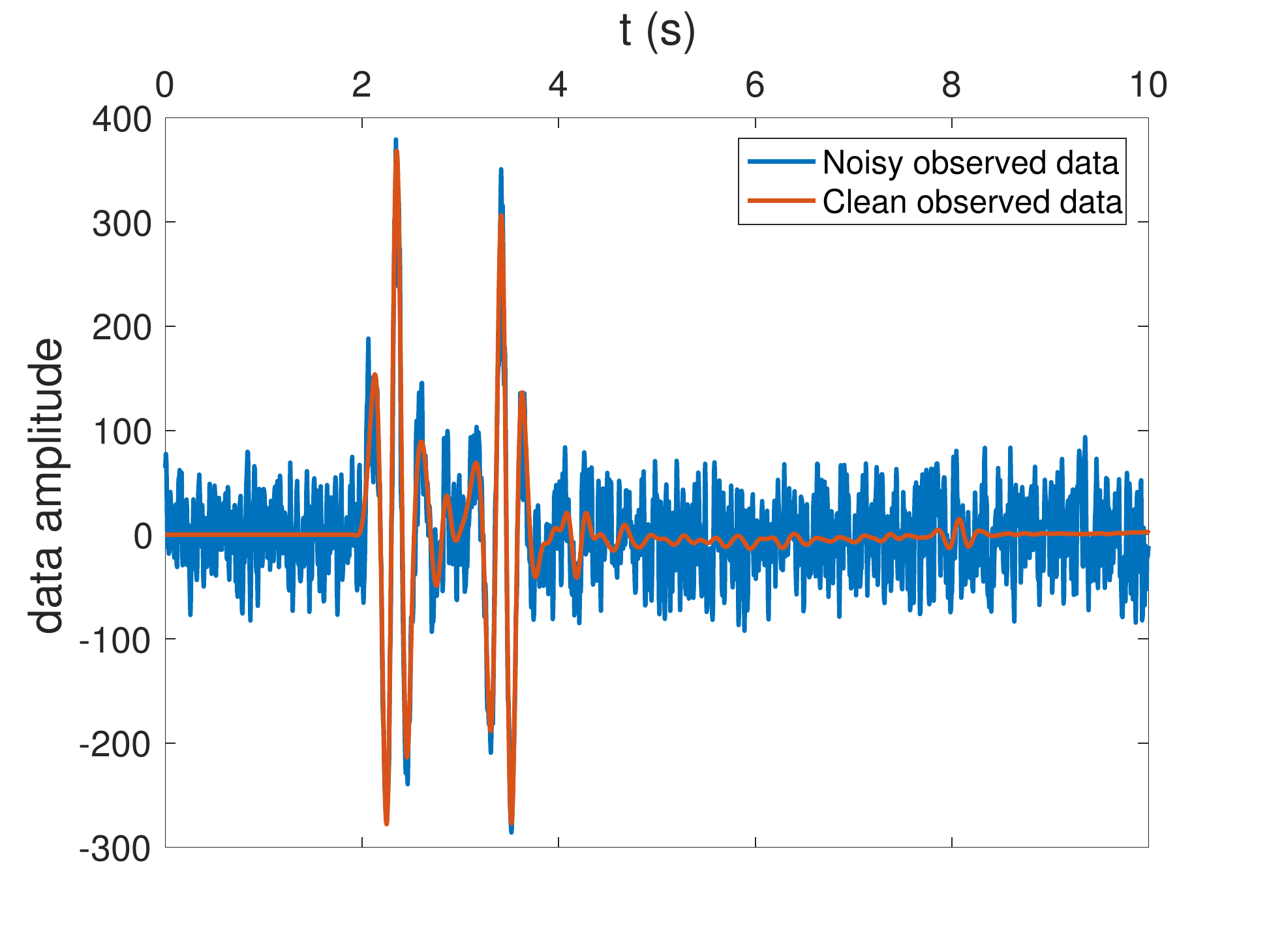}\label{fig:BP_noise_trace}} \hspace{0.3cm}
      \subfloat[$W_2$-FWI using the noisy data in Figure~\ref{fig:BP_noise_trace}]{\includegraphics[height = 3cm]{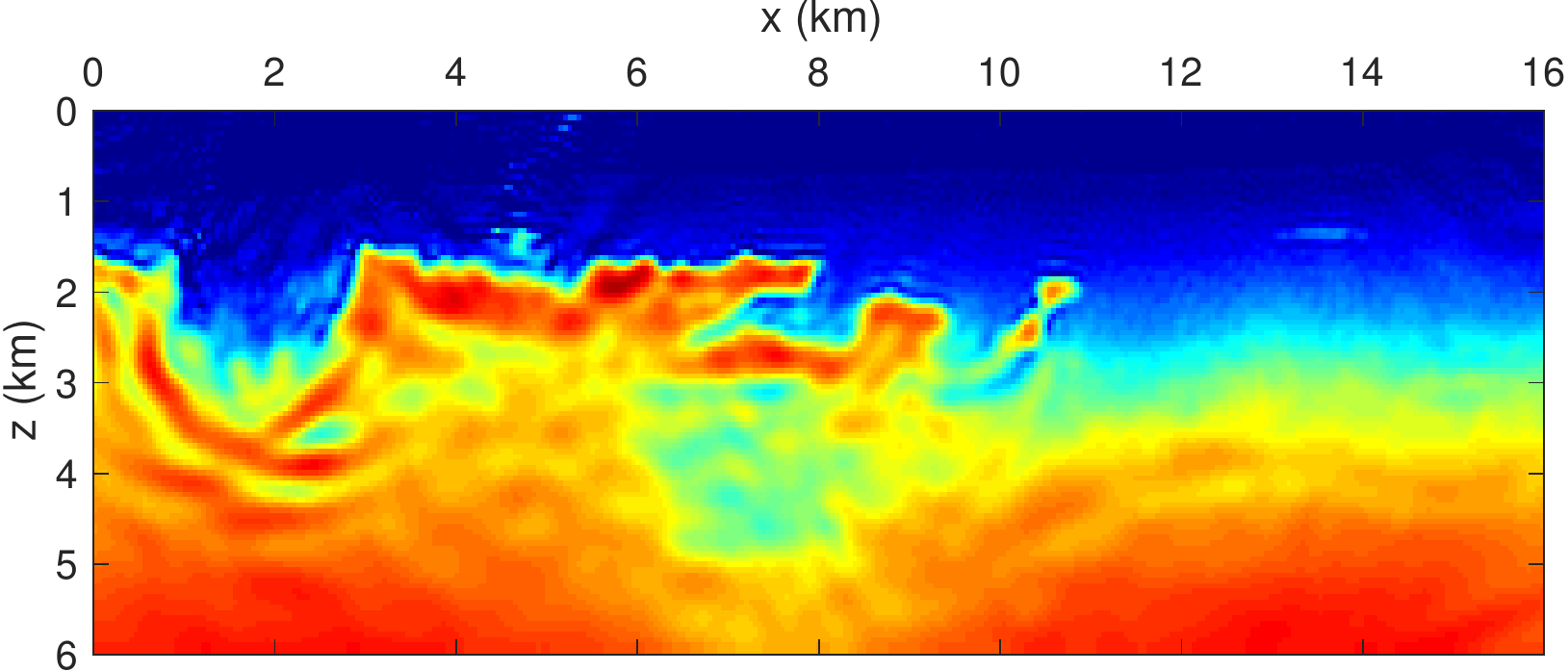}\label{fig:BP_noise_vel}}\\
  \caption{2004 BP model: (a)~true velocity; (b)~initial velocity; Inversion result using (c) $L^2$ norm and (d)~trace-by-trace $W_2$ distance; (e)~noisy (blue) and noise-free data (red) at one trace; (f)~$W_2 $ inversion result from noisy data.}
  \label{fig:BP2_true,BP2_v0}
\end{figure}

Next, we will demonstrate insensitivity to noise, an ideal property of the \QW, by inverting the modified 2004 BP Benchmark, which is 6 km in depth and 16 km in width (Figure~\ref{fig:BP2_true}). Part of the model is representative of the complex geology in the deepwater Gulf of Mexico. The main challenges in this area are related to obtaining a precise delineation of the salt and recover information on the sub-salt velocity variations~\cite{billette20052004}. The inversion starts with a smoothed background without the salt (Figure~\ref{fig:BP2_v0}). We put 11 equally spaced sources on top at 250 m depth in the water layer and 321 receivers on top at the same depth with a 50 m fixed acquisition. The discretization of the forward wave equation is 50 m in the $x$ and $z$ directions and 5 ms in time. The source is a Ricker wavelet with a peak frequency of 5 Hz, with a band-pass filter applied to keep 3 to 9 Hz frequency. The total acquisition time is 10 seconds. With clean target data, inversion with the trace-by-trace $W_2$ misfit successfully reconstructs the shape of the salt bodies (Figure~\ref{fig:BP_w2_1D}), while FWI using the conventional $L^2$ norm fails to recover neither the main body nor the entire boundaries of the salt as shown by Figure~\ref{fig:BP_L2}.

Next, we set up the experiment in a more challenging and realistic setting by adding correlated noise to the observed data as the new reference (Figure~\ref{fig:BP_noise_trace}). At each time grid, the noise is 
\bq
\tilde{r}(j) = \frac{r(j-1)+2r(j)+r(j+1)}{4} \left(1+\frac{g(j)}{||g||_\infty}\right),
\eq 
where $r$ is mean-zero uniform iid noise and $g$ is the clean observed data. The SNR ratio is 5.98 dB. The observed data to be fitted in this experiment is $\tilde{g} = g + \tilde{r}$ which is the original reference data plus noise. 

After 185 iterations, the optimization converges to a velocity presented in Figure~\ref{fig:BP_noise_vel}. Compared with Figure~\ref{fig:BP_w2_1D}, the result has lower accuracy around the salt bottom due to the strong noise added to the target data (Figure~\ref{fig:BP_noise_trace}). However, we still can recover the salt body and upper boundaries reasonably well compared with Figure~\ref{fig:BP_L2}  where the $L^2$ norm barely generates meaningful results, even with the noise-free target data. 

Although Theorem~\ref{thm:noise} is proved in \cite{engquist2016optimal} under the condition of uniform iid noise, we observe that $W_2$ is insensitive to the noise of other distributions. 
The last example demonstrates the robustness of the $W_2$ distance with respect to mean-zero \textit{correlated} noise.
The robustness to noise comes from local cancellation in optimal transport~\cite{engquist2016optimal}. This advantage is favorable for the trace-by-trace technique to deal with various types of mean-zero noise. 

\section{Challenge Three: Inversion with Reflection-Dominated Data}\label{sec:challenge3}
Another common type of recorded data is the seismic reflection. Reflections carry essential information of the deep region in the subsurface, especially when there are no diving waves or other refracted waves traveling through due to a narrow range of recording, i.e., short offsets. Inversions using the $L^2$ norm as a misfit are prone to update the high-wavenumber components in the model parameter as it is a faster way to decrease the $L^2$ norm objective function than updating the low-wavenumber components of the velocity model. It is related to scattering theory and the particular structure of Born approximation~\cite{Mora1988}. 


In this section, we discuss a particular layered model, whose velocity only varies vertically. The only available data for inversion is the reflection. Conventional $L^2$-based FWI for this problem has been problematic in the absence of a good initial model since the high-wavenumber features updated by reflections often slow the recovery of the missing low-wavenumber components. Most of the time, the convergence stalls. 
We will see that partial inversion for velocity below the deepest reflecting interface is still possible by using the $W_2$ distance as the objective function in optimization. 


\subsection{Experiment Setup and Inversion Results}
Let us consider a three-layer model illustrated in Figure~\ref{fig:test1_vG}, where the deepest third layer is not present in the initial model as seen in Figure~\ref{fig:test1_vF}. The source in the test is a Ricker wavelet with a peak frequency of 5 Hz. There are in total 52 sources and 301 receivers on top in the first (water) layer. We do not update the model parameter in the first layer because of the strong artifacts in the gradient around the source location. The total recording time is 3.8 seconds. There is naturally no back-scattered information from the interior of the third layer returning to the receivers. Due to natural boundary conditions, the second reflection is the only available information for reconstruction. Nevertheless, we will see that the velocity profile of the third layer can be partially recovered using $W_2$ as the objective function in this PDE-constrained optimization. 

\begin{figure}
\centering
   \subfloat[True velocity]{\includegraphics[height = 3cm]{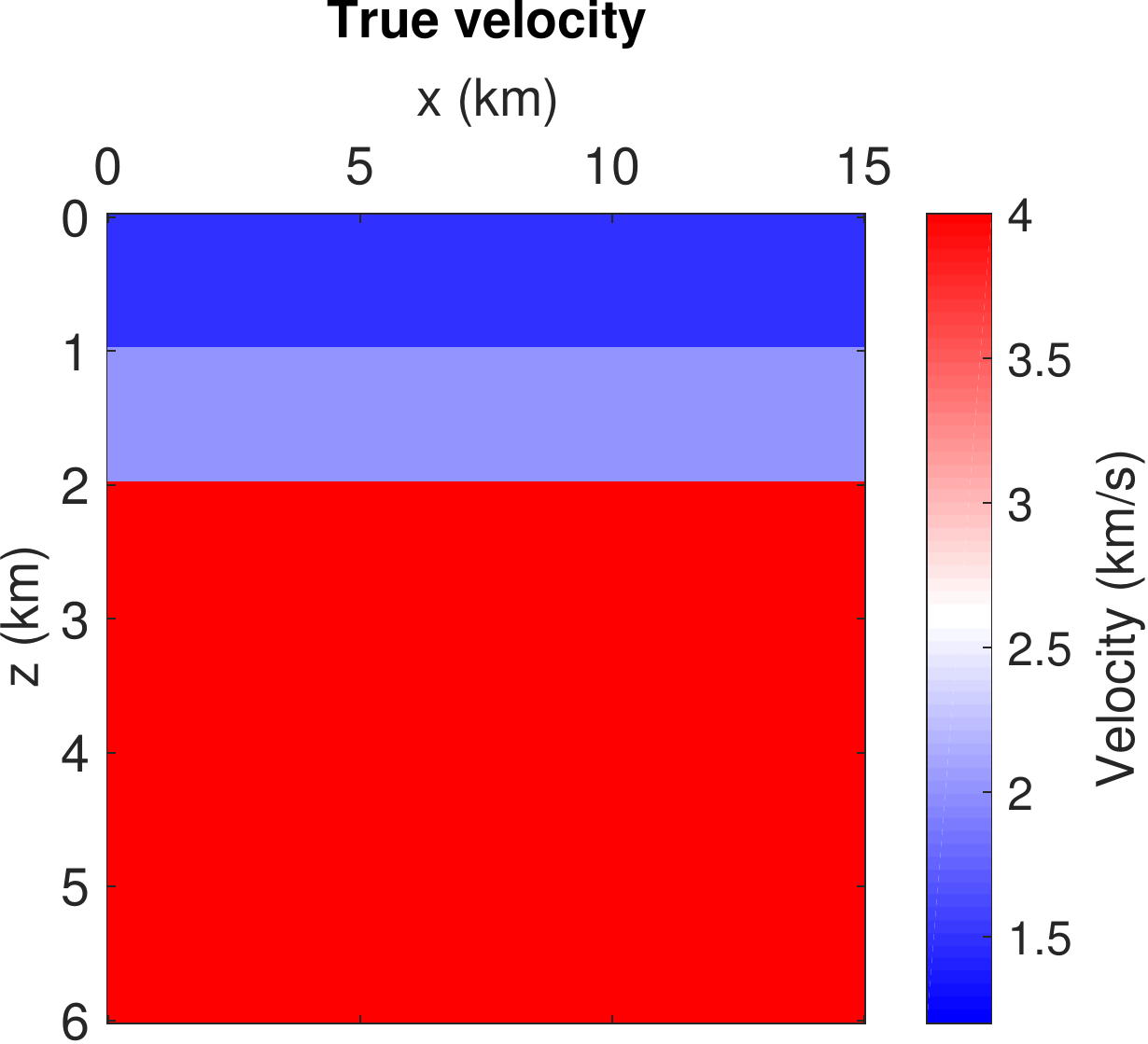}\label{fig:test1_vG}}
   \subfloat[Initial velocity]{\includegraphics[height = 3cm]{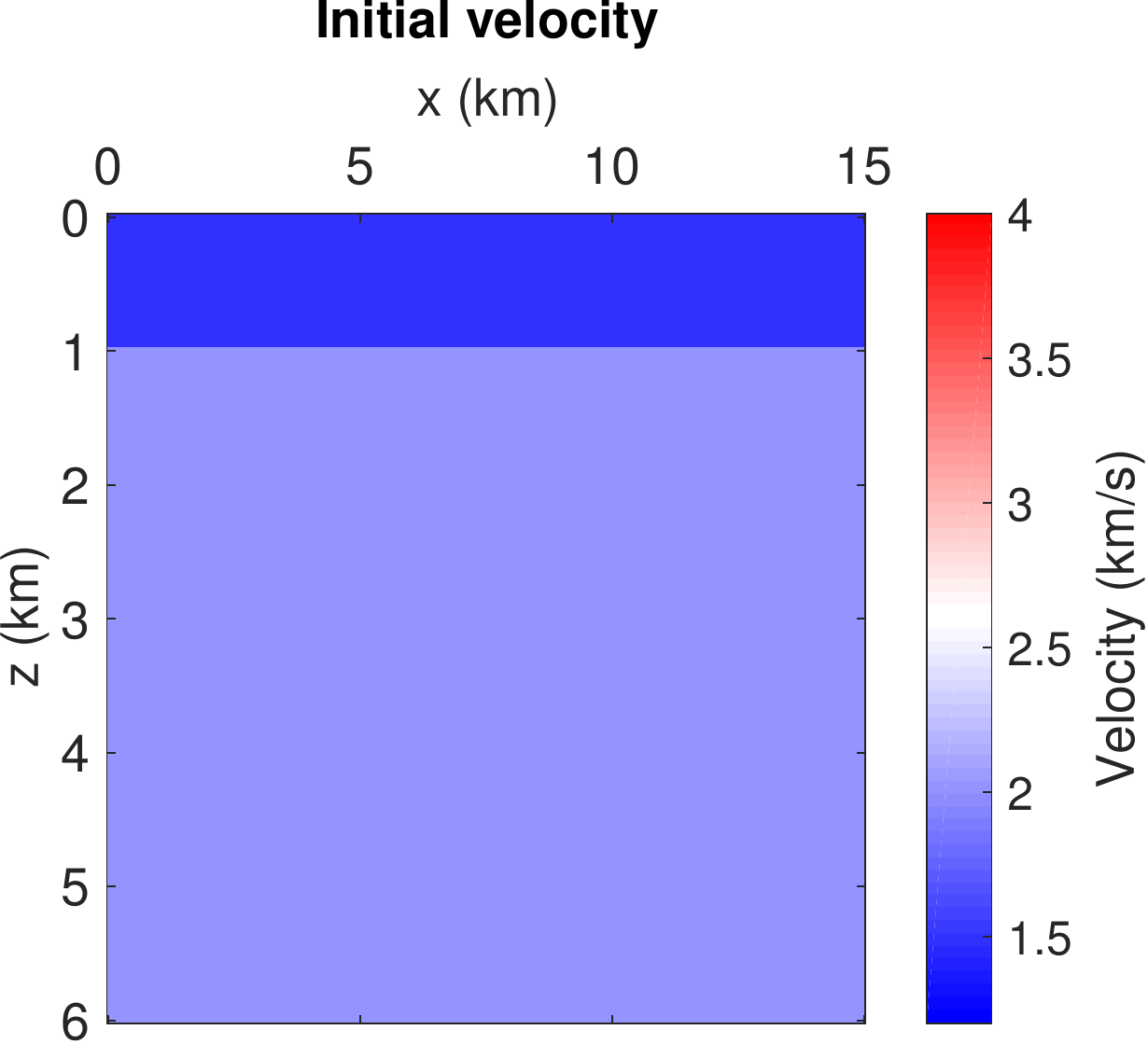}\label{fig:test1_vF}} 
   \subfloat[$L^2$-FWI result]{\includegraphics[height = 3cm]{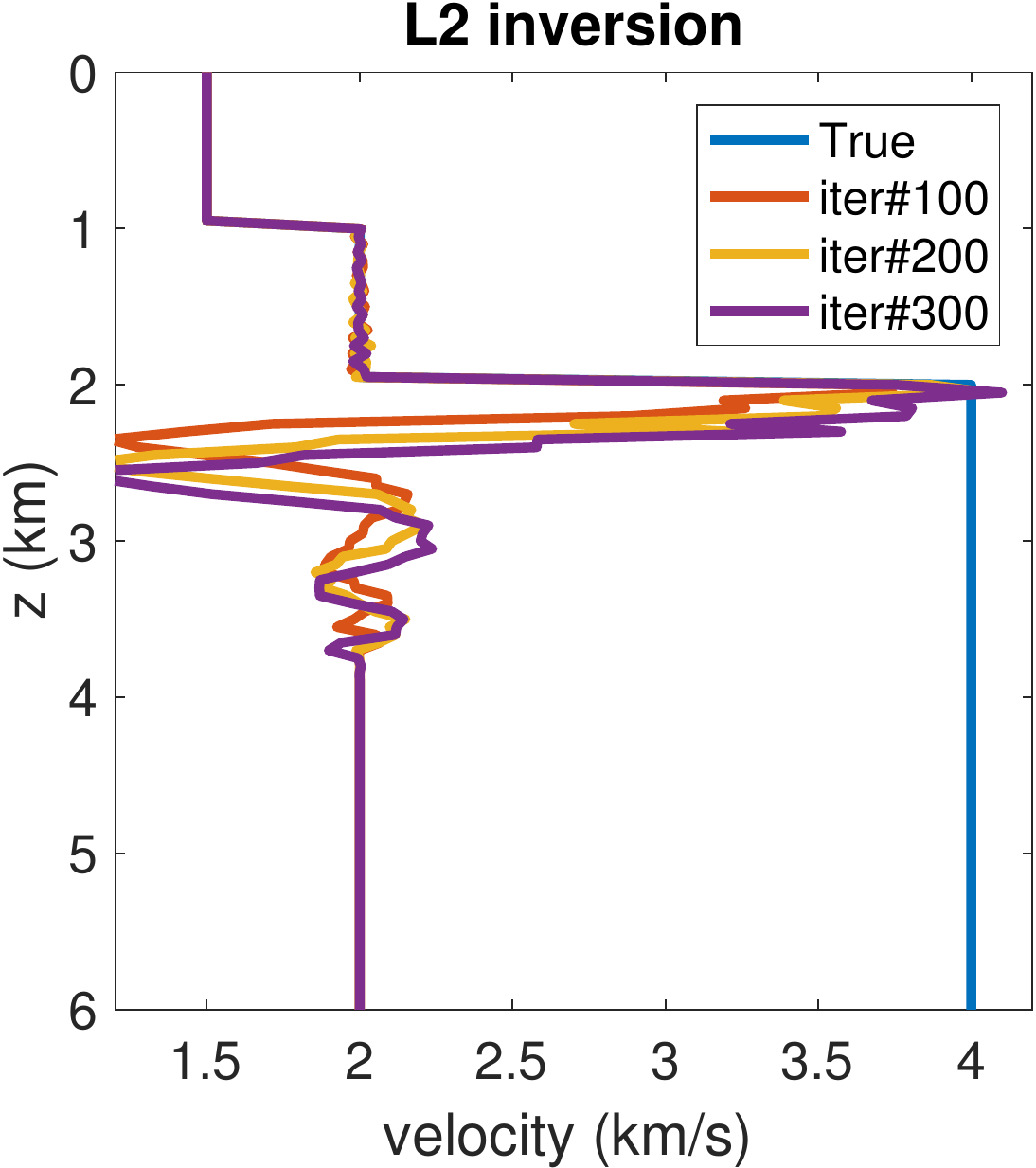}\label{fig:test1_L2}}
   \subfloat[$W_2$-FWI result]{\includegraphics[height = 3cm]{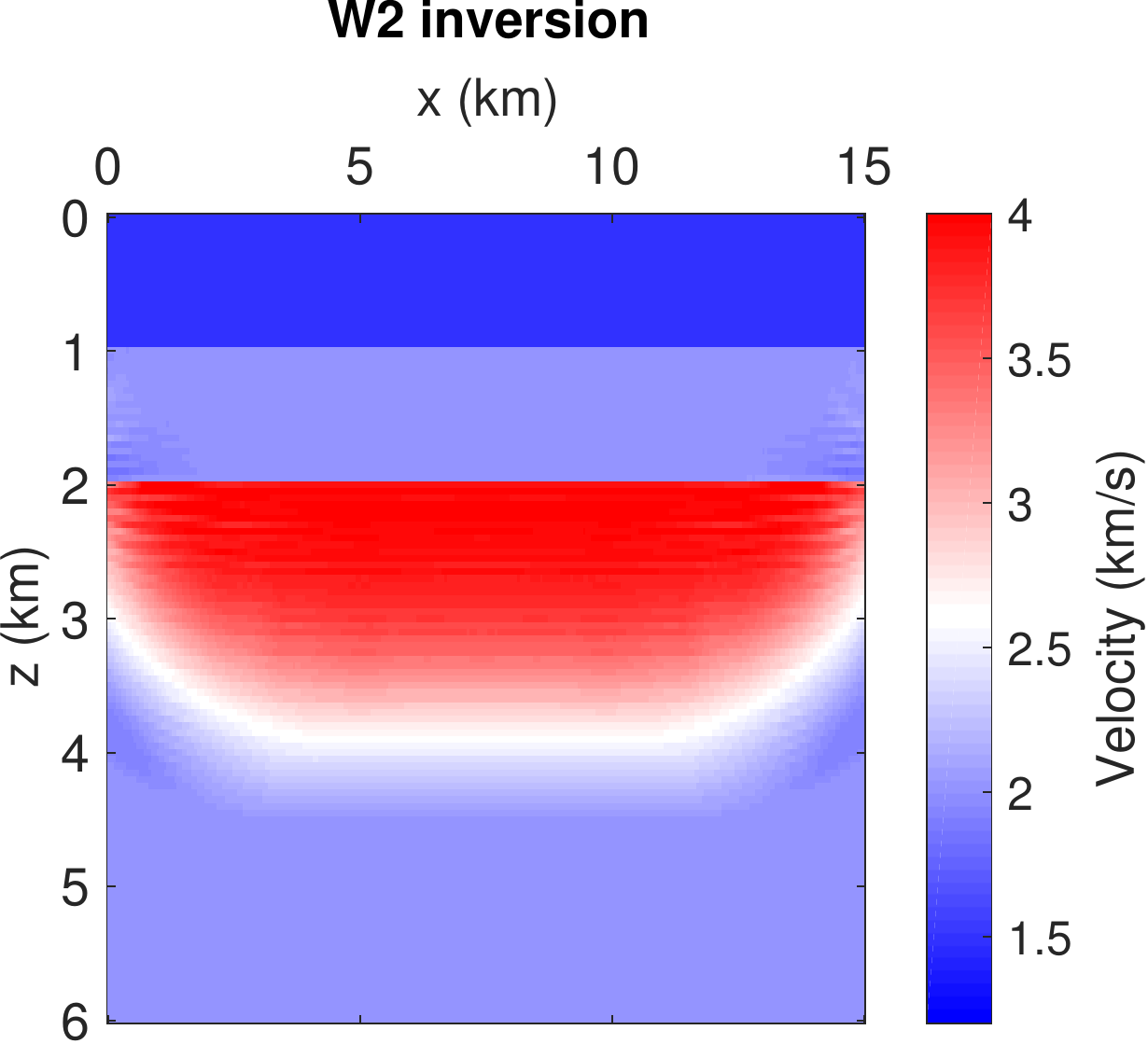}\label{fig:test1_W2}}  
  \caption{Three-layer model: (a) true velocity model; (b) initial velocity model; (c)~$L^2$-FWI final inversion result and (d)~$W_2$-FWI final inversion result.} \label{fig:test1_vel2D}
\end{figure}

After 1000 iterations, $L^2$ based inversion only recovers a migration-type structure with severe overshooting as seen in Figure~\ref{fig:test1_L2}, while FWI with the $W_2$ distance correctly recovers part of the third layer as seen in Figure~\ref{fig:test1_W2}. Figure~\ref{fig:test1_L2_iter} shows that the velocity model of $L^2$-FWI changes so slowly that it does not give much more information than the local velocity change around the interface. Nevertheless, the $W_2$ result in Figure~\ref{fig:test1_W2_iter} gradually recovers part of the sub-layer velocity as the iteration continues. The radius of the model update is much wider in the $W_2$ case. 

\subsection{Discussion}
It seems puzzling at first glance that $W_2$ based FWI can even recover velocity in the model where no seismic wave returned to the surface from below the lowest reflecting interface. In this subsection, we discuss several factors that contribute to the model recovery below the reflectors when using $W_2$ distance. 

\subsubsection{The Role of the Simulated Reflection}
\begin{figure}
\centering
  \subfloat[]{\includegraphics[height = 4cm]{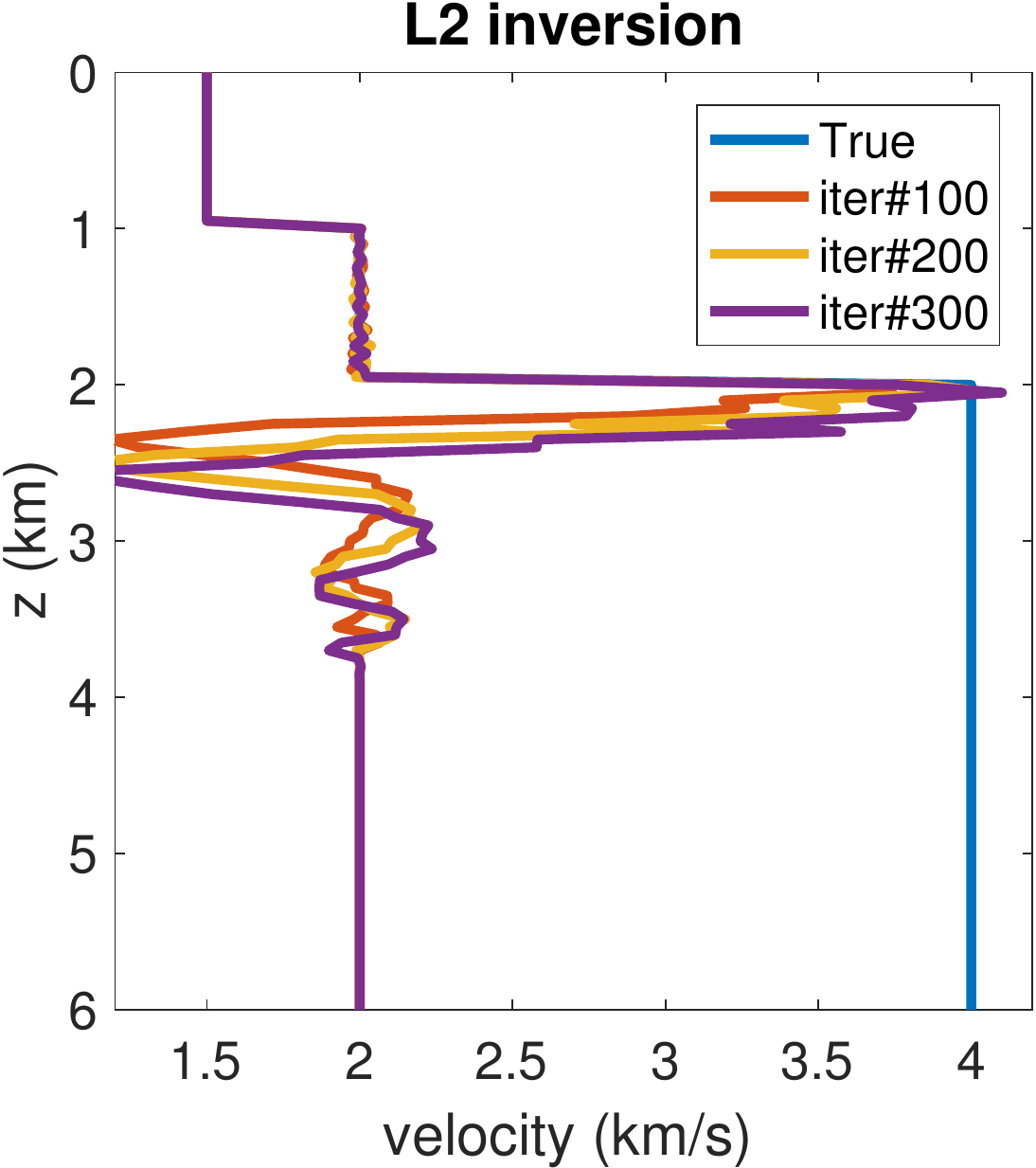}\label{fig:test1_L2_iter}}\hspace{1cm}
  \subfloat[]{\includegraphics[height = 4cm]{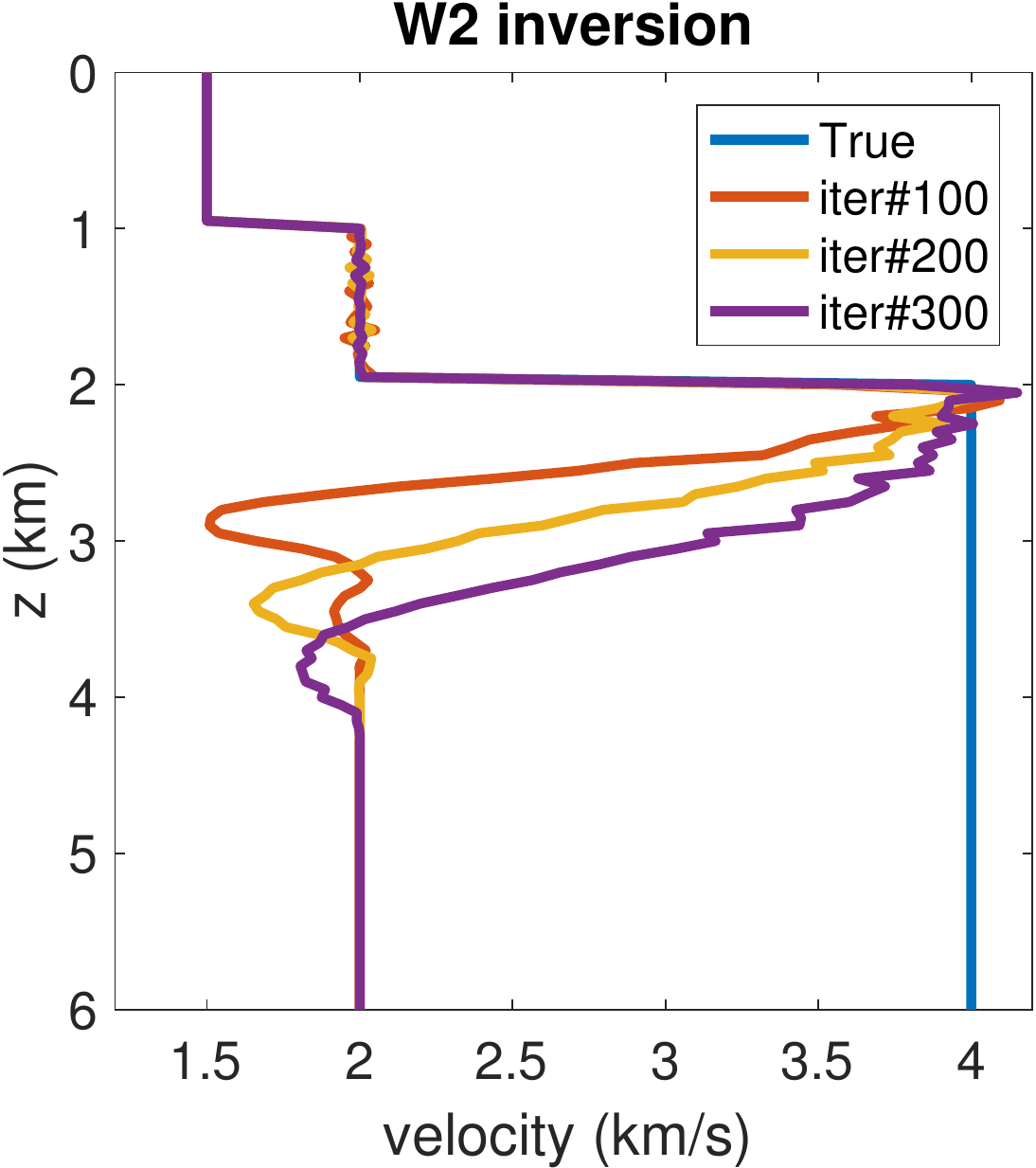}\label{fig:test1_W2_iter}}
  \caption{Three-layer model: (a)~$L^2$ and (b) $W_2$ inversion velocity after 100, 200 and 300 iterations.}\label{fig:test1_iter}
\end{figure}

Figure~\ref{fig:test1_L2_iter} and Figure~\ref{fig:test1_W2_iter} are the vertical velocity profiles of $L^2$ and $W_2$ based FWI after 100, 200 and 300 iterations. The inverted velocity profile in the earlier iterations has a sharp lower transition back to 2 km/s causing simulated reflections in the synthetic data, which are not in the observed data. The simulated reflector is present in the earlier iterations of both $L^2$ and $W_2$ based FWI, as an overshooting phenomenon from gradients based on the adjoint-state method (Equation~\ref{eq:adj_grad3}). However, this simulated reflector can provide us with additional information, especially the velocity model below the reflector.

In the $W_2$ case, the additional reflection contributes greatly to the $W_2$ objective function since it is not in the original target data. Further iterations update the velocity model with the goal of removing the simulated reflection and decreasing the $W_2$ distance. This results in a further correction of the velocity that starts recovery below the deepest reflector. This correction gradually slows when the lower transition is no longer sharp. As seen in Figure~\ref{fig:test1_W2}, the velocity of the third layer smoothly decreases from 4 km/s to 2 km/s. 
As more of the model below the reflector is correctly reconstructed, the simulated reflection appears gradually later, and finally is not recorded within the time range of 3.8 seconds. The short recording time only gives us limited information about the subsurface model. The incomplete reconstruction in Figure~\ref{fig:test1_W2} is a result of using \textit{partial} Cauchy boundary data, which again demonstrates the ill-posedness of the problem.

Although the simulated reflection can provide additional information about the sub-layer velocity model, it does not affect the overall $L^2$ misfit as much as it does in the $W_2$ case. The fundamental reason lies in the properties of these two metrics. The $L^2$ norm only considers signal intensities while the simulated reflection has a quite small amplitude. Therefore, FWI using the $L^2$ norm cannot quite ``see'' the existence of the simulated reflection. On the other hand, the Wasserstein distance is a global metric in the sense that both phase and amplitude differences are considered in the metric definition. For the case of $p=2$ proposed in this paper, the phase difference is squared and consequently plays a more critical role in the FWI objective function. The large phase mismatch caused by the simulated reflection strongly affects the $W_2$ distance, and consequently leads to further updates in the model below the reflectors.

\begin{figure}
\centering
  \subfloat[Velocity of thin layer]{\includegraphics[width=0.3\textwidth]{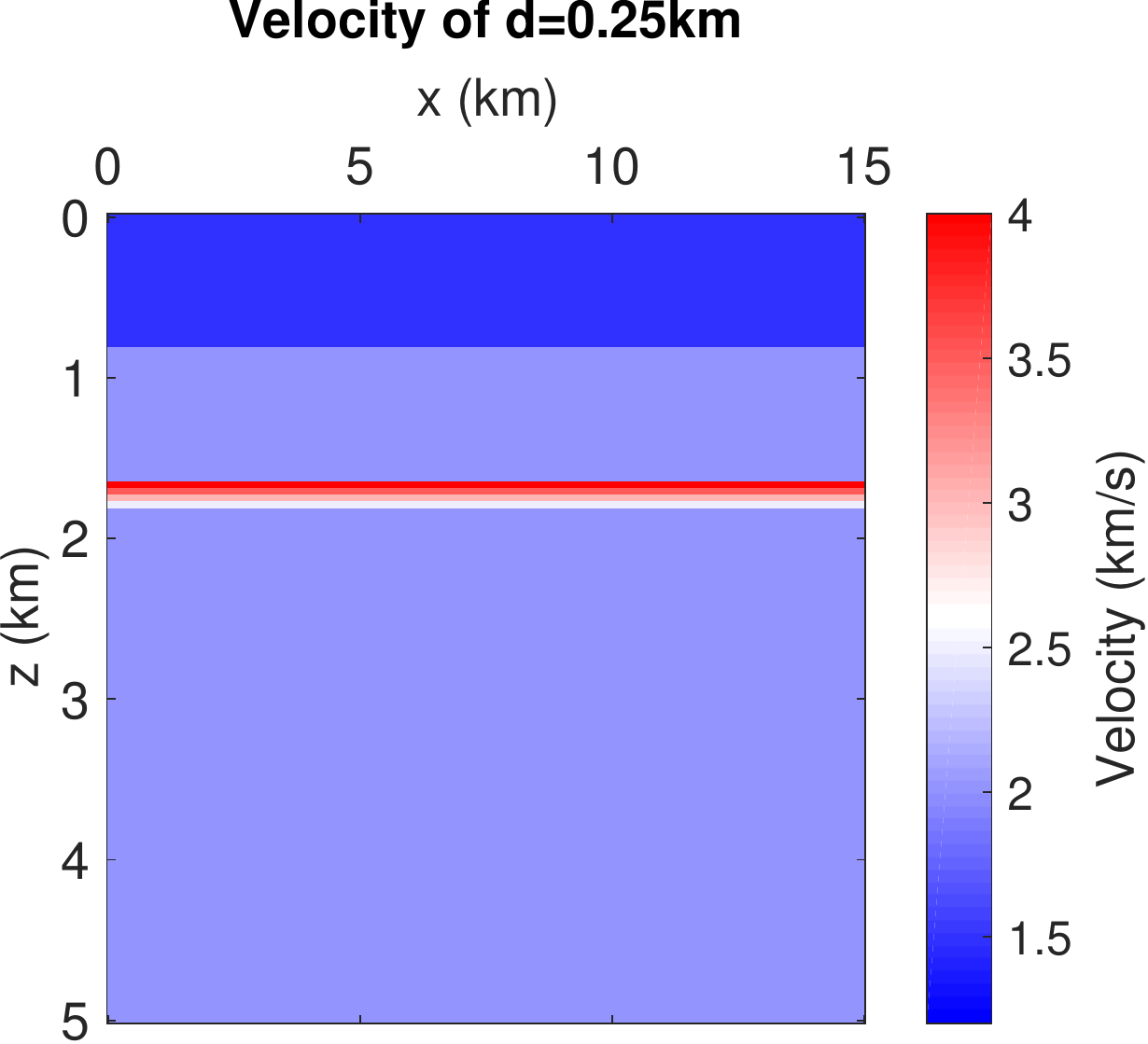}\label{fig:test1_vF5}} \hspace{0.2cm}
  \subfloat[Velocity of thick layer]{\includegraphics[width=0.3\textwidth]{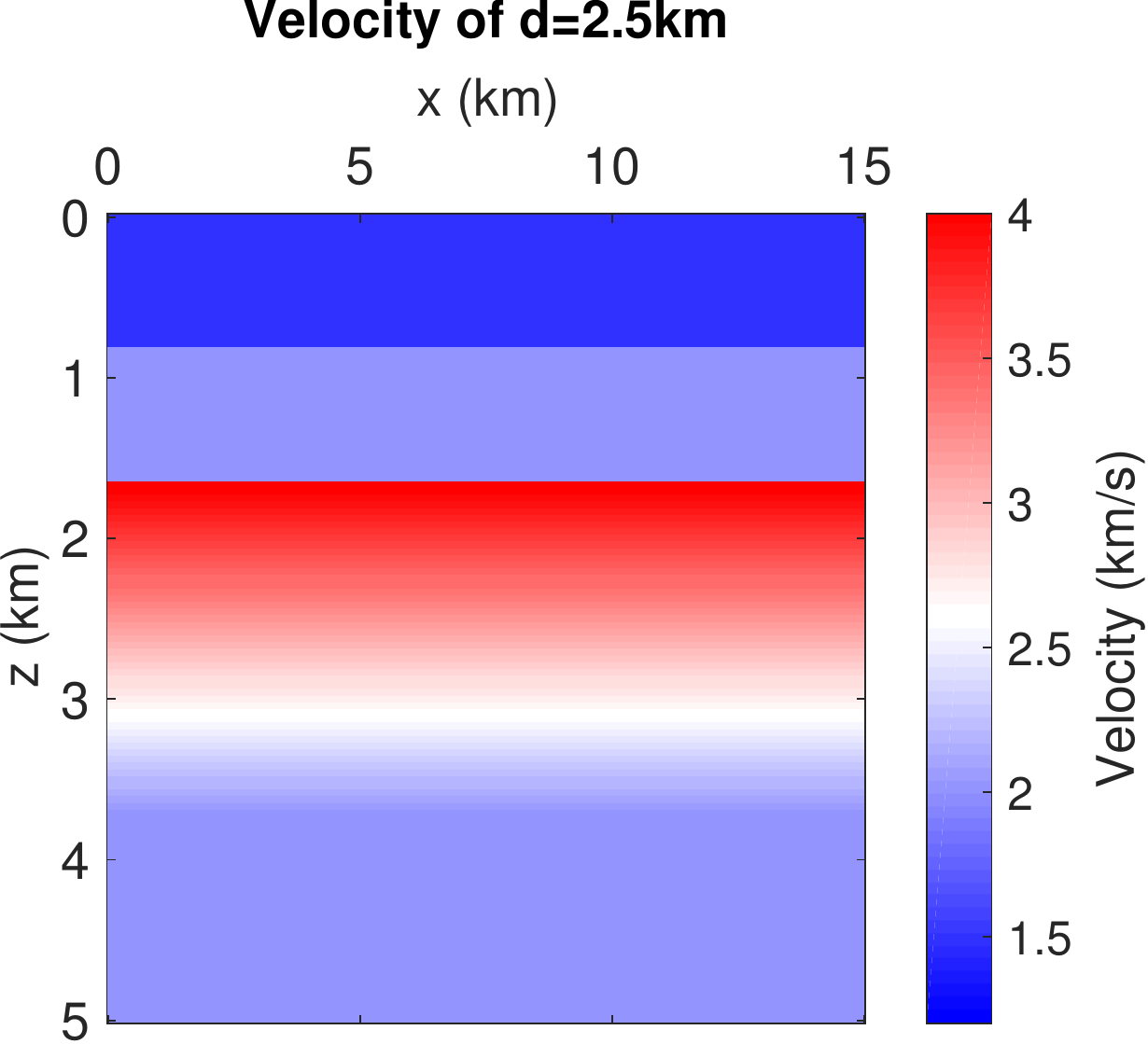}\label{fig:test1_vF50}}  \hspace{0.2cm}
  \subfloat[True velocity]{\includegraphics[width=0.3\textwidth]{test1-vG.pdf}\label{fig:test1_vtrue}}\\
  \subfloat[Wavefield generated by (a)]{\includegraphics[width=0.33\textwidth]{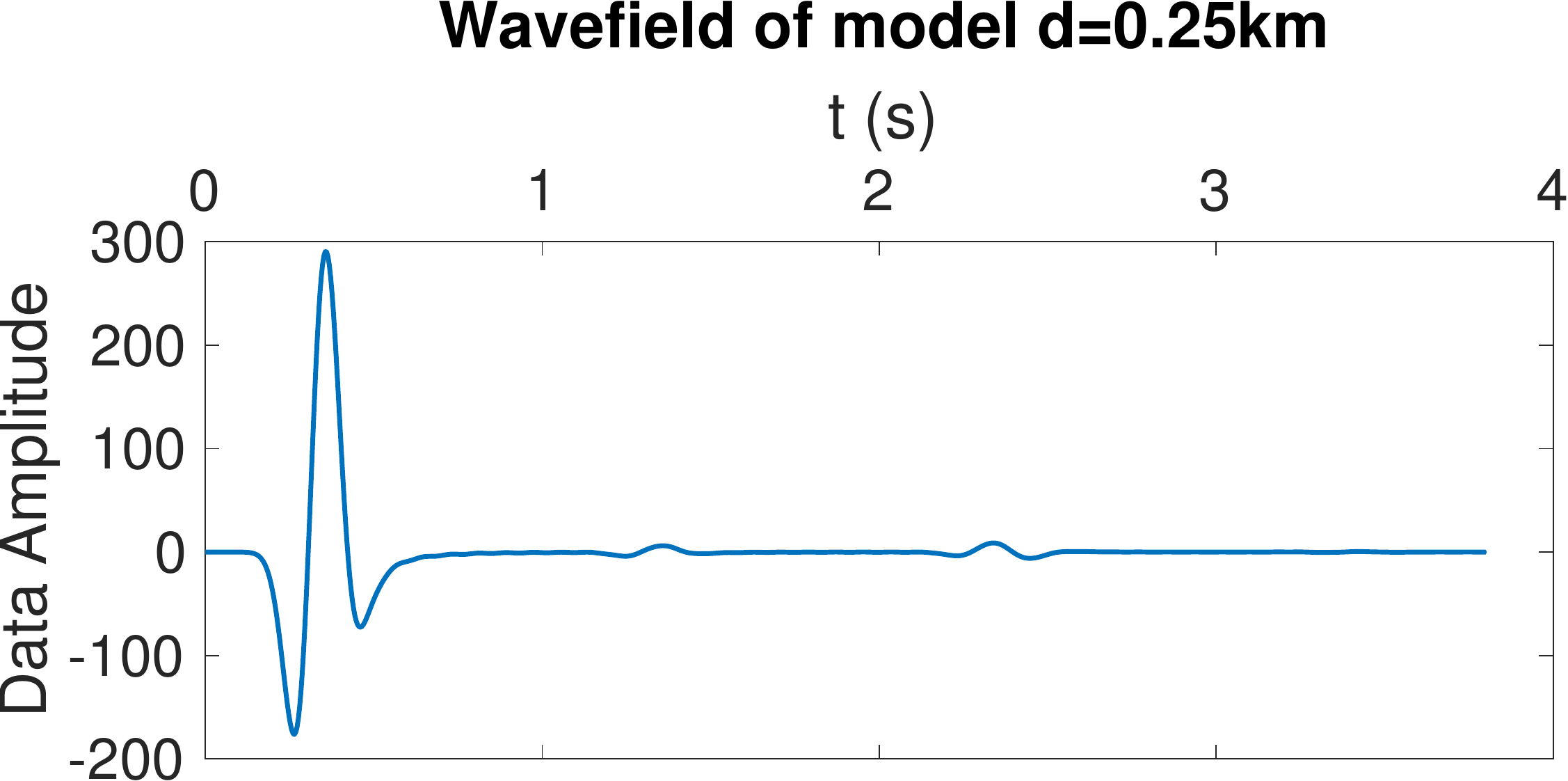}\label{fig:test1_d5}}
  \subfloat[Wavefield generated by (b)]{\includegraphics[width=0.33\textwidth]{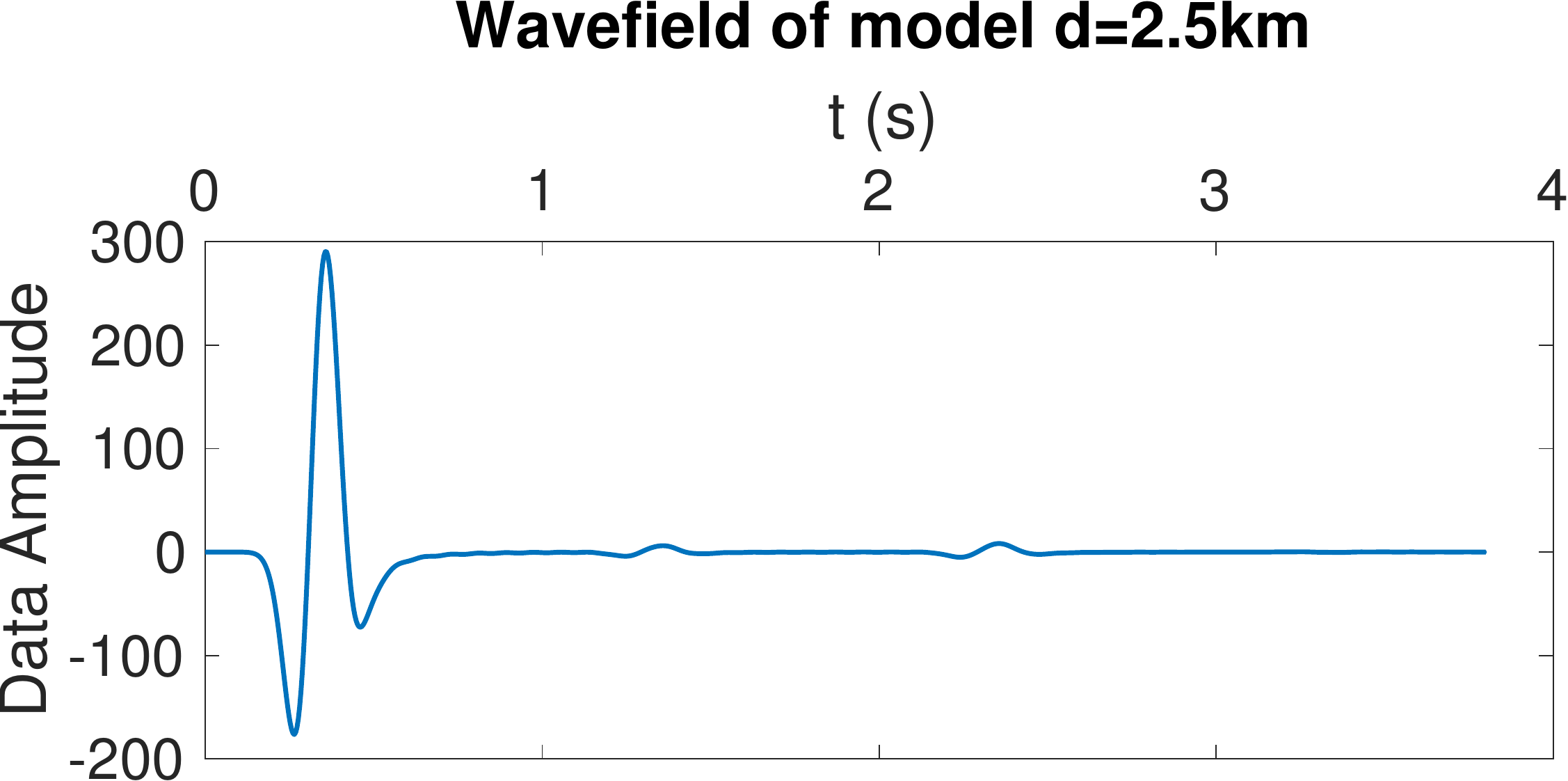}\label{fig:test1_d50}}
  \subfloat[Wavefield generated by (c)]{\includegraphics[width=0.33\textwidth]{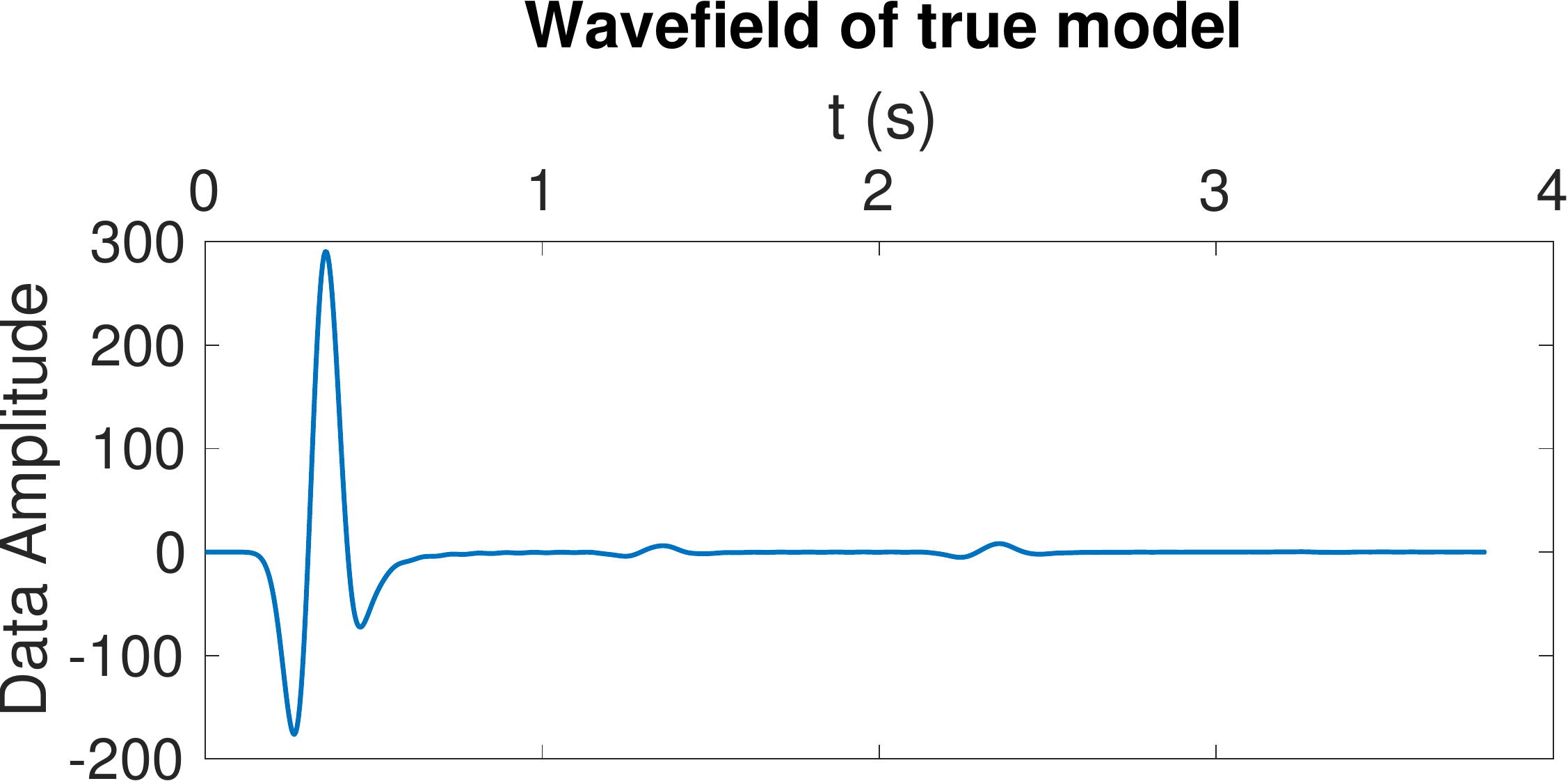}\label{fig:test1_dtrue}}
  \caption{Three-layer model: (a)(b)(c) velocity models with different layer thickness and (d)(e)(f) wavefield generated by three velocity models at $x=7.5$km, $z=0.25$km, i.e.,  $u(x=7.5,z=0.25,t)$ and $0\leq t\leq 3.8s$.}
  \label{fig:test1_same}
\end{figure}

\begin{figure}
\centering
  \subfloat[Data residual of model $d=0.25$km]{\includegraphics[height=3cm]{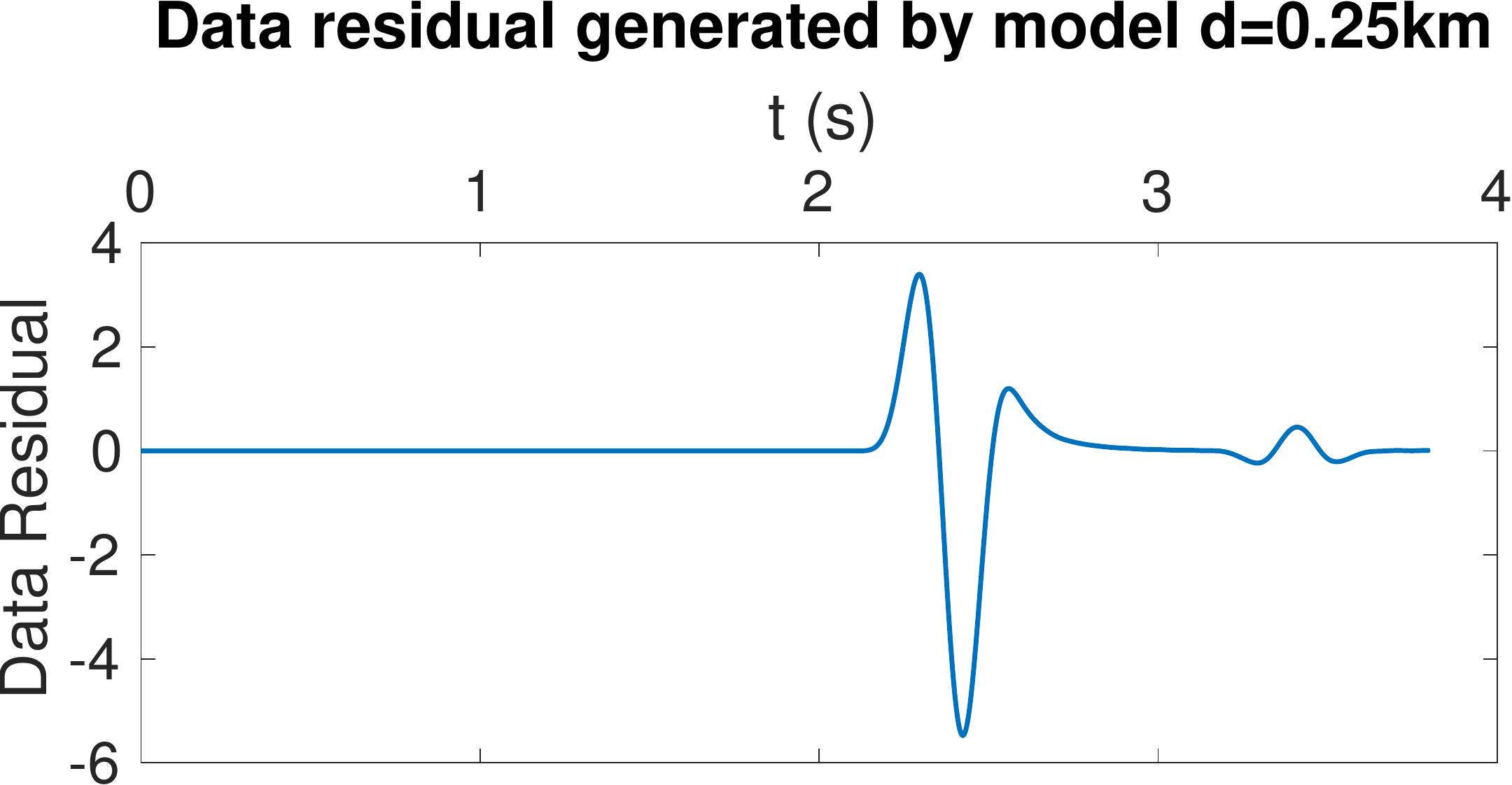}\label{fig:test1_d5_res}} \hspace{1cm}
  \subfloat[Data residual of model $d=2.5$km]{\includegraphics[height=3cm]{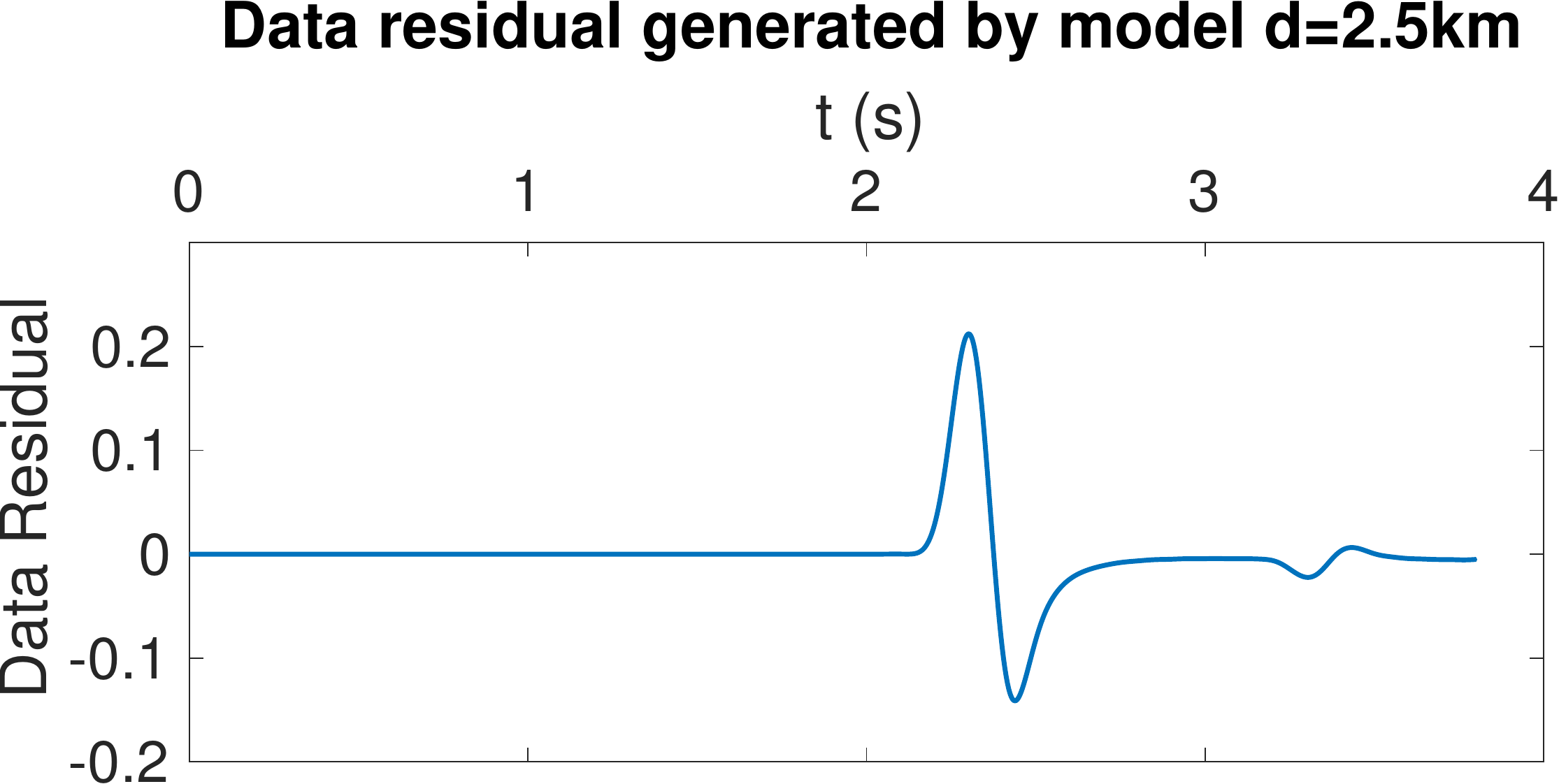}\label{fig:test1_d50_res}} 
  \caption{Three-layer model: (a)~the difference between data in Figure~\ref{fig:test1_d5} and Figure~\ref{fig:test1_dtrue};~(b)~the difference between data in Figure~\ref{fig:test1_d50} and Figure~\ref{fig:test1_dtrue}.}
  \label{fig:test1_diff}
\end{figure}

\subsubsection{Layer Thickness is Embedded in Reflection Amplitude}
Besides the simulated reflection as a result of overshooting, the velocity or the thickness of the third layer is implicitly embedded in the amplitude of the second reflection as well as the head waves. As low-frequency components of the wavefield keep propagating downward into a thick layer, less energy is reflected back to the receivers. However, the majority of the low- and high- frequency energy will be reflected if it were a thin layer instead. In other words, the amplitude of the reflection itself already contains enough information to uncover the thickness of the layer below the reflectors theoretically.

For further demonstration, we construct three velocity models in Figure~\ref{fig:test1_vF5},~\ref{fig:test1_vF50} and \ref{fig:test1_vtrue} with different thickness in the third layer, which are 0.25km, 2.5km and 4km respectively. In the Fourier domain, these three velocity models only differ in low-wavenumber components. 
Keeping other settings the same, wave propagation through the three velocity models renders three wavefields. At the spatial location $x=7.5$km, $z=0.25$km, the recorded waveforms in time are plotted in Figure~\ref{fig:test1_d5},~\ref{fig:test1_d50} and~\ref{fig:test1_dtrue}. As seen in Figure~\ref{fig:test1_same}, the three corresponding waveforms, which is often referred to as data, are quite similar in both amplitude and phase. Nevertheless, the data itself carries the information about the thickness of the layer. The key differences lie in the amplitude of the second reflection. We plot the data residuals in Figure~\ref{fig:test1_diff}. Figure~\ref{fig:test1_d5_res} shows the difference between the data in Figure~\ref{fig:test1_d5} and \ref{fig:test1_dtrue} while Figure~\ref{fig:test1_d50_res} shows the difference between the data in Figure~\ref{fig:test1_d50} and \ref{fig:test1_dtrue}. As the layer thickness increases from $0.25$km to $2.5$km, the $\ell^\infty$ norm of the data residual decreases from $5$ to $0.2$, which again demonstrates that the layer thickness is embedded in the reflection amplitudes.

\subsubsection{The Residual in the Fourier Domain}
In this section, we compare the differences of data residual measured by $L^2$ and $W_2$ in the Fourier domain. Figure~\ref{fig:refl_res_fft} shows the residual spectrum of both $L^2$-FWI and $W_2$-FWI at different iterations. The residual is computed as the difference between synthetic data $f(\mathbf{x_r},t;m)$ generated by the model $m$ at current iteration and the true data $g(\mathbf{x_r},t)$:
\bq \label{eq:residual}
\text{residual} = f(\mathbf{x_r},t;m) - g(\mathbf{x_r},t).
\eq
Figure~\ref{fig:refl_L2_iter1_fft} and Figure~\ref{fig:refl_W2_iter1_fft} show that $L^2$ and $W_2$ inversions have the same initial data residual. As seen in both Figure~\ref{fig:refl_L2_iter51_fft} and Figure~\ref{fig:refl_L2_iter101_fft}, the inversion driven by $L^2$ focuses on reducing the high-frequency residual first and slowly decreases the low-frequency residual later. On the contrary, $W_2$ based inversion reduces the smooth parts of the residual first and then gradually switches to the oscillatory parts by comparing between Figure~\ref{fig:refl_W2_iter51_fft}  and Figure~\ref{fig:refl_W2_iter101_fft}. 

\begin{figure}
\centering
   \subfloat[$L^2$-FWI, iteration 1]{\includegraphics[width=0.33\textwidth]{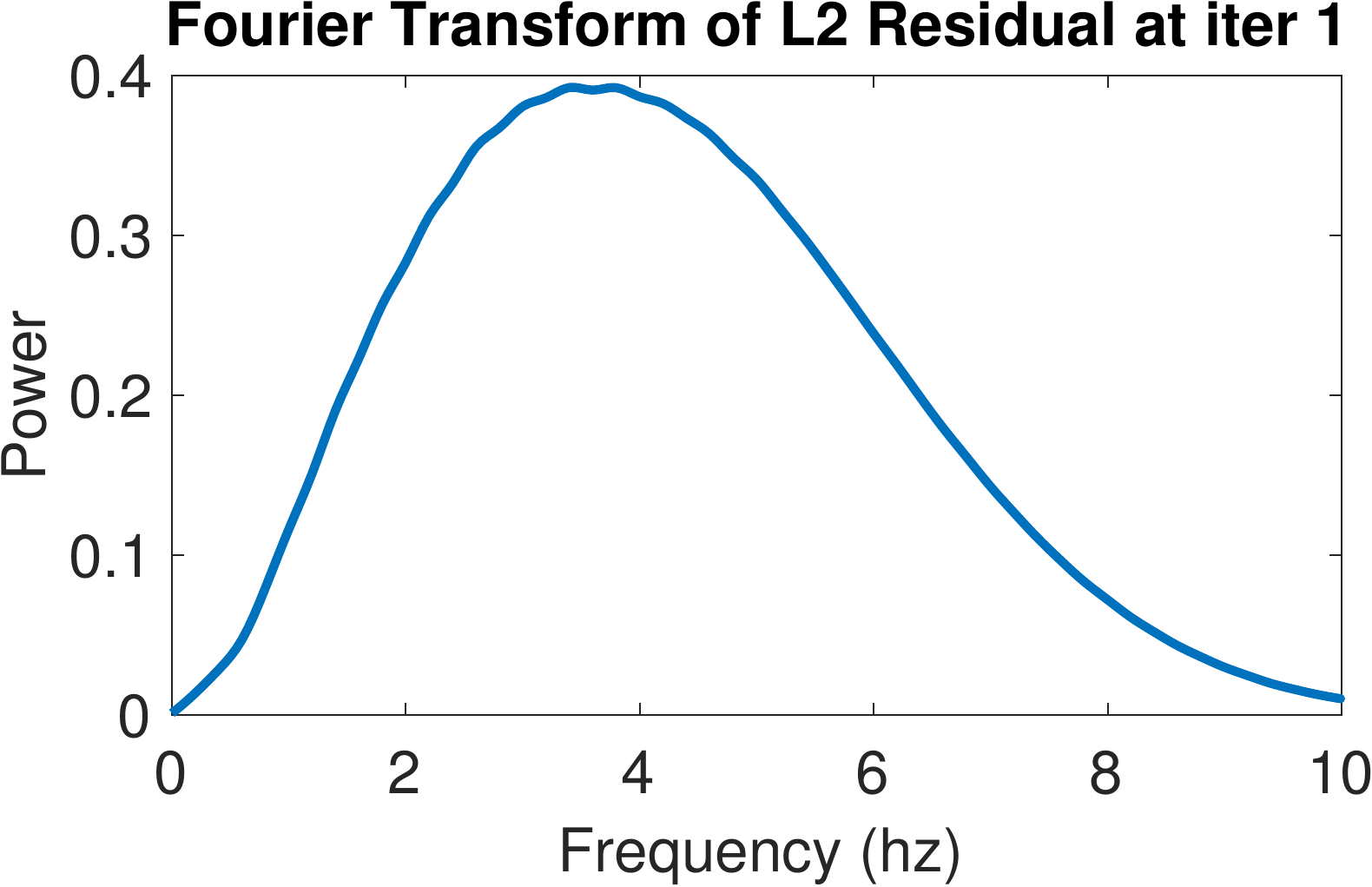}\label{fig:refl_L2_iter1_fft}}
   \subfloat[$L^2$-FWI, iteration 51]{\includegraphics[width=0.33\textwidth]{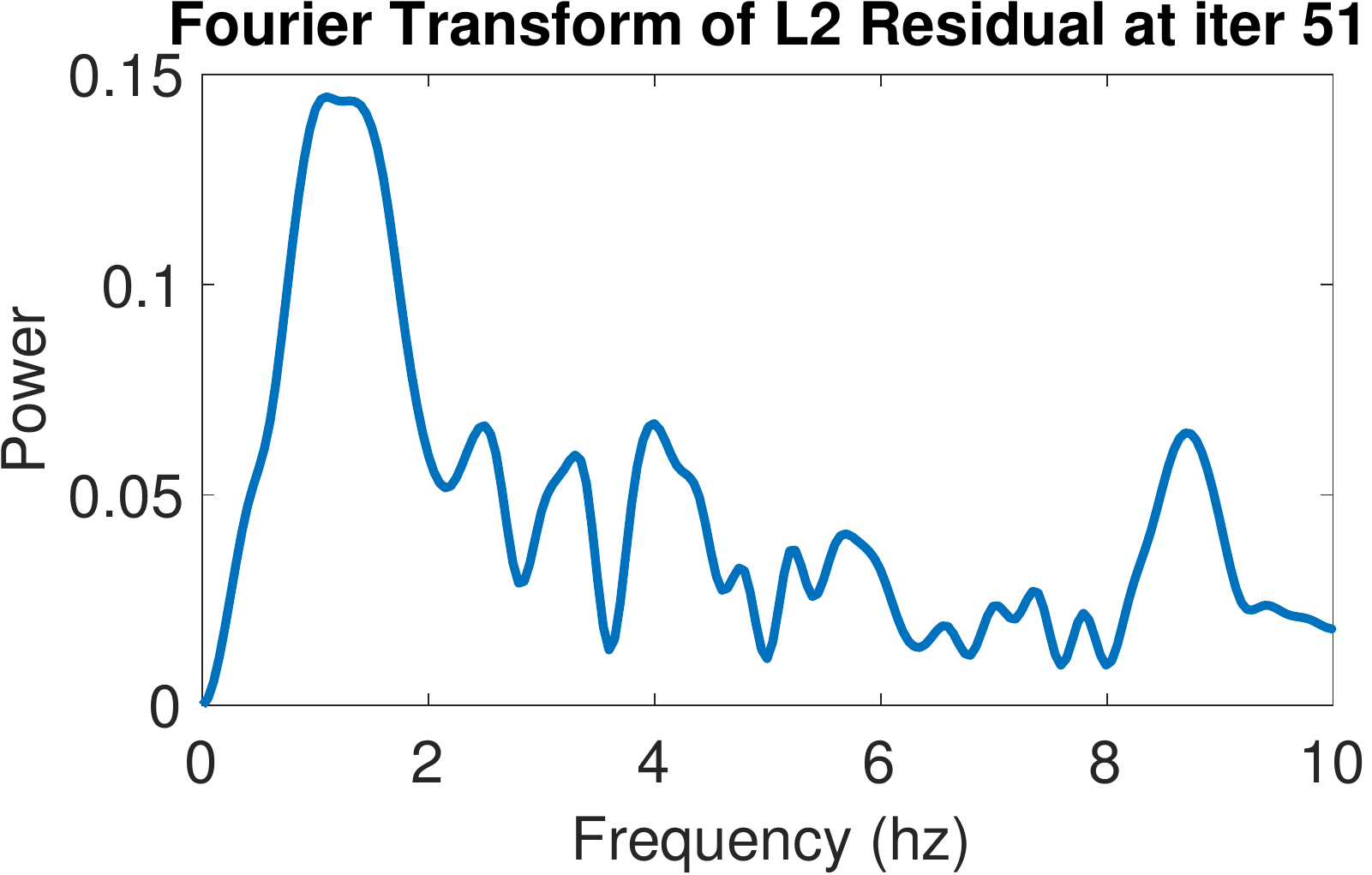}\label{fig:refl_L2_iter51_fft}}
   \subfloat[$L^2$-FWI, iteration 101]{\includegraphics[width=0.33\textwidth]{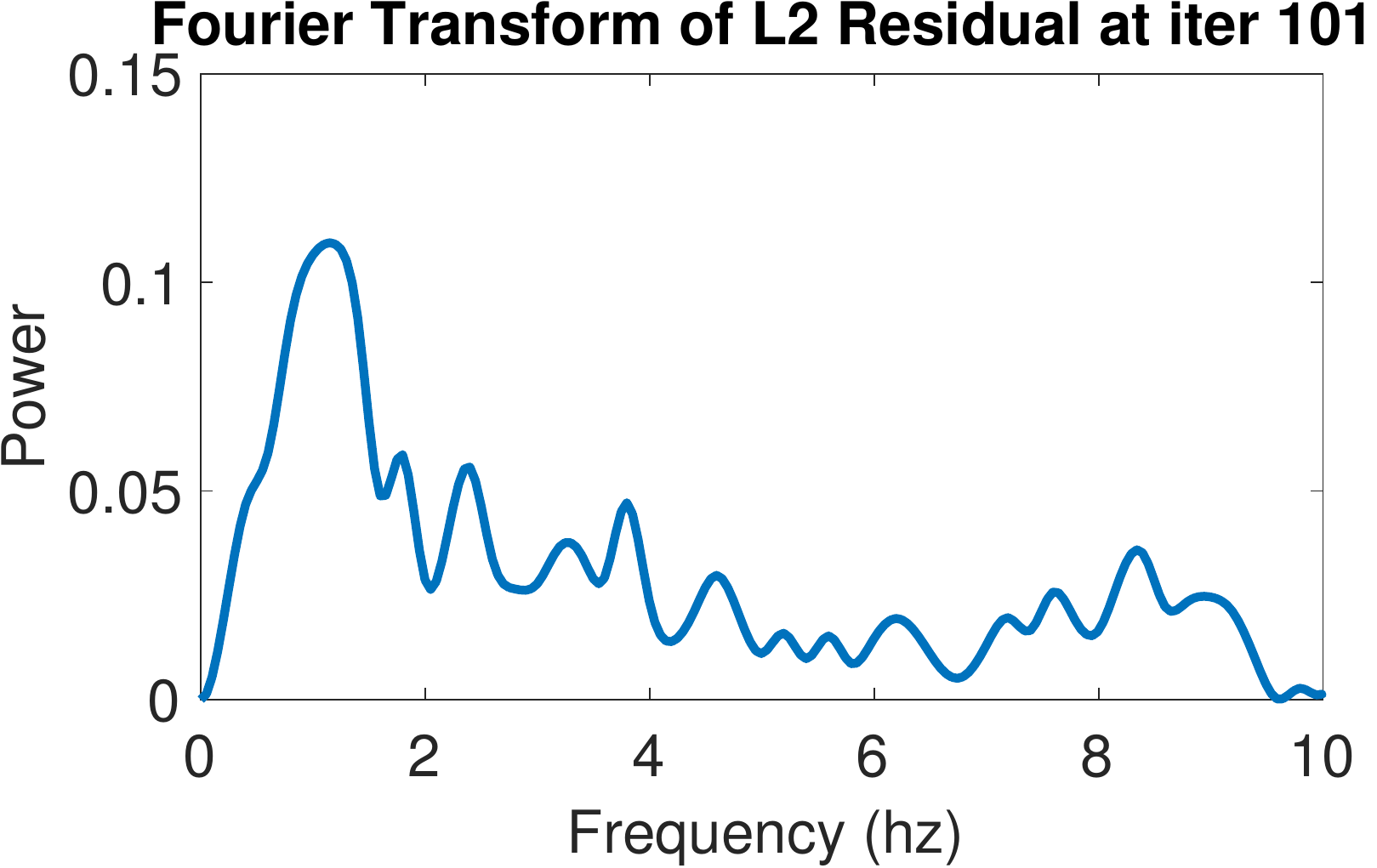}\label{fig:refl_L2_iter101_fft}}\\
      \subfloat[$W_2$-FWI, iteration 1]{\includegraphics[width=0.33\textwidth]{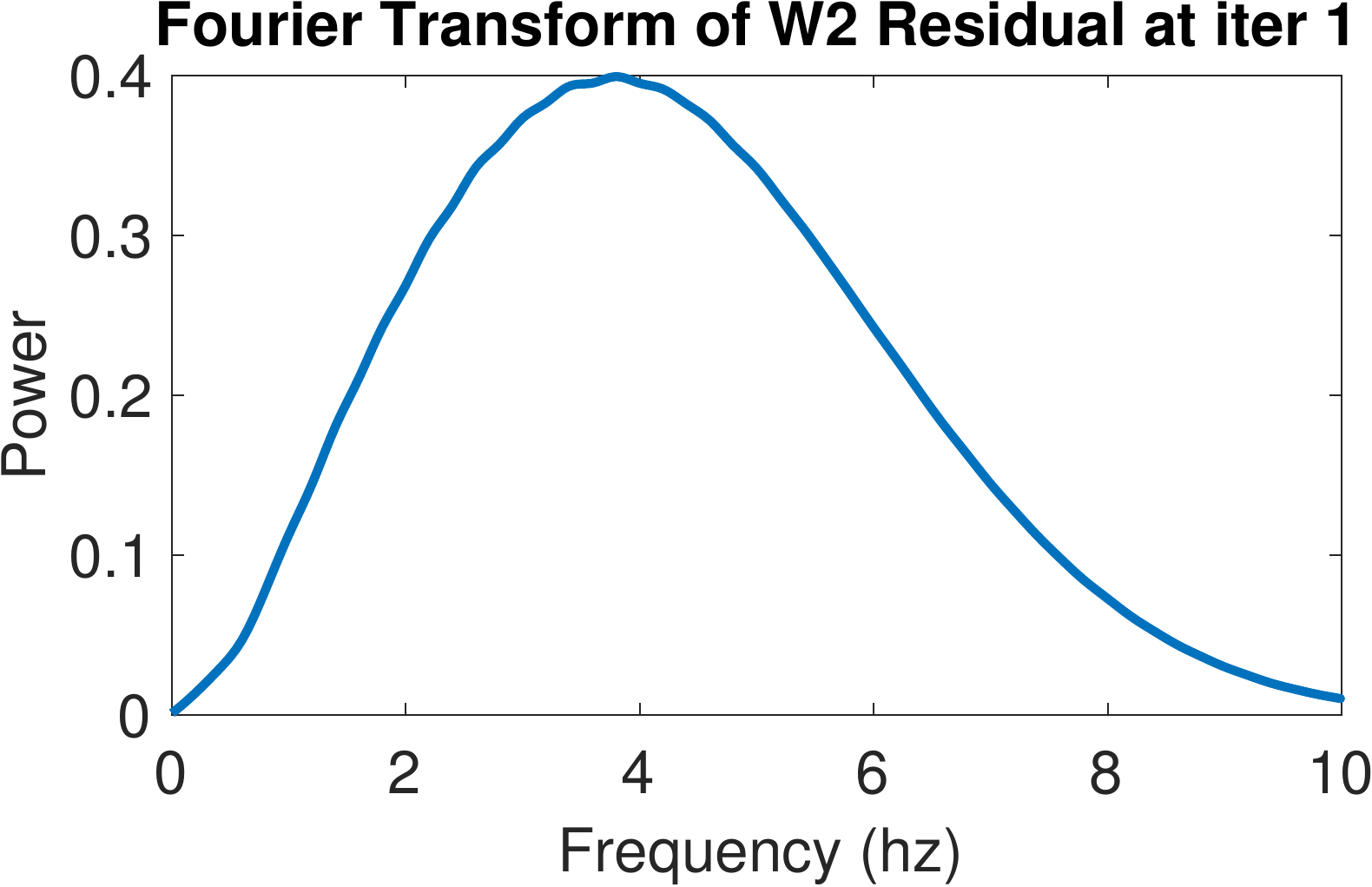}\label{fig:refl_W2_iter1_fft}}
   \subfloat[$W_2$-FWI, iteration 51]{\includegraphics[width=0.33\textwidth]{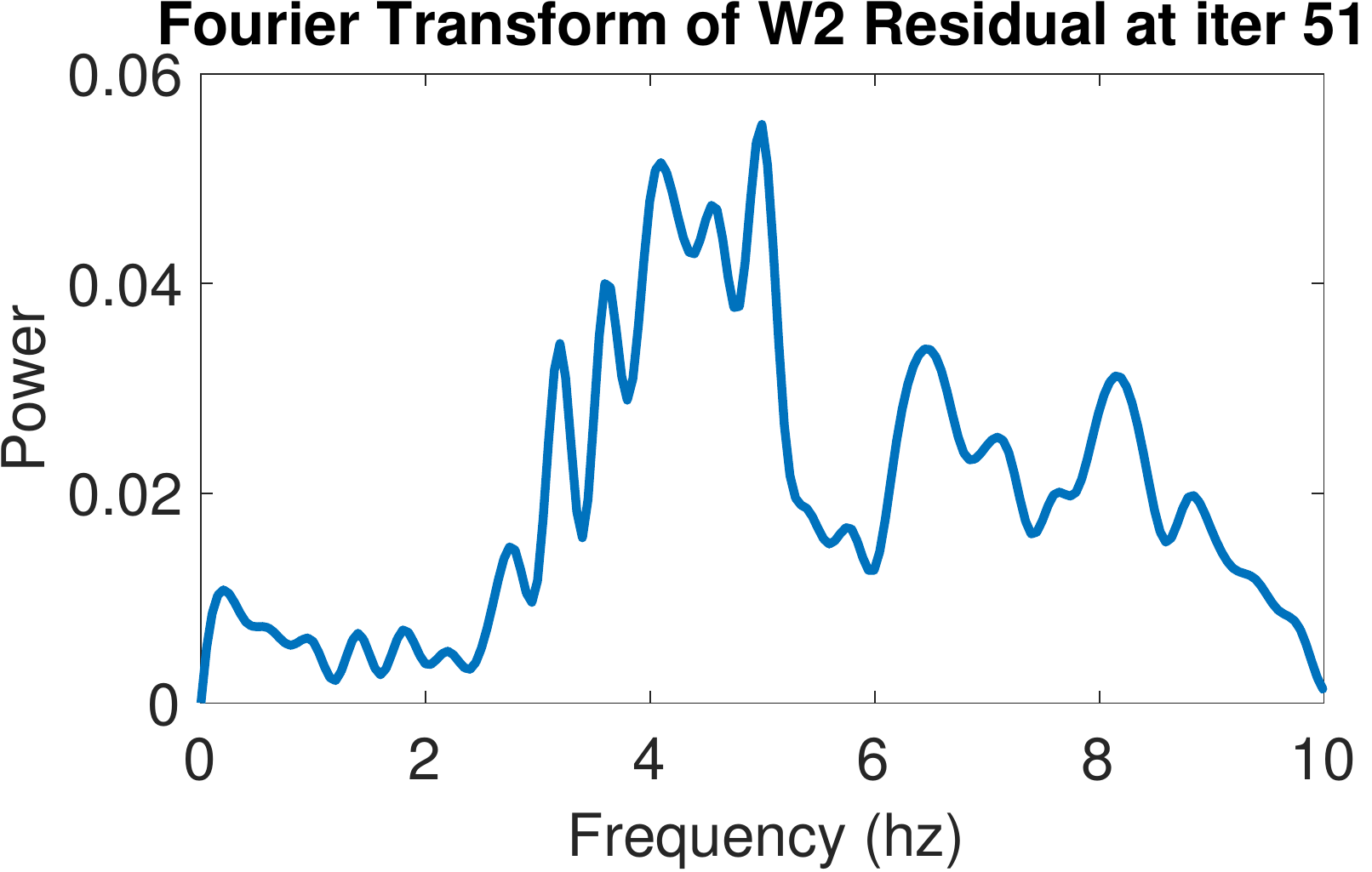}\label{fig:refl_W2_iter51_fft}}
   \subfloat[$W_2$-FWI, iteration 101]{\includegraphics[width=0.33\textwidth]{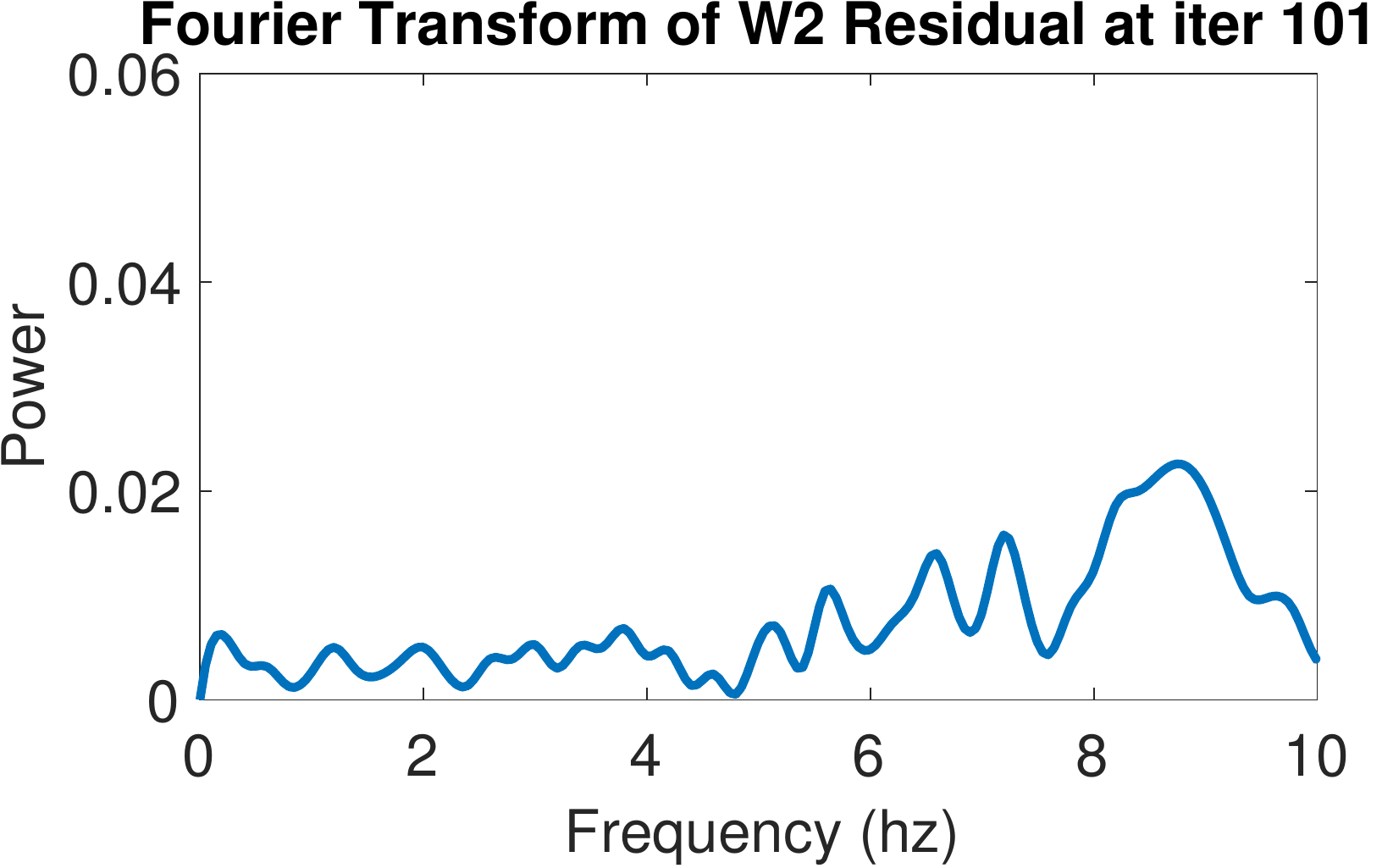}\label{fig:refl_W2_iter101_fft}}
   \caption{Three-layer model: the Fourier transform of the data residual computed as Equation~\eqref{eq:residual} at iteration 1, 51, and 101 of $L^2$ based inversion in (a), (b) and (c) and $W_2$ based inversion in (d), (e) and (f).}~\label{fig:refl_res_fft}
\end{figure}

The Fr\'{e}chet derivatives of the objective function with respect to the simulated data $f$ could help explain different features seen in Figure~\ref{fig:refl_res_fft}. 
When the phases of $f$ and $g$ do not match,  the term $t-G^{-1}(F(t))$ in Equation~\eqref{eq:1D_ADS_C} (the $W_2$ data gradient) represents strong low-wavenumber update. This step is similar to traveltime tomography~\cite{luo1991wave} and other phase-based inversions. Once the phases of $f$ and $g$ match, the term $t-G^{-1}(F(t))$ no longer represents low-frequency parts of the residual. Instead, it represents amplitude differences and strong high-frequency information. As inversion continues, there is a gradual change in the data gradient from smooth modes to oscillatory modes.

Consequently, the model gradient shares a similar trend because of the chain rule. At this later stage, the $W_2$ metric is similar to a local metric that focuses on the point-by-point comparison. The Fourier transform of the data residual in Figure~\ref{fig:refl_W2_iter51_fft} and Figure~\ref{fig:refl_W2_iter101_fft} illustrate this intrinsic property of $W_2$, that the smooth-mode data residual is corrected in earlier iterations (by the smooth gradients), while the oscillatory-mode data residual is only corrected once the smooth-mode error is diminished.
To be more precise, once $f$ is updated to be close enough to $g$ in the optimization, the distance $W_2(f,g)$ is equivalent to the weighted $H^{-1}$ norm, as we have discussed in~\cite{yangletter}.

\subsubsection{Issues Beyond Local Minima}
The data and model convergence curves in~Figure~\ref{fig:errors} raise this question about the goal of fitting the model versus fitting the data. In Figure~\ref{fig:test2_data_conv}, both the normalized $L^2$ norm and $W_2$ distance are reduced from 1 to nearly 0. However, moving from the data convergence curves to the model convergence in Figure~\ref{fig:test2_model_conv2}, $L^2$-FWI is much slower in reducing model error than the performance of $W_2$-FWI. Even if we allow 1800 iterations for $L^2$-FWI, the Frobenius norm of the model error decreases by only 7\% as seen in Figure~\ref{fig:test2_model_conv}. The fact that the $L^2$ misfit decreases from 1 to almost 0 is another demonstration that $L^2$ norm based inversion \textit{does not} suffer from local minima in this particular three-layer example. It is the particular sensitivity to high-frequency residual that prevents $L^2$ from efficiently extracting the model kinematics in the reflection data beside the reflector location.

On the other hand, $W_2$ based inversion is prone to reduce low-frequency residual first and high-frequency residual later. It reduces the model error by more than twice that amount in only 150 iterations.
Since the computational cost per iteration is the same in both cases by computing the $W_2$ distance explicitly in 1D as we have discussed in~\eqref{eq:myOT1D}, inversion using the $W_2$ distance as the objective function lowers the model error more quickly. Reflection differences among Figure~\ref{fig:test1_d5}, Figure~\ref{fig:test1_d50} and Figure~\ref{fig:test1_dtrue} are measured in a nonlinear way by the $W_2$ distance. From Figure~\ref{fig:refl_W2_iter1_fft} to Figure~\ref{fig:refl_W2_iter101_fft}, the residual spectrum changes in a hierarchical order. 

The successful reconstruction of the BP model in Figure~\ref{fig:BP2_true,BP2_v0} shows that properties of the $W_2$ distance illustrated in the simple layered models of Section~\ref{sec:challenge3} can also be seen in a more realistic setting. $W_2$ based inversion captures the essential amplitude information in the reflection in a particular order and successfully reconstructs the thickness of the layer below the reflectors. It updates the missing low-wavenumber components of the model and recovers velocity information below the deepest reflecting interface. The focus on the low-frequency content of the data is the key manifestation of the ``transport'' idea.

\begin{figure}
\centering
   \subfloat[Data Error ]{\includegraphics[width=0.33\textwidth]{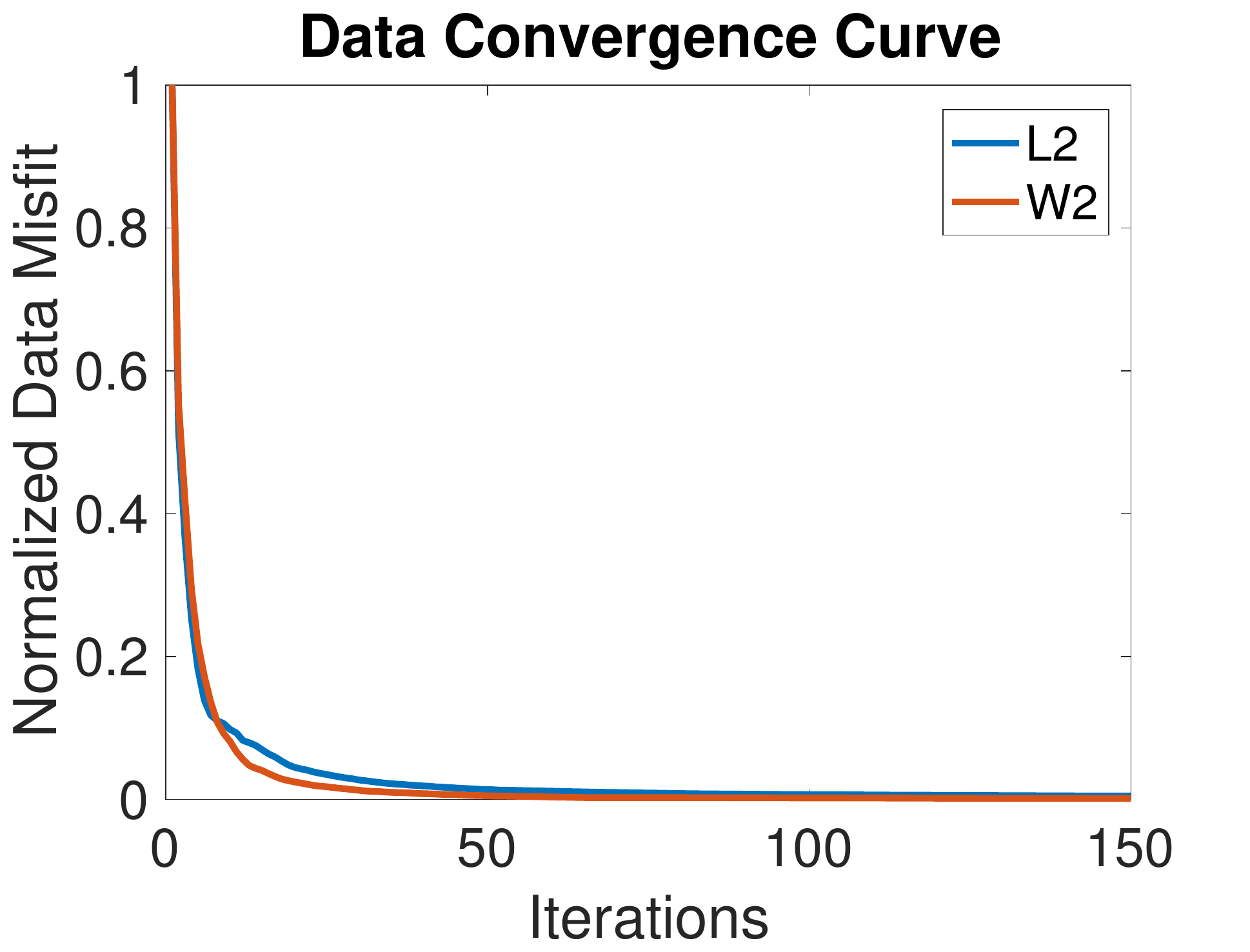}\label{fig:test2_data_conv}}
      \subfloat[Model Error]{\includegraphics[width=0.33\textwidth]{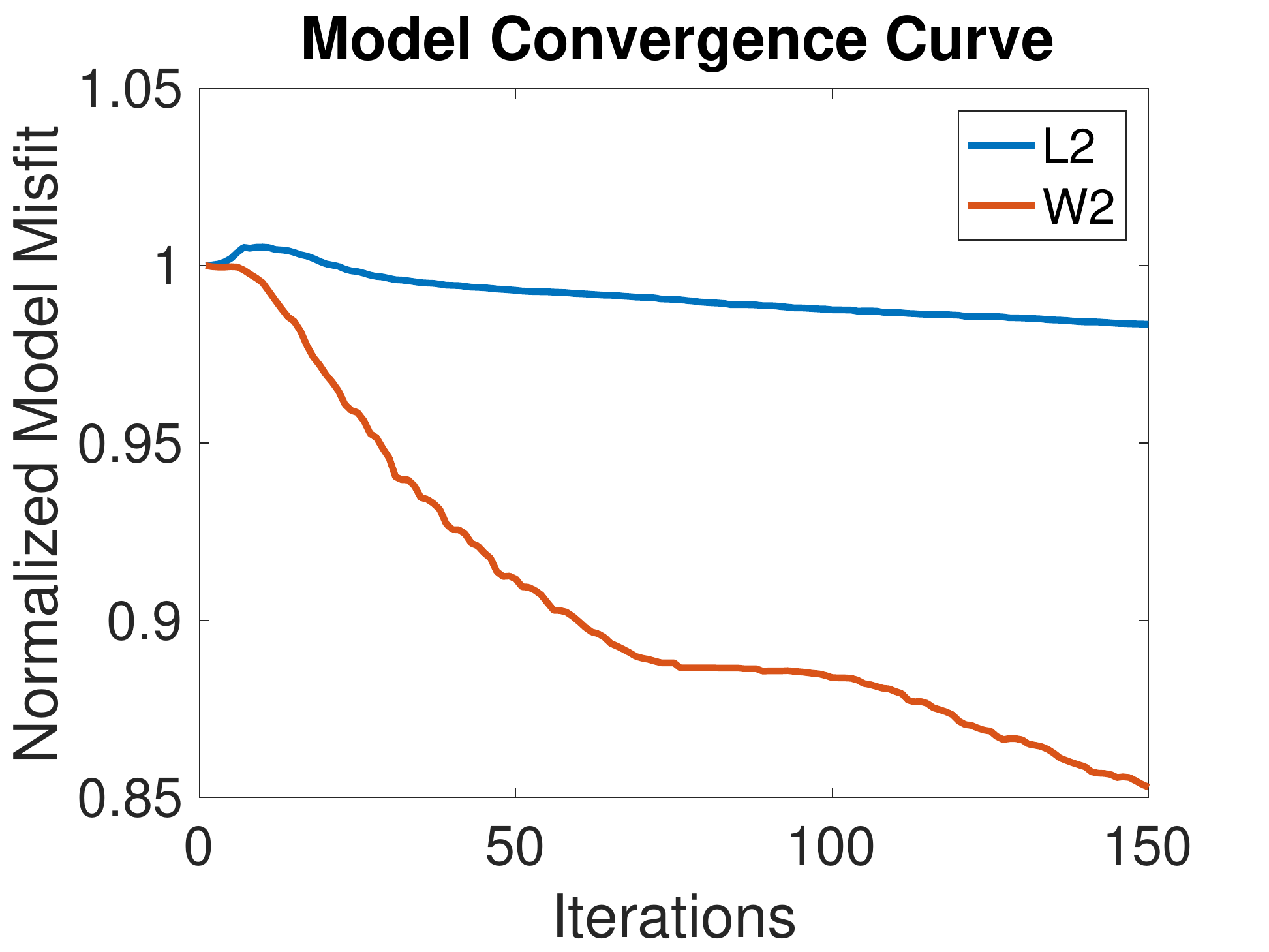}\label{fig:test2_model_conv2}}
   \subfloat[Model Error (more iterations)]{\includegraphics[width=0.33\textwidth]{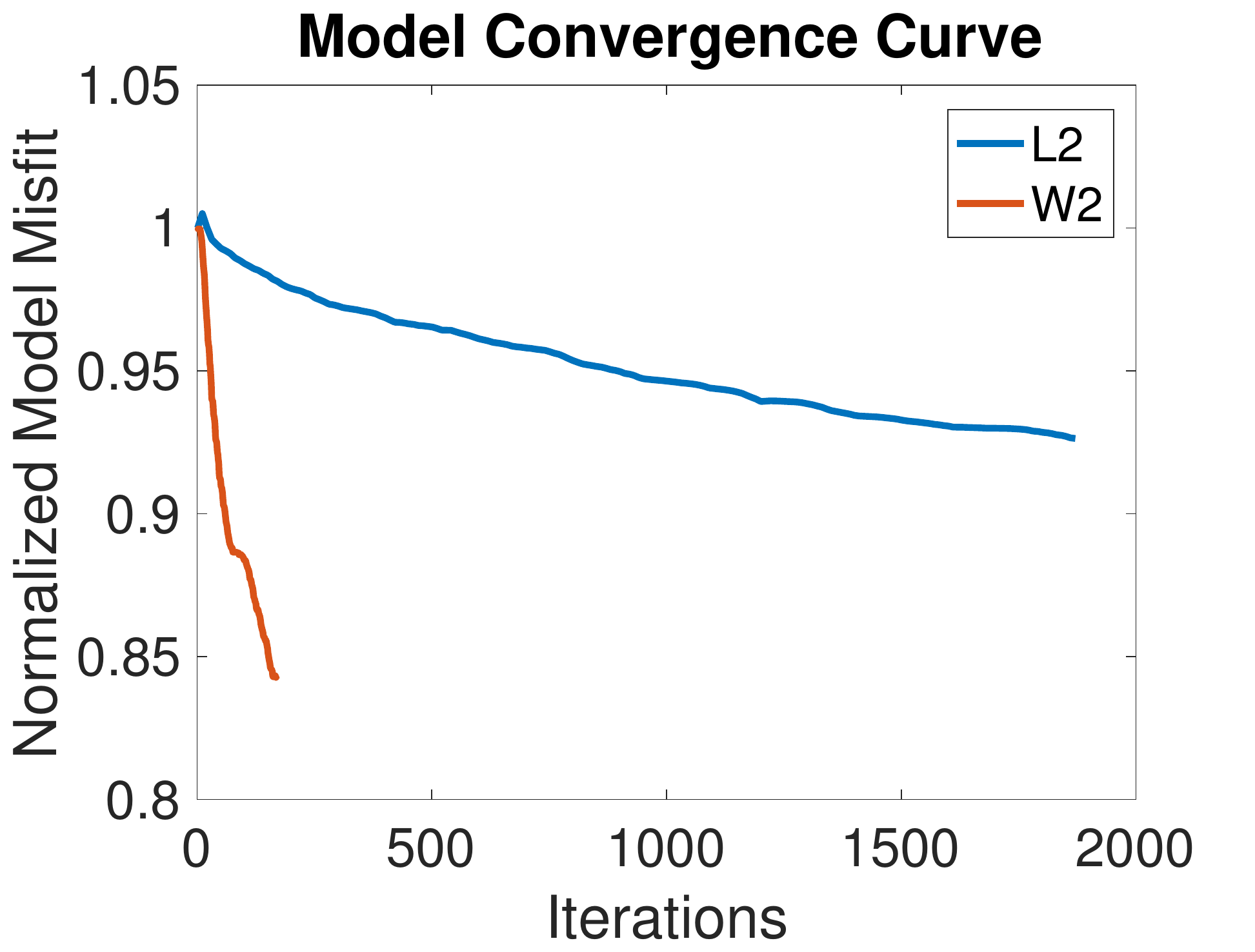}\label{fig:test2_model_conv}}
   \caption{Three-layer model: (a) Normalized $L^2$ and $W_2$ misfit value convergence curves;~(b) and (c)~velocity model parameter convergence curves for $L^2$-FWI and $W_2$-FWI.}~\label{fig:errors}
\end{figure}

\section{Conclusion}


In this paper, we have proved a sharper convexity theorem for the quadratic Wasserstein distance which is more general regarding multiple signal translation and dilation parameters. The improved theorem guarantees that $W_2$ is effective to deal with many issues $L^2$ suffers from, which we also illustrate in several large-scale benchmark inversions. In addition to the long-standing local minima issues, we bring in a new analysis of the sub-layer inversion where the difficulty of the $L^2$ norm is more significant. We have further illustrated that the $W_2$ distance captures the essential low-frequency components of the data residual, which is directly linked to the low-wavenumber structures of the velocity model as seen in the reconstruction of the layered models as well as the realistic 2004 BP salt model. Both mathematical analysis and numerical examples demonstrate that $W_2$ is an advantageous choice of objective function in data-driven inversion. 

Here we have demonstrated that optimal transport offers a new way of measuring signal misfit that considers both phase and intensity differences. Observations from numerical examples are in the same vein as the theorems addressing the convexity of the quadratic Wasserstein distance with respect to common data changes, as well as the robustness with respect to noise. These ideal properties from optimal transport facilitate successful trials of tackling the three long-standing challenges faced by the traditional $L^2$ norm based inversion techniques: the problem of local minima, the ill-posedness with respect to noise and the slow convergence with reflection-based inversion.



\appendix
\numberwithin{equation}{section}


\section{Adjoint-State Method} ~\label{sec:adjoint_state_method}
Large-scale realistic 3D inversion is possible today. The advances in numerical methods and computational power allow for solving the 3D wave equations and compute the Fr\'{e}chet derivative with respect to model parameters, which are needed in the optimization.  In the adjoint-state method~\cite{vogel2002computational,Plessix,tarantola2005inverse}, one only needs to solve two wave equations numerically, the forward propagation and the backward adjoint wavefield propagation. Different misfit functions typically only affect the source term in the adjoint wave equation.

Let us consider the misfit function $J(m)$ for computing the difference between predicted data $f$ and observed data $g$ where $m$ is the model parameter, $F(m)$ is the forward modeling operator, $u(\mathbf{x},t)$ is the wavefield and $s(\mathbf{x},t)$ is the source. The predicted data $f$ is the partial Cauchy boundary data of $u$ which can be written as $f = Ru$ where $R$ is a restriction operator only at the receiver locations. The wave equation~\eqref{eq:FWD} can be denoted as
\bq ~\label{eq:adj_fwd}
F(m) u = s.
\eq
Taking first derivative regarding model $m$ on both sides gives us:
\bq
\frac{\partial F}{\partial m} u + F \frac{\partial u}{\partial m} = 0.
\eq
Therefore,
\bq~\label{eq:adj_grad1}
\frac{\partial f}{\partial m}  = -RF^{-1} \frac{\partial F}{\partial m} u.
\eq

By the chain rule, the gradient of misfit function $J$ with respect to $m$ is
\bq~\label{eq:adj_grad0}
\frac{\partial J}{\partial m} = \left(\frac{\partial f}{\partial m} \right)^T  \frac{\partial J}{\partial f}
\eq
We can derive the following equation by plugging~\eqref{eq:adj_grad1} into ~\eqref{eq:adj_grad0} :
\bq~\label{eq:adj_grad2}
\frac{\partial J}{\partial m}  =-u^T \left(\frac{\partial F}{\partial m}\right)^T F^{-T}R^T  \frac{\partial J}{\partial f}
\eq

Equation~\eqref{eq:adj_grad2} is the adjoint-state method. The term $F^{-T}R^T  \frac{\partial J}{\partial f}$ denotes the backward wavefield $v$ generated by the adjoint wave equation whose source is the data residual $R^T  \frac{\partial J}{\partial f}$. The gradient is similar to the conventional imaging condition~\cite{Claerbout1971}:
\bq~\label{eq:adj_grad3}
\frac{\partial J}{\partial m}  =- \int_0^T \frac{\partial^2 u(\mathbf{x},t)}{\partial t^2} v(\mathbf{x},t)dt,
\eq
where $v$ is the solution to the adjoint wave equation:
\begin{equation} \label{eq:FWI_adj}
     \left\{
     \begin{array}{rl}
     & m\frac{\partial^2 v(\mathbf{x},t)}{\partial t^2}- \Laplace v(\mathbf{x},t)  = R^T\frac{\partial J}{\partial f}\\
    & v(\mathbf{x}, T) = 0                \\
    & v_t(\mathbf{x}, T ) = 0                \\
     \end{array} \right.
\end{equation}
Therefore $F^T$ can be seen as the backward modeling operator. There are many other equivalent ways to formulate the adjoint-state method. One can refer to~\cite{Demanet2016,Plessix} for more details.

In FWI, our aim is to find the model parameter $m^{\star}$ that minimizes the objective function, i.e. \(m^{\star} = \argmin J(m) \). For this PDE-constrained optimization, one can use the Fr\'{e}chet derivative in a gradient-based iterative scheme to update the model $m$, such as steepest descent, conjugate gradient descent (CG), L-BFGS and Gauss-Newton method. One can also derive the second-order adjoint equation for the Hessian matrix and use the full Newton's method in each iteration, but it is not practical regarding memory and current computing power. It is one of the author's current research interests to analyze and approximate the Hessian matrix in optimization~\cite{Virieux2017}.


\begin{thebibliography}{10}

\bibitem{ambrosio2013user}
Luigi Ambrosio and Nicola Gigli.
\newblock A user's guide to optimal transport.
\newblock In {\em Modelling and Optimisation of Flows on Networks}, pages
  1--155. Springer, 2013.

\bibitem{Ballesio2018}
Marco Ballesio, Joakim Beck, Anamika Pandey, Laura Parisi, Erik von Schwerin,
  and Raul Tempone.
\newblock {Multilevel Monte Carlo Acceleration of Seismic Wave Propagation
  under Uncertainty}.
\newblock {\em arXiv preprint arXiv:1810.01710}, 2018.

\bibitem{barles1991convergence}
Guy Barles and Panagiotis~E Souganidis.
\newblock Convergence of approximation schemes for fully nonlinear second order
  equations.
\newblock {\em Asymptotic Analysis}, 4(3):271--283, 1991.

\bibitem{BenBre}
Jean-David Benamou and Yann Brenier.
\newblock A computational fluid mechanics solution to the {M}onge-{K}antorovich
  mass transfer problem.
\newblock {\em Numerische Mathematik}, 84(3):375--393, 2000.

\bibitem{benamou2015iterative}
Jean-David Benamou, Guillaume Carlier, Marco Cuturi, Luca Nenna, and Gabriel
  Peyr{\'e}.
\newblock {Iterative Bregman projections for regularized transportation
  problems}.
\newblock {\em {SIAM Journal on Scientific Computing}}, 37(2):A1111--A1138,
  2015.

\bibitem{benamou2017minimal}
Jean-David Benamou and Vincent Duval.
\newblock {Minimal convex extensions and finite difference discretization of
  the quadratic {M}onge-{K}antorovich problem}.
\newblock {\em arXiv preprint arXiv:1710.05594}, 2017.

\bibitem{benamou2014numerical}
Jean-David Benamou, Brittany~D Froese, and Adam~M Oberman.
\newblock Numerical solution of the optimal transportation problem using the
  {M}onge-{A}mp{\`e}re equation.
\newblock {\em Journal of Computational Physics}, 260:107--126, 2014.

\bibitem{Beydoun1988}
Wafik~B. Beydoun and Albert Tarantola.
\newblock {First {B}orn and {R}ytov approximations: Modeling and inversion
  conditions in a canonical example}.
\newblock {\em The Journal of the Acoustical Society of America},
  83(3):1045--1055, Mar 1988.

\bibitem{billette20052004}
FJ~Billette and Sverre Brandsberg-Dahl.
\newblock The 2004 {BP} velocity benchmark.
\newblock In {\em 67th EAGE Conference \& Exhibition}, 2005.

\bibitem{brenier1991polar}
Yann Brenier.
\newblock {Polar factorization and monotone rearrangement of vector-valued
  functions}.
\newblock {\em {Communications on Pure and Applied Mathematics}},
  44(4):375--417, 1991.

\bibitem{chen2017quadratic}
Jing Chen, Yifan Chen, Hao Wu, and Dinghui Yang.
\newblock The quadratic {W}asserstein metric for earthquake location.
\newblock {\em arXiv preprint arXiv:1710.10447}, 2017.

\bibitem{Claerbout1971}
Jon~F. Claerbout.
\newblock Toward a unified theory of reflector mapping.
\newblock {\em Geophysics}, 36(3):467--481, Jun 1971.

\bibitem{Cuturi2013}
Marco Cuturi.
\newblock Sinkhorn distances: lightspeed computation of optimal transport.
\newblock In {\em Proceedings of the 26th International Conference on Neural
  Information Processing Systems-Volume 2}, pages 2292--2300. Curran Associates
  Inc., 2013.

\bibitem{Dai2013}
Wei Dai and Gerard~T. Schuster.
\newblock Plane-wave least-squares reverse-time migration.
\newblock {\em Geophysics}, 78(4):S165--S177, Jul 2013.

\bibitem{Demanet2016}
Laurent Demanet.
\newblock Waves and imaging class notes -18.325.
\newblock 2016.

\bibitem{EFWass}
Bj{\"o}rn Engquist and Brittany~D Froese.
\newblock Application of the {W}asserstein metric to seismic signals.
\newblock {\em Communications in Mathematical Sciences}, 12(5):979--988, 2014.

\bibitem{engquist2016optimal}
Bj{\"o}rn Engquist, Brittany~D Froese, and Yunan Yang.
\newblock Optimal transport for seismic full waveform inversion.
\newblock {\em Communications in Mathematical Sciences}, 14(8):2309--2330,
  2016.

\bibitem{engquist1977absorbing}
Bj{\"o}rn Engquist and Andrew Majda.
\newblock Absorbing boundary conditions for numerical simulation of waves.
\newblock {\em Proceedings of the National Academy of Sciences},
  74(5):1765--1766, 1977.

\bibitem{Survey1}
Bj{\"o}rn Engquist and Yunan Yang.
\newblock Seismic imaging and optimal transport.
\newblock {\em arXiv preprint arXiv:1808.04801}, 2018.

\bibitem{Survey2}
Bj{\"o}rn Engquist and Yunan Yang.
\newblock Seismic inversion and the data normalization for optimal transport.
\newblock {\em arXiv preprint arXiv:1810.08686}, 2018.

\bibitem{essid2018quadratically}
Montacer Essid and Justin Solomon.
\newblock {Quadratically regularized optimal transport on graphs}.
\newblock {\em {SIAM Journal on Scientific Computing}}, 40(4):A1961--A1986,
  2018.

\bibitem{Evans}
Lawrence~C Evans and Wilfrid Gangbo.
\newblock {\em Differential Equations Methods for the {M}onge-{K}antorovich
  Mass Transfer Problem}, volume 653.
\newblock American Mathematical Society, 1999.

\bibitem{FroeseTransport}
Brittany~D Froese.
\newblock A numerical method for the elliptic {M}onge--{A}mp{\`e}re equation
  with transport boundary conditions.
\newblock {\em SIAM Journal on Scientific Computing}, 34(3):A1432--A1459, 2012.

\bibitem{FOFiltered}
Brittany~D Froese and Adam~M Oberman.
\newblock Convergent filtered schemes for the {M}onge--{A}mp{\`e}re partial
  differential equation.
\newblock {\em SIAM Journal on Numerical Analysis}, 51(1):423--444, 2013.

\bibitem{Gangbo1996}
Wilfrid Gangbo and Robert~J. McCann.
\newblock {The geometry of optimal transportation}.
\newblock {\em Acta Mathematica}, 177(2):113--161, 1996.

\bibitem{kantorovich1960mathematical}
Leonid~Vitalʹevich Kantorovich.
\newblock Mathematical methods of organizing and planning production.
\newblock {\em Management Science}, 6(4):366--422, 1960.

\bibitem{KnottSmith}
M.~Knott and C.~S. Smith.
\newblock On the optimal mapping of distributions.
\newblock {\em Journal of Optimization Theory and Applications}, 43(1):39--49,
  1984.

\bibitem{kolouri2016transport}
Soheil Kolouri, Serim Park, Matthew Thorpe, Dejan Slep{\v{c}}ev, and Gustavo~K
  Rohde.
\newblock Transport-based analysis, modeling, and learning from signal and data
  distributions.
\newblock {\em arXiv preprint arXiv:1609.04767}, 2016.

\bibitem{kuhn1955hungarian}
Harold~W Kuhn.
\newblock The {H}ungarian method for the assignment problem.
\newblock {\em Naval Research Logistics (NRL)}, 2(1-2):83--97, 1955.

\bibitem{lailly1983seismic}
P.~Lailly.
\newblock The seismic inverse problem as a sequence of before stack migrations.
\newblock In {\em Conference on Inverse Scattering: Theory and Application},
  pages 206--220. Society for Industrial and Applied Mathematics, Philadelphia,
  PA, 1983.

\bibitem{Li2016}
Wuchen Li, Stanley Osher, and Wilfrid Gangbo.
\newblock A fast algorithm for earth mover's distance based on optimal
  transport and $l_1$ type regularization.
\newblock {\em arXiv preprint arXiv:1609.07092}, 2016.

\bibitem{liu1989limited}
Dong~C Liu and Jorge Nocedal.
\newblock On the limited memory {BFGS} method for large scale optimization.
\newblock {\em Mathematical Programming}, 45(1):503--528, 1989.

\bibitem{liu2018multilevel}
Jialin Liu, Wotao Yin, Wuchen Li, and Yat~Tin Chow.
\newblock {Multilevel Optimal Transport: a Fast Approximation of
  {W}asserstein-1 distances}.
\newblock {\em arXiv preprint arXiv:1810.00118}, 2018.

\bibitem{luo1991wave}
Yi~Luo and Gerard~T Schuster.
\newblock Wave-equation traveltime inversion.
\newblock {\em Geophysics}, 56(5):645--653, 1991.

\bibitem{mccann1995existence}
R.~J. McCann.
\newblock Existence and uniqueness of monotone measure-preserving maps.
\newblock {\em Duke Mathematical Journal}, 80(2):309--324, 1995.

\bibitem{metivier2018graph}
L~M{\'e}tivier, A~Allain, R~Brossier, Q~M{\'e}rigot, E~Oudet, and J~Virieux.
\newblock {A graph-space approach to optimal transport for full waveform
  inversion}.
\newblock In {\em SEG Technical Program Expanded Abstracts 2018}, pages
  1158--1162. Society of Exploration Geophysicists, 2018.

\bibitem{W1_2D}
L~M{\'e}tivier, R~Brossier, Q~M{\'e}rigot, E~Oudet, and J~Virieux.
\newblock Measuring the misfit between seismograms using an optimal transport
  distance: application to full waveform inversion.
\newblock {\em Geophysical Journal International}, 205(1):345--377, 2016.

\bibitem{W1_3D}
L~M\'etivier, R~Brossier, Q~M{\'e}rigot, E~Oudet, and J~Virieux.
\newblock An optimal transport approach for seismic tomography: application to
  {3D} full waveform inversion.
\newblock {\em Inverse Problems}, 32(11):115008, 2016.

\bibitem{moczo2007finite}
Peter Moczo, Johan~OA Robertsson, and Leo Eisner.
\newblock The finite-difference time-domain method for modeling of seismic wave
  propagation.
\newblock {\em Advances in Geophysics}, 48:421--516, 2007.

\bibitem{Monge}
Gaspard Monge.
\newblock M\'{e}moire sur la th\'{e}orie des d\'{e}blais et de remblais.
  histoire de l'acad\'{e}mie royale des sciences de paris.
\newblock {\em avec les M\'{e}moires de Math\'{e}matique et de Physique pour la
  même ann\'{e}e}, pages 666--704, 1781.

\bibitem{Mora1988}
Peter Mora.
\newblock {Elastic wave‐field inversion of reflection and transmission data}.
\newblock {\em Geophysics}, 53(6):750--759, Jun 1988.

\bibitem{Motamed2018}
Mohammad Motamed and Daniel Appel{\"o}.
\newblock {Wasserstein metric-driven Bayesian inversion with application to
  wave propagation problems}.
\newblock {\em arXiv preprint arXiv:1807.09682}, 2018.

\bibitem{Oberman2015}
Adam~M Oberman and Yuanlong Ruan.
\newblock An efficient linear programming method for optimal transportation.
\newblock {\em arXiv preprint arXiv:1509.03668}, 2015.

\bibitem{Plessix}
R.-E. Plessix.
\newblock A review of the adjoint-state method for computing the gradient of a
  functional with geophysical applications.
\newblock {\em Geophysical Journal International}, 167(2):495--503, 2006.

\bibitem{poncet2018fwi}
R~Poncet, J~Messud, M~Bader, G~Lambar{\'e}, G~Viguier, and C~Hidalgo.
\newblock {FWI} with optimal transport: a 3{D} implementation and an
  application on a field dataset.
\newblock In {\em 80th EAGE Conference and Exhibition 2018}, 2018.

\bibitem{pratt1990inverse1}
R~Gerhard Pratt and MH~Worthington.
\newblock {Inverse theory applied to multi-source cross-hole tomography. Part
  1: Acoustic wave-equation method}.
\newblock {\em Geophysical Prospecting}, 38(3):287--310, 1990.

\bibitem{Puthawala2018}
Michael~A Puthawala, Cory~D Hauck, and Stanley~J Osher.
\newblock Diagnosing forward operator error using optimal transport.
\newblock {\em arXiv preprint arXiv:1810.12993}, 2018.

\bibitem{qiu2017full}
Lingyun Qiu, Jaime Ramos-Mart{\'\i}nez, Alejandro Valenciano, Yunan Yang, and
  Bj{\"o}rn Engquist.
\newblock Full-waveform inversion with an exponentially encoded
  optimal-transport norm.
\newblock In {\em SEG Technical Program Expanded Abstracts 2017}, pages
  1286--1290. Society of Exploration Geophysicists, 2017.

\bibitem{Julien2011}
Julien Rabin, Gabriel Peyr{\'e}, Julie Delon, and Marc Bernot.
\newblock Wasserstein barycenter and its application to texture mixing.
\newblock In {\em International Conference on Scale Space and Variational
  Methods in Computer Vision}, pages 435--446. Springer, 2011.

\bibitem{rachev1998mass}
Svetlozar~T Rachev and Ludger R{\"u}schendorf.
\newblock {\em Mass Transportation Problems: Volume I: Theory}, volume~1.
\newblock Springer Science \& Business Media, 1998.

\bibitem{Ramos2018}
J~Ramos-Mart{\'\i}nez, L~Qiu, J~Kirkeb{\o}, AA~Valenciano, and Y~Yang.
\newblock Long-wavelength fwi updates beyond cycle skipping.
\newblock In {\em SEG Technical Program Expanded Abstracts 2018}, pages
  1168--1172. Society of Exploration Geophysicists, 2018.

\bibitem{ryu2018unbalanced}
Ernest~K Ryu, Wuchen Li, Penghang Yin, and Stanley Osher.
\newblock {Unbalanced and Partial $L_1$ Monge--Kantorovich Problem: A Scalable
  Parallel First-Order Method}.
\newblock {\em Journal of Scientific Computing}, 75(3):1596--1613, 2018.

\bibitem{Santambrogio}
Filippo Santambrogio.
\newblock {\em Optimal Transport for Applied Mathematicians: Calculus of
  Variations, PDEs, and Modeling}, volume~87.
\newblock Birkh{\"a}user, 2015.

\bibitem{schmitzer2016sparse}
Bernhard Schmitzer.
\newblock A sparse multiscale algorithm for dense optimal transport.
\newblock {\em Journal of Mathematical Imaging and Vision}, 56(2):238--259,
  2016.

\bibitem{tarantola2005inverse}
Albert Tarantola.
\newblock {\em Inverse Problem Theory: Methods for Data Fitting and Model
  Parameter Estimation}.
\newblock {SIAM}, 2005.

\bibitem{tarantola1982generalized}
Albert Tarantola and Bernard Valette.
\newblock Generalized nonlinear inverse problems solved using the least squares
  criterion.
\newblock {\em Reviews of Geophysics}, 20(2):219--232, 1982.

\bibitem{versteeg1994marmousi}
Roelof Versteeg.
\newblock {The Marmousi experience: Velocity model determination on a synthetic
  complex data set}.
\newblock {\em The Leading Edge}, 13(9):927--936, 1994.

\bibitem{Villani}
C.~Villani.
\newblock {\em Topics in Optimal Transportation}, volume~58 of {\em Graduate
  Studies in Mathematics}.
\newblock American Mathematical Society, Providence, RI, 2003.

\bibitem{villani2008optimal}
C.~Villani.
\newblock {\em Optimal Transport: Old and New}, volume 338.
\newblock Springer Science \& Business Media, 2008.

\bibitem{Virieux2017}
J.~Virieux, A.~Asnaashari, R.~Brossier, L.~M{\'{e}}tivier, A.~Ribodetti, and
  W.~Zhou.
\newblock {6. An introduction to full waveform inversion}.
\newblock In {\em Encyclopedia of Exploration Geophysics}, pages R1--1--R1--40.
  Society of Exploration Geophysicists, Jan 2014.

\bibitem{vogel2002computational}
Curtis~R Vogel.
\newblock {\em {Computational Methods for Inverse Problems}}, volume~23.
\newblock {SIAM}, 2002.

\bibitem{yang2016review}
Pengliang Yang, Romain Brossier, Ludovic M{\'e}tivier, and Jean Virieux.
\newblock A review on the systematic formulation of {3-D} multiparameter full
  waveform inversion in viscoelastic medium.
\newblock {\em Geophysical Journal International}, 207(1):129--149, 2016.

\bibitem{yangletter}
Yunan Yang and Bj{\"o}rn Engquist.
\newblock Analysis of optimal transport and related misfit functions in
  full-waveform inversion.
\newblock {\em Geophysics}, 83(1):A7--A12, 2018.

\bibitem{yang2017application}
Yunan Yang, Bj{\"o}rn Engquist, Junzhe Sun, and Brittany~D Froese.
\newblock Application of optimal transport and the quadratic {W}asserstein
  metric to full-waveform inversion.
\newblock {\em Geophysics}, 83(1):1--103, 2017.

\end{thebibliography}
\end{document}